\documentclass[a4paper,fleqn]{cas-sc}

\usepackage[numbers,square,sort&compress]{natbib}
\usepackage{tikz}
\usepackage{float}
\usepackage{amsmath,amssymb,bm}
\usetikzlibrary{arrows.meta,positioning,calc,fit,backgrounds,shapes.geometric}
\usepackage{amsthm}

\newtheorem{theorem}{Theorem}
\newtheorem{lemma}{Lemma}
\newtheorem{proposition}{Proposition}
\newtheorem{corollary}{Corollary}

\theoremstyle{definition}

\theoremstyle{remark}
\newtheorem{remark}{Remark}
\newcommand{\stitle}[1]{{\bfseries\footnotesize #1}}
\definecolor{bluebg}{RGB}{222,233,248}
\definecolor{bluehead}{RGB}{31,84,150} 
\definecolor{greystep}{RGB}{226,226,226}
\definecolor{lilac}{RGB}{224,221,242}
\definecolor{ctbg}{RGB}{214,232,224}
\definecolor{cthead}{RGB}{190,219,205}
\definecolor{greybg}{RGB}{238,238,238}
\definecolor{linecol}{RGB}{90,90,100}
\def\tsc#1{\csdef{#1}{\textsc{\lowercase{#1}}\xspace}}
\tsc{WGM}
\tsc{QE}

\usepackage{subcaption}
\usepackage{float}

\usepackage{placeins}

\begin{document}
\let\WriteBookmarks\relax
\def\floatpagepagefraction{1}
\def\textpagefraction{.001}

\shorttitle{}    

\shortauthors{}  

\title [mode = title]{A Structure- and Pressure-Positivity-Preserving Semi-implicit IMEX Finite Volume Scheme for Ideal MHD at All Acoustic Mach and Alfvén Mach Numbers with Generic Equation of State}  


\tnotetext[1]{} 

%

\author[1]{Zefeng Chen}
\ead{zc378@cam.ac.uk}

\author[1]{Riccardo Demattè}
\ead{rd609@cam.ac.uk}

\author[2,3]{Walter Boscheri}
\cormark[1]
\ead{walter.boscheri@univ-smb.fr}

\author[1]{Stephen Millmore}
\ead{stm31@cam.ac.uk}

\affiliation[1]{organization={Laboratory for Scientific Computing, Cavendish Laboratory, Department of Physics},
            addressline={University of Cambridge, J. J. Thomson Avenue},
            city={Cambridge},
            postcode={CB3 0HE},
            country={United Kingdom}}

\affiliation[2]{organization={Laboratoire de Mathématiques UMR 5127 CNRS},
            addressline={Université Savoie Mont Blanc},
            city={Le Bourget du Lac},
            postcode={73376},
            country={France}}

\affiliation[3]{organization={Department of Mathematics and Computer Science, University of Ferrara},
            city={Ferrara},
            postcode={44121},
            state={Ferrara},
            country={Italy}}

\cortext[1]{Corresponding author}


\begin{abstract}
In this work we present a conservative, structure-preserving, pressure-positivity-preserving, cell-centred semi-implicit finite-volume scheme for the ideal magnetohydrodynamics (MHD) equations, which remains uniformly applicable across all acoustic Mach and Alfvén Mach number regimes and accommodates general nonlinear equations of state. The scheme is based on a flux-vector splitting of the MHD system into three sub-systems, partitioned according to the characteristic wave scales of the model: an advective sub-system associated with the hydrodynamic transport, a magnetic sub-system describing the velocity–magnetic field coupling, and a pressure sub-system describing the pressure–velocity coupling. The nonlinear advective terms are discretised explicitly as part of a standard conservative finite-volume update, while the two sub-systems responsible for the Alfvén, magnetosonic and acoustic waves are treated implicitly, so that the resulting algorithm is constrained only by a mild CFL condition depending solely on the fluid velocity. This renders the scheme highly suitable for both gas-pressure and magnetic-pressure dominated regimes, as well as for the incompressible limit of the MHD equations. The implicit discretisation gives rise to a pressure equation that reduces, in the low acoustic Mach limit, to an elliptic equation applying uniformly to ideal-gas and general thermodynamic closures. Within this same framework, pressure-positivity-preserving robustness is obtained through a local, conservation-preserving modification of the pressure–internal-energy relation within the implicit solver, which keeps the discrete pressure positive without a posteriori clipping, while preserving the conservation property of the scheme. The divergence-free constraint of the magnetic field is enforced exactly at the discrete level through a constrained transport formulation. Second-order accuracy is achieved by means of an implicit--explicit (IMEX)
Runge--Kutta time integration, with a TVD reconstruction of the explicit fluxes and
a central discretisation of the implicit terms. The scheme is validated thoroughly against a broad set of benchmark problems spanning high- and low-acoustic-Mach-number regimes, strongly magnetized configurations, and standard stringent MHD shock problems in one and two dimensions, where it is shown to provide accurate and stable solutions while retaining excellent shock-capturing capabilities.
\end{abstract}


\begin{highlights}
\item A conservative, pressure positive-preserving, cell-centred, structure-preserving semi-implicit IMEX finite-volume scheme for ideal MHD.
\item A mild, velocity-based CFL condition supported by extensive numerical evidence.
\item A pressure-based formulation supporting general nonlinear equations of state via a nested Newton solver.
\item Pressure-positivity-preserving robustness enforced within the pressure system without compromising conservation.
\end{highlights}


\maketitle

\section{Introduction}

The equations of magnetohydrodynamics (MHD) describe an electrically conducting fluid
coupled to a magnetic field and form a first-order nonlinear system of non-strictly
hyperbolic conservation laws. They arise in astrophysical plasma dynamics
\cite{GoedbloedPoedts2019,Priest2014} and in inertial and magnetic confinement fusion
\cite{Jardin2010,AtzeniMeyer2004}, where a magnetically dominated plasma may evolve
rapidly from a near-equilibrium configuration into a fully time-dependent compressible
regime: this may happen during disruption events in confined plasmas \cite{disruption}, or during the
compression of matter in magneto-inertial fusion concepts such as MagLIF
\cite{Slutz2010,Gomez2014,Sinars2020}. The model is intrinsically multiscale, as
quantified by the acoustic Mach number $M_c$ and the Alfv\'en number $M_b$: as $M_c
\to 0$ the system approaches a weakly compressible regime governed by an elliptic
pressure balance, for $M_c = \mathcal{O}(1)$ acoustic waves and shocks coexist with
the flow, and an analogous singular limit arises as $M_b \to 0$, typical of the
magnetic-pressure dominated plasmas of tokamak scenarios. These regimes may coexist
within a single simulation.

Numerical methods for these flows have traditionally been optimised for one of the two
regimes: pressure-based projection methods for the nearly incompressible low acoustic Mach limit
\cite{Chorin1968}, and density-based Godunov-type finite volume schemes
\cite{BrioWu1988,MiyoshiKusano2005} for the compressible, shock-dominated one. The
latter faces two difficulties in the stiff limits: the time step is dictated by the
fast magnetosonic speed, and the artificial dissipation of upwind fluxes scales with
that same speed rather than with the fluid velocity, so that the asymptotic scaling $p
= p_0 + \mathcal{O}(M_c^2)$ is lost \cite{GuillardViozat2017,Dellacherie2010}. Since
the fast magnetosonic speed degenerates to the Alfv\'en speed in the
magnetic-pressure dominated regime, both reappear as $M_b \to 0$. Rescaling the flux
dissipation \cite{barsukow2021truly,barsukow2023all} restores the correct scaling but does not improve the time step size, which then calls for a fully implicit integration
\cite{leidi2022finite,viallet2011towards} and hence for large nonsymmetric nonlinear
systems whose convergence is difficult to control.

By treating only the stiff terms implicitly, while keeping nonlinear transport
explicit, avoids both difficulties, and is naturally realised within the class of
implicit--explicit (IMEX) Runge--Kutta integrators
\cite{RKIMEX,RKIMEX2}. Pressure-based semi-implicit schemes trace their origins to the seminal work of
Harlow and Welch~\cite{HarlowWelch1965}, with important subsequent developments by
Casulli and Greenspan~\cite{CasulliGreenspan1984}. Building on this foundation, Dumbser
and Casulli~\cite{DumbserCasulli2016} later introduced a conservative, shock-capturing
formulation by coupling the Toro--V\'azquez flux splitting~\cite{ToroVazquez2012} with an
implicitly treated pressure sub-system. The
resulting system has a structure that recurs throughout this family and is central to
what follows: it is only mildly nonlinear, the nonlinearity being confined to a
diagonal term carrying the equation of state, while the remaining coupling is linear, positive semi-definite. Such systems are solved by the nested Newton
technique of Casulli and Zanolli \cite{CasulliZanolli2012}, which admits general
equations of state. High-order all acoustic Mach number extensions followed
\cite{BoscheriPareschi2021,TavelliDumbser2017}.

The extension to MHD has proceeded by successively removing the stiff wave families
from the explicit time step. Dumbser et al. \cite{Dumbser2019} kept the pressure terms
implicit while discretising the convective terms and the magnetic field explicitly, so
that the time step remains governed by the Alfv\'en speed: the scheme is efficient in
gas-pressure dominated scenarios and loses its advantage precisely where the Alfv\'en
speed approaches the fast magnetosonic one. Removing that restriction requires the
magnetic field to be treated implicitly as well, as done by Fambri \cite{Fambri2021}
through a structure-preserving three-way splitting on staggered dual meshes.
 On collocated grids, Dematt\'e et al.~\cite{Dematte2024} coupled implicit magnetic and
pressure sub-systems through a nested Picard iteration, building on the flux vector
splitting of Balsara, Montecinos and Toro~\cite{BalsaraMontecinosToro2016}, which
retains a well-behaved hyperbolic structure in all regimes; Boscheri and
Thomann~\cite{BoscheriThomann2024} instead formulated the implicit step for the total
energy. in both cases the time step is constrained only by the material
velocity. In this work, we attempt to develop a scheme which guarantees pressure positivity, whilst remaining general enough to use a nonlinear equation of state

The present paper constructs a conservative, cell-centred semi-implicit finite volume
scheme for ideal MHD that is uniformly applicable at all acoustic Mach and Alfv\'en Mach numbers,
admits a general nonlinear equation of state, and guarantees pressure positivity
within the implicit solve. It adopts the three-way splitting of Fambri
\cite{Fambri2021} with the semi-implicit discretisation proposed by Boscheri and Thomann \cite{BoscheriThomann2024}: the advective sub-system, whose nonzero eigenvalues equal the fluid velocity, is
discretised explicitly with a Rusanov flux. The energy is split analogously to the
Toro--V\'azquez splitting for the Euler equations, such that only the kinetic-energy
contribution appears in the explicitly treated flux, while the magnetic and pressure
sub-systems are integrated implicitly. In place of the energy-based formulation of
\cite{BoscheriThomann2024}, the implicit step is discretised directly in terms of the
pressure, the magnetic field being solved first and entering the pressure sub-system
once updated, so that the two implicit solves form a single one-directional sequence
without an outer coupling iteration. The pressure sub-system is solved twice per step
following \cite{FVVEM}, which conserves the total energy at the discrete level, and
second-order accuracy in time is provided by the LSDIRK2 integrator
\cite{LSDIRK2}. Since only the convective terms enter the explicit stability
constraint, the time step is governed solely by the fluid velocity.

The choice of a pressure-based implicit step is not a mere generalisation. In
magneto-inertial fusion the compressed material is described by tabulated closures such
as SESAME \cite{SESAME} or QEOS \cite{QEOS}, in which the pressure depends nonlinearly
on the internal energy; energy-based formulations rely on that dependence being linear,
as it is for a $\gamma$-law gas, and lose the linearity of the resulting system as soon
as a realistic closure is used. The pressure-based formulation confines the
nonlinearity to the diagonal and leaves the coupling untouched: a single
GMRES solve suffices for a linear closure, and a nested Newton iteration for a
nonlinear one. As a verification closure we employ the Redlich--Kwong equation of state
\cite{RedlichKwong1949}, which is analytic, admits an exact Riemann solution
(Section~\ref{sec:rk_riemann}), and whose nonlinearity enters the discrete system at
the same place as that of a tabulated closure. The collocated cell-centred layout,
adopted here in place of the staggered meshes of \cite{Dumbser2019,Fambri2021}, keeps
the full set of conserved variables at a single point, which is what mixed-material
Riemann solvers require at sharp interfaces and hence what makes the algorithm
extensible to plasma--solid multi-physics configurations.

One of the central contributions is a mechanism that guarantees positivity of the pressure as
an intrinsic part of the implicit solve. Existing pressure-positivity-preserving technology for
MHD is almost exclusively explicit and high-order, relying on a posteriori limiters
combined with Harten--Lax--van Leer fluxes and a discrete control of the magnetic field
divergence \cite{Balsara2012,wu2018,ChristliebEtAl2015,DingWu2024,WLSENO_PP2021};
such clipping restores positivity at the expense of conservation. Here positivity is
built into the pressure sub-system by adapting the wetting-and-drying treatment of
Casulli \cite{Casulli2009} and the piecewise-linear-system theory of Brugnano and
Casulli \cite{BrugnanoCasulli2009}. A steep-slope relaxation makes the bounded
constraint compatible with the same nested Newton iteration already used for the
equation of state and acts exclusively on the diagonal, leaving the flux coupling that
carries the discrete energy conservation unaffected. It yields a
configuration-independent lower bound on the pressure which can be made strictly
positive by construction, so that pressure positivity and exact energy conservation hold
simultaneously. Well-posedness and convergence follow the framework of
\cite{Casulli2009,BrugnanoCasulli2009,CasulliZanolli2012}; the novelty is to transfer
this idea to the MHD pressure sub-system with a general equation of state.

Finally, the solenoidal constraint $\nabla\!\cdot\mathbf{B}=0$ is not automatically
inherited at the discrete level \cite{Toth2000}. It is usually enforced through
constrained transport \cite{EvansHawley1988,GardinerStone2005,BalsaraSpicer1999,Rossmanith2006},
hyperbolic divergence cleaning \cite{Dedner2002}, or source-term and projection
approaches \cite{Powell1999,BrackbillBarnes1980}. In this work, we employ two constrained transport
variants, both built on the same electric field, obtained directly from the interface
fluxes of the implicit magnetic sub-system rather than from an auxiliary Riemann solve:
a staggered adaptation of Gardiner and Stone \cite{GardinerStone2005}, and a fully
unstaggered, potential-based variant in the spirit of Rossmanith \cite{Rossmanith2006}
which retains the cell-centred layout.

The rest of the paper is organised as follows.
Section~\ref{sec:governing} introduces the governing equations, the equations of state
and the flux splitting; Section~\ref{sec:scheme} the semi-implicit IMEX discretisation;
Sections~\ref{sec:positivity} and \ref{sec:properties} the pressure positivity construction and
its properties. Section~\ref{sec:results} assesses the scheme on a suite of benchmarks
in one and two space dimensions, and conclusions are drawn in
Section~\ref{sec:conclusions}.

\section{System of Governing Equations}
\label{sec:governing}
In conservative form, the ideal MHD equations read
\begin{equation}
\renewcommand{\arraystretch}{1.6}
\label{equ:ideal MHD}
\begin{cases}
\dfrac{\partial \rho}{\partial t}
+ \nabla \cdot (\rho \mathbf{u}) = 0, \\
\dfrac{\partial (\rho \mathbf{u})}{\partial t}
+ \nabla \cdot \left( \rho \mathbf{u} \otimes \mathbf{u}
+ \left( p + \dfrac{\|\mathbf{B}\|^2}{2} \right)\mathbf{I}
- \mathbf{B} \otimes \mathbf{B} \right) = \mathbf{0}, \\
\dfrac{\partial E}{\partial t}
+ \nabla \cdot \left( \left( E + p + \dfrac{\|\mathbf{B}\|^2}{2} \right)\mathbf{u}
- (\mathbf{u}\cdot\mathbf{B})\,\mathbf{B} \right) = 0, \\
\dfrac{\partial \mathbf{B}}{\partial t}
+ \nabla \cdot \left( \mathbf{B} \otimes \mathbf{u}
- \mathbf{u} \otimes \mathbf{B} \right) = \mathbf{0}.
\end{cases}
\end{equation}
The first equation enforces mass conservation, the second governs momentum 
(one equation per spatial direction), the third accounts for the conservation 
of total energy, and the last evolves the magnetic field. Here $\rho$ denotes 
the fluid density, $\mathbf{u} = [u, v, w]^{\mathsf T}$ is the velocity 
field, $\mathbf{B} = [B_x, B_y, B_z]^{\mathsf T}$ is the magnetic field, and 
$E$ is the total energy per unit volume, which is defined as the sum of
internal ($\rho e$), kinetic $k$, and magnetic ($m$) energy contributions, namely
\begin{equation}
\label{equ:defofU}
E = \rho e + k + m, \ k = \tfrac{1}{2}\rho |\mathbf{u}|^2, \ m =  \tfrac{1}{2}||\mathbf{B}||^2,
\end{equation}
with $e$ being the specific internal energy. 

The magnetic field is nondimensionalized such that the vacuum magnetic 
permeability is set to unity. In addition, the magnetic field must satisfy the 
solenoidal (divergence-free) constraint
\begin{equation}
\label{equ:divB}
\nabla \cdot \mathbf{B} = 0,
\end{equation}
which is an constraint of the system: if it holds at the initial time, the 
evolution equations~\eqref{equ:ideal MHD} preserve it for all later times 
at the continuous level. At the discrete level, however, this property 
is not automatically inherited, and a dedicated treatment is required to enforce 
it.

An equation of state of the form $p = p(\rho, e)$ is required to close the 
system. A key feature of the proposed framework is that it does not rely on any particular 
thermodynamic model, but accommodates a generic equation of state (EOS), 
encompassing both linear and nonlinear pressure laws. In this work we consider two 
representative closures: the ideal gas law, as a linear EOS for which the pressure 
depends linearly on the internal energy (Section~\ref{sec:ideal gas eos}), and the 
Redlich--Kwong equation of state, as a nonlinear example (Section~\ref{sec:redlich-kwong eos}).

Letting $\mathbf{U} = (\rho, \rho\mathbf{u}, E, \mathbf{B})^{\mathsf T}$ denote 
the vector of conserved quantities, the  ideal MHD system \eqref{equ:ideal MHD} can be 
written in compact form as
\begin{equation}
\label{equ:compact_MHD}
\partial_t \mathbf{U} + \nabla \cdot \boldsymbol{\mathcal{F}}(\mathbf{U}) = \mathbf{0},
\end{equation}
where $\boldsymbol{\mathcal{F}}(\mathbf{U})$ 
represents the set of flux vectors for all the variables. For one-dimensional flows aligned with the $x$-direction, the governing equations \eqref{equ:ideal MHD} read:
\begin{equation}
\label{equ:1DMHD}
\frac{\partial}{\partial t}
\begin{bmatrix}
\rho \\
\rho u \\
\rho v \\
\rho w \\
E \\
B_x \\
B_y \\
B_z
\end{bmatrix}
+
\frac{\partial}{\partial x}
\begin{bmatrix}
\rho u \\
\rho u^2 + p + \frac{1}{2}||\mathbf{B}||^2 - B_x^2 \\
\rho u v - B_x B_y \\
\rho u w - B_x B_z \\
u (E + p + \frac{1}{2}||\mathbf{B}||^2) - B_x (\mathbf{u}\cdot\mathbf{B}) \\
0 \\
u B_y - v B_x \\
u B_z - w B_x
\end{bmatrix}
= 0 .
\end{equation}
The above system has eight eigenvalues given by
\begin{equation}
\lambda_{1,8} = u \pm c_f, \qquad
\lambda_{2,7} = u \pm c_a, \qquad
\lambda_{3,6} = u \pm c_s, \qquad
\lambda_4 = u, \qquad
\lambda_5 = 0 ,
\end{equation}
with the sound speed $a$, Alfvén speed $c_a$, and the fast and slow magnetosonic wave speeds
$c_f$ and $c_s$, defined as:

\begin{equation}
\begin{gathered}
    a^2 = \frac{\partial p}{\partial \rho} + \frac{p}{\rho^2} \frac{\partial p}{\partial e}, \\
    c_a = \frac{B_x}{\sqrt{\rho}},\\
    c_{f,s}^2 =
\frac{1}{2}
\left(
\frac{||\mathbf{B}||^2}{\rho} + a^2
\pm
\sqrt{
\left(
\frac{||\mathbf{B}||^2}{\rho} + a^2
\right)^2
-
\frac{4 B_x^2 a^2}{\rho}
}
\right).
\end{gathered}
\end{equation}

\subsection{Ideal Gas EOS}
\label{sec:ideal gas eos}
For an ideal gas, the thermal and caloric equations of state are given by
\begin{equation}
\frac{p}{\rho} = R T, \qquad e = c_v T.
\end{equation}
where $T$ represents the temperature and $R$ is the gas constant. By eliminating the temperature using these relations, the equation of state
can be recast in the form $e = e(p,\rho)$, yielding

\begin{equation}
e(p,\rho) = \frac{p}{(\gamma - 1)\rho}.
\end{equation}
Here, $c_v$ and $c_p$ are the specific heat capacities at constant volume and constant
pressure, respectively, and $\gamma = c_p / c_v$ denotes the ratio of specific heats. In this case, the relation between pressure and the internal energy is linear.

For ideal gases the sound speed reduces to the standard expression

\begin{equation}
    a = \sqrt{\frac{\gamma p}{\rho}}.
\end{equation}

\subsection{Redlich-Kwong EOS}
\label{sec:redlich-kwong eos}
We consider a class of general cubic equations of state, which can be written
in the form
\begin{equation}
p(T,\rho) = \frac{RT}{v_{sp}-b}
- \frac{\mathfrak{a}(T)}{(v_{sp} - b r_1)(v_{sp} - b r_2)},
\end{equation}
where $v_{sp} = 1/\rho$ denotes the specific volume, $b$ is the covolume parameter,
and $r_1$ and $r_2$ are model constants. The temperature-dependent function
$\mathfrak{a}(T)$ represents the attractive contribution of the equation of state.
The corresponding caloric equation of state associated with this formulation
reads
\begin{equation}
e(T,\rho) = c_v T
+ \frac{\mathfrak{a}(T) - T \mathfrak{a}'(T)}{b}\, \mathcal{U}(v_{sp},b,r_1,r_2),
\end{equation}
where $\mathfrak{a}'(T) = \mathrm{d}\mathfrak{a}(T)/\mathrm{d}T$ and
\begin{equation}
\mathcal{U}(v_{sp},b,r_1,r_2)
= \frac{1}{r_1 - r_2}
\ln\!\left(\frac{v_{sp} - b r_1}{v_{sp} - b r_2}\right).
\end{equation}

Different cubic equations of state are recovered through appropriate choices
of the parameters appearing in the above expressions. In particular, the
Redlich--Kwong equation of state is obtained by setting $r_1 = 0$ and
$r_2 = -1$, together with the temperature-dependent attraction term
$\mathfrak{a}(T) = \frac{\mathfrak{a}_0}{\sqrt{T}}$. Under these assumptions, the relation between pressure
and internal energy becomes nonlinear. To evaluate the internal energy in the
form $e(p,\rho)$, the temperature is first determined by numerically solving
the nonlinear thermal equation of state for a given density and pressure.
In the present work, this inversion is performed using an efficient Newton
iteration. Once the temperature has been obtained, it is substituted into the
caloric equation of state to recover the desired relation $e(p,\rho)$.

For Redlich-Kwong EOS, the sound speed is

\begin{equation}
\begin{aligned}
a= \sqrt{
\frac{RT}{(v_{sp}-b)^2}\,\frac{1}{\rho^2}
-
\frac{\mathfrak{a}(T)}{\rho^2}
\left[
\frac{1}{v_{sp}^2(v_{sp}+b)}
+
\frac{1}{v_{sp}(v_{sp}+b)^2}
\right]
+
\frac{T}{\rho^2 c_v}
\left(
\frac{R}{v_{sp}-b}
-
\frac{\mathfrak{a}'(T)}{v_{sp}(v_{sp}+b)}
\right)^2}.
\end{aligned}
\end{equation}

\subsection{3-Split form of the Ideal MHD Equations}
\label{sec:3-split}
To handle the multiple time scales of the model we resort to a flux splitting technique,
which allows a different time discretisation to be designed for each sub-system
according to its stiffness. We adopt here the three-way splitting introduced by Fambri
\cite{Fambri2021}, in which the fluxes of the MHD system \eqref{equ:ideal MHD} are
separated into three contributions: (i) advective fluxes $\mathbf{F}^c$, associated
with the transport by the fluid velocity; (ii) pressure fluxes $\mathbf{F}^p$,
associated with the gas pressure $p$; and (iii) magnetic fluxes $\mathbf{F}^{\mathbf B}$,
associated with the magnetic field $\mathbf B$. Explicitly, the three contributions are
defined as

\begin{equation}
\label{equ:3_split}
\renewcommand{\arraystretch}{} 
\begin{array}{c}
\mathbf{F}^c(\mathbf{U}) =
\begin{bmatrix}
\rho u \\
\rho u^2\\
\rho u v\\
\rho u w\\
 k u\\
0 \\
0 \\
0
\end{bmatrix}, 
\begin{array}{cc}
\mathbf{F}^p(\mathbf{U}) =
\begin{bmatrix}
0 \\
p \\
0 \\
0 \\
\rho h u\\
0 \\
0 \\
0
\end{bmatrix},
&
\mathbf{F}^\mathbf{B}(\mathbf{U}) =
\begin{bmatrix}
0 \\
\frac{\|\mathbf{B}\|^2}{2} - B_x^2 \\
-B_x B_y \\
-B_x B_z \\
\|\mathbf{B}\|^2 u - B_x (\mathbf{u} \cdot \mathbf{B}) \\
0 \\
u B_y - v B_x \\
u B_z - w B_x
\end{bmatrix}
\end{array}
\end{array},
\end{equation}
where $h = e + \frac{p}{\rho}$ represents the specific enthalpy. A distinctive 
feature of this splitting is that the total energy flux is decomposed consistently 
across the three sub-systems: the advective flux $\mathbf{F}^c$ transports the 
kinetic energy through the term $ku$, the magnetic flux $\mathbf{F}^\mathbf{B}$ 
accounts for the magnetic energy, and the pressure flux $\mathbf{F}^p$ carries the 
internal energy via the enthalpy term $\rho h u$. The evolution of the total energy 
is thus shared among the kinetic, magnetic, and internal contributions, each 
updated by its corresponding sub-system.

Accordingly, the one-dimensional system in its novel three-way split form can be concisely represented as
\begin{equation}
    \frac{\partial \mathbf{U}}{\partial t} + \frac{\partial \mathbf{F}^c(\mathbf{U})}{\partial x} + \frac{\partial \mathbf{F}^p(\mathbf{U})}{\partial x} + \frac{\partial \mathbf{F}^{\mathbf{B}}(\mathbf{U})}{\partial x} = \mathbf{0}.
\end{equation}

\begin{enumerate}
\item The advective sub-system is
\begin{equation}
\label{equ:advective sub-system}
\frac{\partial \mathbf U}{\partial t}
+ \frac{\partial \mathbf F^c(\mathbf U)}{\partial x}
= 0,
\end{equation}

whose Jacobian admits the following set of real eigenvalues:
\begin{equation}
\lambda^c_{1,2,3,4} = 0,
\qquad
\lambda^c_{5,6,7,8} = u .
\end{equation}

Thus, the sub-system is weakly hyperbolic, since one eigenvector associated with
$\lambda^c = 0$ is missing. The characteristic speeds depend solely on the fluid velocity.

\item The pressure sub-system is
\begin{equation}
\frac{\partial \mathbf U}{\partial t}
+ \frac{\partial \mathbf F^{p}(\mathbf U)}{\partial x}
= 0,
\label{equ:pressure sub-system}
\end{equation}

whose Jacobian admits the following set of eigenvalues:
\begin{equation}
\lambda^{p}_{1,2,3,4,5,6} = 0,
\qquad
\lambda^{p}_{7,8}
= \frac{1}{2} (u \pm \sqrt{u^2 +  4a^2}).
\end{equation}

Consequently, the pressure sub-system is hyperbolic, though not strictly: the
eigenvalue $\lambda^p = 0$ has algebraic multiplicity six and admits a complete set of
eigenvectors, so that the Jacobian remains diagonalizable.

\item The magnetic sub-system is
\begin{equation}
\frac{\partial \mathbf U}{\partial t}
+ \frac{\partial \mathbf F^{\mathbf{B}}(\mathbf U)}{\partial x}
= 0,
\label{equ:magnective sub-system}
\end{equation}
and the following real eigenvalues are found
\begin{equation}
    \lambda^\mathbf{B}_{1,2,3,4} = 0, \ 
    \lambda^\mathbf{B}_{5,6} = \frac{1}{2}\left( u \pm \sqrt{u^2 + 4 c_a^2} \right), \ 
\lambda^\mathbf{B}_{7,8} =     \frac{1}{2}\left( u \pm \sqrt{u^2 + 4 c_\mathbf{B}^2} \right).
\end{equation}
Here, $c_\mathbf{B} = \frac{\lVert \mathbf{B} \rVert}{\sqrt{\rho}}$. Similarly, associated with the eigenvalue $\lambda^{\mathbf B}=0$ there are two
eigenvectors missing, so the magnetic sub-system is weakly hyperbolic.
\end{enumerate}

By inspecting the eigenvalues of the three sub-systems, it becomes evident that, in the low acoustic Mach number regime, the acoustic sound speed $a$ appearing in pressure sub-system dominates the material velocity $u$, thus introducing severe stiffness associated with the pressure-related terms.
Likewise, as the magnitude of the magnetic field $\|\mathbf B\|$ increases, the MHD model becomes increasingly governed by the magnetic sub-system.

Taking these observations into account, together with the characteristic structure of the three subsystems, it is clear that the proposed flux vector splitting is particularly well suited for the construction of semi-implicit numerical schemes for MHD flows. In the present approach, the advective sub-system is treated explicitly, while both the pressure and magnetic sub-systems are discretised implicitly.

This approach yields two key advantages. On the one hand, it avoids the loss of accuracy
typically observed in fully explicit schemes in the low acoustic Mach number regime, where the
artificial dissipation of upwind fluxes scales with the fastest signal speed rather than
with the fluid velocity, so that the asymptotic scaling $p = p_0 + \mathcal{O}(M_c^2)$ is
not recovered on a fixed mesh \cite{GuillardViozat2017,Dellacherie2010}. Since the fast
magnetosonic speed degenerates to the Alfv\'en speed in the magnetic-pressure dominated
regime, the same degradation reappears as $M_b \to 0$. In the present splitting only the
advective sub-system is discretised with an upwind flux, whose dissipation is therefore
proportional to the material velocity alone.

At the continuous level, the simultaneous low-acoustic-Mach and low-Alfvén-Mach-number limit of
the compressible ideal MHD equations converges, for well-prepared initial data, to
the incompressible MHD equations \cite{ChengJuSchochet2021}. The
present semi-implicit construction is consistent with this limit: as $M_c\to 0$ the
implicit pressure sub-system reduces to a discrete elliptic equation enforcing the
divergence constraint on the velocity field, while the implicit treatment of the
magnetic sub-system removes the Alfvén time step restriction as $M_b\to 0$. The consistency with the incompressible limit, and the insensitivity
of the time step to the acoustic and Alfvén stiffness,
are confirmed by the numerical experiments of Section~\ref{sec:results}.

\section{Semi-Implicit Numerical Scheme}
\label{sec:scheme}

We now proceed to detail the semi-implicit algorithm developed in this work. After
introducing some general notation, we describe the explicit update of the advective
sub-system, followed by the implicit updates of the magnetic and pressure sub-systems
and the shock-activated stabilisation. We then present the positivity-preserving
modification of the pressure system together with its properties, before addressing
the enforcement of the divergence constraint and the extension of the scheme to
second-order accuracy in time.

For clarity of presentation, the numerical scheme is formulated on a two-dimensional
uniform Cartesian mesh assuming slab symmetry in the $z$-direction. The extension to
three spatial dimensions follows in a straightforward manner. We therefore consider
a two-dimensional computational domain
\[
\Omega(\mathbf{x}) = [x_{\min}, x_{\max}] \times [y_{\min}, y_{\max}],
\]
which is discretised by a regular Cartesian grid composed of
\[
N_e = N_x \times N_y
\]
control volumes. The corresponding mesh spacings are defined as
\[
\Delta x = \frac{x_{\max} - x_{\min}}{N_x}, \qquad
\Delta y = \frac{y_{\max} - y_{\min}}{N_y}.
\]

Each control volume is uniquely identified by an integer index pair $(i,j)$, associated
with the $x$- and $y$-direction, respectively. Cell interfaces are denoted by the indices
$(i+\tfrac{1}{2},j)$ and $(i,j+\tfrac{1}{2})$, while cell corners are referred to as
$(i+\tfrac{1}{2},j+\tfrac{1}{2})$. With this notation, the cell centre is located at
\[
\mathbf{x}_{i,j} = (x_i, y_j),
\]
whereas the centres of the cell faces are given by
\[
\mathbf{x}_{i+\frac{1}{2},j}
=
\left(
\frac{x_i + x_{i+1}}{2},\, y_j
\right),
\qquad
\mathbf{x}_{i,j+\frac{1}{2}}
=
\left(
x_i,\, \frac{y_j + y_{j+1}}{2}
\right),
\]
and the corner of the cells are given by
\[
\mathbf{x}_{i+\frac{1}{2},j + \frac{1}{2}}
=
\left(
\frac{x_i + x_{i+1}}{2},\, \frac{y_i + y_{j+1}}{2}
\right).
\]

The advective sub-system (\ref{equ:advective sub-system}) is discretised on a collocated grid, with all conservative variables defined at the centres of the control volumes. The magnetic-pressure sub-system is likewise treated in a cell-centred framework, following the approach adopted by Boscheri and Pareschi ~\cite{BoscheriPareschi2021} for the Euler and Navier--Stokes equations, and by Dematt\'e et al.~\cite{Dematte2024} for the MHD
equations. In addition, for the staggered and unstaggered adaptation of the constrained transport method 
described in Section~\ref{chp:divfree_CT}, the magnetic field is reconstructed 
from a cell-centred magnetic vector potential in order to enforce the 
divergence-free constraint \cite{Dematte2024}. Finally, discrete time levels are denoted by 
superscripts, while the spatial discretisation is indicated by subscripts.

\subsection{Advective sub-system}
\label{sec:advective}
The advective sub-system~\eqref{equ:advective sub-system} is discretised by means of a conservative, explicit
finite-volume formulation. Given the solution vector $\mathbf{U}_{i,j}^n$ at time
level $t^n$, an intermediate state $\mathbf{U}_{i,j}^\star$ is computed as
\begin{equation}
\mathbf{U}_{i,j}^\star
=
\mathbf{U}_{i,j}^n
-
\frac{\Delta t}{\Delta x}
\left(
\mathbf{F}^c_{i+\frac12,j}
-
\mathbf{F}^c_{i-\frac12,j}
\right)
-
\frac{\Delta t}{\Delta y}
\left(
\mathbf{G}^c_{i,j+\frac12}
-
\mathbf{G}^c_{i,j-\frac12}
\right),
\label{eq:advective_update}
\end{equation}
where $\mathbf{F}^c$ and $\mathbf{G}^c$ denote the numerical fluxes associated with
the convective sub-system in the $x$- and $y$-directions, respectively.
The numerical fluxes are evaluated using a Rusanov-type
approximate Riemann solver. In the $x$-direction, the flux at the cell interface
$x_{i+\frac12,j}$ is given by
\begin{equation}
\mathbf{F}^c_{i+\frac12,j}
=
\frac12
\left(
\mathbf{F}^c(\mathbf{U}^R_{i+\frac12,j})
+
\mathbf{F}^c(\mathbf{U}^L_{i+\frac12,j})
\right)
-
\frac12
s^{\max}_{i+\frac12,j}\,
\mathsf{P}
\left(
\mathbf{U}^R_{i+\frac12,j}
-
\mathbf{U}^L_{i+\frac12,j}
\right),
\qquad
\mathsf{P} = \operatorname{diag}(1,1,1,1,1,0,0,0),
\label{eq:rusanov_flux}
\end{equation}
where $\mathbf{U}^L_{i+\frac12,j}$ and $\mathbf{U}^R_{i+\frac12,j}$ are the left and
right reconstructed states at the interface. The quantity
$s^{\max}_{i+\frac12,j}$ represents the maximum characteristic speed of the
advective sub-system and is defined as
\begin{equation}
s^{\max}_{i+\frac12,j}
=
\max\!\left(
\left|u^L_{i+\frac12,j}\right|,
\left|u^R_{i+\frac12,j}\right|
\right),
\end{equation}
so that the numerical dissipation is proportional only to the material velocity.
The projection $\mathsf{P}$ restricts the upwind dissipation to the variables
advanced by the advective sub-system. 
The evaluation of the numerical
fluxes in the $y$-direction follows the same procedure and is therefore omitted for
brevity.

A first order accurate numerical flux is employed by
setting
$\mathbf{U}^L_{i+\frac12,j} = \mathbf{U}_{i,j}$ and
$\mathbf{U}^R_{i+\frac12,j} = \mathbf{U}_{i+1,j}$.
To achieve second-order spatial accuracy, a TVD reconstruction is adopted with the minmod limiter.

The time step $\Delta t$ is selected according to a CFL-type stability condition,
\begin{equation}
\Delta t
=
\mathrm{CFL}
\left(
\frac{\max|\lambda_x^c|}{\Delta x}
+
\frac{\max|\lambda_y^c|}{\Delta y}
\right)^{-1},
\label{eq:cfl_advective}
\end{equation}
where $\lambda_d^c$ $(d=x,y)$ denotes the eigenvalues of the advective sub-system
and $\mathrm{CFL}<1$ is the Courant number. Since only the advective sub-system is treated explicitly, the time
step~\eqref{eq:cfl_advective} is dictated by the fluid velocity alone. A rigorous proof that
the scheme is stable under a CFL condition based only on the advective eigenvalues
is not available; such an analysis is not straightforward even for a simplified
problem, and is left to future work. Nonetheless, all the numerical experiments of
Section~\ref{sec:results} are advanced with this mild velocity-based condition,
providing extensive evidence that stability holds independently of the acoustic and
Alfvén (fast magnetosonic) wave speeds.

\subsection{Magnetic-Pressure sub-system}
Having evolved the advective sub-system and following the flux splitting approach
detailed in Section~\ref{sec:3-split}, we now move to consider the solution of the
magnetic and pressure sub-systems. All remaining flux contributions are treated
implicitly. Rather than coupling the two sub-systems through an outer iteration, they
are solved in sequence within a single time step: the magnetic field is advanced first
and, once known, enters the pressure sub-system.

1. Through the coupling between the magnetic field and the momentum equation, the
   magnetic field is advanced to the new time level, yielding $\mathbf{B}^{n+1}$.

2. By coupling the energy equation with the momentum equation, the pressure
   is updated to the new time level, resulting in $p^{n+1}$.

Since the algorithm involves several nested routines, each possibly consisting of multiple steps, a simplified overview of a single time step is provided in Fig.~\ref{fig: Schematic illustration}. Although the three-dimensional extension is not detailed here, it follows in a relatively straightforward manner.

\subsubsection{Magnetic sub-system}
\label{sec:magnetic_subsystem}
Following Boscheri and Thomann \cite{BoscheriThomann2024}, we first advance the magnetic
sub-system by substituting the momentum equations into the induction equation, with the
gas pressure gradient retained at the old time level and the quadratic magnetic terms
linearised as $\|\mathbf{B}\|^2 \approx \mathbf{B}^{n}\cdot\mathbf{B}^{n+1}$, so that the
resulting system is linear in the unknown $\mathbf{B}^{n+1}$. The time discretisation of
the implicit magnetic sub-system reads

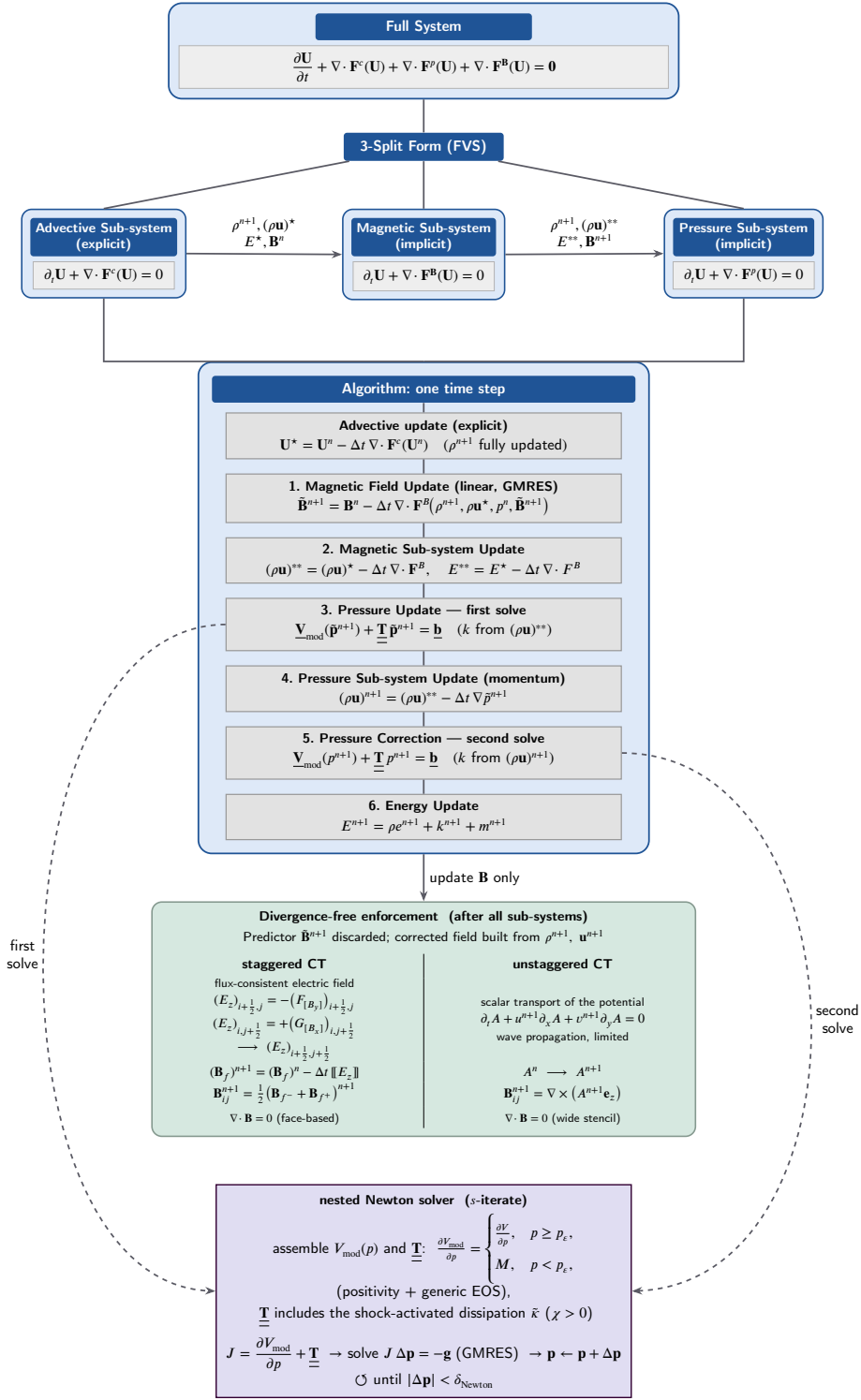
\begin{figure}[H]
\centering
\resizebox{!}{0.88\textheight}{%
\begin{tikzpicture}[
  font=\small,
  >={Stealth[length=2.2mm]},
  every node/.style={transform shape},
  container/.style={rounded corners=8pt, fill=bluebg, draw=bluehead, line width=1pt},
  subcont/.style={rounded corners=6pt, fill=bluebg, draw=bluehead, line width=0.9pt},
  header/.style={rounded corners=2pt, fill=bluehead, draw=none, font=\bfseries, text = white,inner sep=4pt, minimum width=3.4cm, align=center},
  subhead/.style={rounded corners=2pt, fill=bluehead, draw=none, text=white, font=\bfseries\footnotesize,
                 inner sep=3pt, minimum width=3.0cm, align=center},
  greybox/.style={fill=greybg, draw=black!30, line width=0.5pt,
                  align=center, inner sep=4pt},
  fluxbox/.style={fill=greybg, draw=black!30, line width=0.5pt,
                  align=center, inner sep=4pt, minimum width=3.0cm},
  stepbox/.style={fill=greystep, draw=black!40, line width=0.7pt,
               align=center, inner sep=4pt, minimum width=8.4cm},
  newtonbox/.style={fill=lilac, draw=violet!40!black, line width=0.8pt,
                    align=center, inner sep=6pt, minimum width=8.4cm},
  ctbox/.style={rounded corners=6pt, fill=ctbg, draw=cthead!70!black, line width=0.9pt,
               align=center, inner sep=5pt, minimum width=9.6cm},
  link/.style={draw=linecol, line width=1pt},
  arrow/.style={draw=linecol, line width=1pt, ->},
  darrow/.style={draw=linecol, line width=1pt, ->, dashed},
]

\node[header, minimum width=4.6cm] (fullhead) {Full System};
\node[greybox, below=1mm of fullhead, minimum width=10.4cm]
  (fulleq) {$\dfrac{\partial \mathbf{U}}{\partial t}
            +\nabla\!\cdot\mathbf{F}^{c}(\mathbf{U})
            +\nabla\!\cdot\mathbf{F}^{p}(\mathbf{U})
            +\nabla\!\cdot\mathbf{F}^{\mathbf{B}}(\mathbf{U})=\mathbf{0}$};
\begin{scope}[on background layer]
  \node[container, fit=(fullhead)(fulleq), inner sep=6pt] (fullbox) {};
\end{scope}

\node[header, below=7mm of fullbox] (split) {3-Split Form (FVS)};
\draw[link] (fullbox) -- (split);

\node[subhead, below=12mm of split, xshift=-6.8cm] (advhead) {Advective Sub-system\\(explicit)};
\node[fluxbox, below=1mm of advhead]
  (adveq) {$\partial_t\mathbf{U}+\nabla\!\cdot\mathbf{F}^{c}(\mathbf{U})=0$};
\begin{scope}[on background layer]
  \node[subcont, fit=(advhead)(adveq), inner sep=5pt] (advbox) {};
\end{scope}

\node[subhead, below=12mm of split] (maghead) {Magnetic Sub-system\\(implicit)};
\node[fluxbox, below=1mm of maghead]
  (mageq) {$\partial_t\mathbf{U}+\nabla\!\cdot\mathbf{F}^{\mathbf{B}}(\mathbf{U})=0$};
\begin{scope}[on background layer]
  \node[subcont, fit=(maghead)(mageq), inner sep=5pt] (magbox) {};
\end{scope}

\node[subhead, below=12mm of split, xshift=6.8cm] (prehead) {Pressure Sub-system\\(implicit)};
\node[fluxbox, below=1mm of prehead]
  (preeq) {$\partial_t\mathbf{U}+\nabla\!\cdot\mathbf{F}^{p}(\mathbf{U})=0$};
\begin{scope}[on background layer]
  \node[subcont, fit=(prehead)(preeq), inner sep=5pt] (prebox) {};
\end{scope}

\draw[link] (split) -- (advbox.north);
\draw[link] (split) -- (magbox.north);
\draw[link] (split) -- (prebox.north);

\draw[arrow] (advbox.east) -- node[above,font=\small,align=center]
  {$\rho^{n+1},(\rho\mathbf{u})^{\star}$\\[-1pt]$E^{\star},\mathbf{B}^{n}$} (magbox.west);
\draw[arrow] (magbox.east) -- node[above,font=\small,align=center]
  {$\rho^{n+1},(\rho\mathbf{u})^{\ast\ast}$\\[-1pt]$E^{\ast\ast},\mathbf{B}^{n+1}$} (prebox.west);

\node[header, below=16mm of magbox, minimum width=9.0cm]
  (alghead) {Algorithm: one time step};

\node[stepbox, below=2mm of alghead] (s0) {%
  \stitle{Advective update (explicit)}\\[1pt]
  $\mathbf{U}^{\star}=\mathbf{U}^{n}-\Delta t\,\nabla\!\cdot\mathbf{F}^{c}(\mathbf{U}^{n})
   \quad(\rho^{n+1}\ \text{fully updated})$};

\node[stepbox, below=3mm of s0] (s1) {%
  \stitle{1.\ Magnetic Field Update (linear, GMRES)}\\[1pt]
  $\tilde{\mathbf{B}}^{n+1}=\mathbf{B}^{n}
   -\Delta t\,\nabla\!\cdot\mathbf{F}^{B}\!\big(\rho^{n+1},\rho\mathbf{u}^{\star},p^{n},\tilde{\mathbf{B}}^{n+1}\big)$};

\node[stepbox, below=3mm of s1] (s2) {%
  \stitle{2.\ Magnetic Sub-system Update}\\[1pt]
  $(\rho\mathbf{u})^{\ast\ast}=(\rho\mathbf{u})^{\star}-\Delta t\,\nabla\!\cdot\mathbf{F}^{B},\quad
   E^{\ast\ast}=E^{\star}-\Delta t\,\nabla\!\cdot F^{B}$};

\node[stepbox, below=3mm of s2] (s3) {%
  \stitle{3.\ Pressure Update --- first solve}\\[1pt]
  $\underline{\mathbf{V}}_{\mathrm{mod}}(\tilde{\mathbf{p}}^{n+1})
   +\underline{\underline{\mathbf{T}}}\,\tilde{\mathbf{p}}^{n+1}=\underline{\mathbf{b}}
   \quad(\text{$k$ from }(\rho\mathbf{u})^{\ast\ast})$};

\node[stepbox, below=3mm of s3] (s4) {%
  \stitle{4.\ Pressure Sub-system Update (momentum)}\\[1pt]
  $(\rho\mathbf{u})^{n+1}=(\rho\mathbf{u})^{\ast\ast}-\Delta t\,\nabla\tilde{p}^{n+1}$};

\node[stepbox, below=3mm of s4] (s5) {%
  \stitle{5.\ Pressure Correction --- second solve}\\[1pt]
  $\underline{\mathbf{V}}_{\mathrm{mod}}(p^{n+1})
   +\underline{\underline{\mathbf{T}}}\,p^{n+1}=\underline{\mathbf{b}}
   \quad(\text{$k$ from }(\rho\mathbf{u})^{n+1})$};

\node[stepbox, below=3mm of s5] (s6) {%
  \stitle{6.\ Energy Update}\\[1pt]
  $E^{n+1}=\rho e^{n+1}+k^{n+1}+m^{n+1}$};

\foreach \a/\b in {alghead/s0, s0/s1, s1/s2, s2/s3, s3/s4, s4/s5, s5/s6}
  \draw[link] (\a) -- (\b);

\begin{scope}[on background layer]
  \node[container, fit=(alghead)(s0)(s1)(s2)(s3)(s4)(s5)(s6),
        inner sep=8pt] (algbox) {};
\end{scope}

\draw[link] (advbox.south) |- ($(algbox.north)+(0,0.0)$);
\draw[link] (prebox.south) |- ($(algbox.north)+(0,0.0)$);

\node[ctbox, below=10mm of algbox] (ct) {%
  \stitle{Divergence-free enforcement\ \ (after all sub-systems)}\\[3pt]
  \begin{minipage}{11.2cm}\centering
  \footnotesize
  Predictor $\tilde{\mathbf{B}}^{n+1}$ discarded; corrected field built from
  $\rho^{n+1},\ \mathbf{u}^{n+1}$
  \\[5pt]
  \renewcommand{\arraystretch}{1.35}
  \begin{tabular}{@{}c@{\hspace{4mm}}|@{\hspace{4mm}}c@{}}
  \textbf{staggered CT} & \textbf{unstaggered CT}\\[2pt]
  \begin{minipage}{5.1cm}\centering
    {\scriptsize flux-consistent electric field}\\[1pt]
    $(E_z)_{i+\frac12,j}=-\big(F_{[B_y]}\big)_{i+\frac12,j}$\\
    $(E_z)_{i,j+\frac12}=+\big(G_{[B_x]}\big)_{i,j+\frac12}$\\[1pt]
    $\longrightarrow\ (E_z)_{i+\frac12,j+\frac12}$
  \end{minipage}
  &
  \begin{minipage}{5.1cm}\centering
    {\scriptsize scalar transport of the potential}\\[1pt]
    $\partial_t A+u^{n+1}\partial_x A+v^{n+1}\partial_y A=0$\\[1pt]
    {\scriptsize wave propagation, limited}
  \end{minipage}\\[4pt]
  $(\mathbf{B}_f)^{n+1}=(\mathbf{B}_f)^{n}-\Delta t\,[\![E_z]\!]$
    & $A^{n}\ \longrightarrow\ A^{n+1}$\\
  $\mathbf{B}^{n+1}_{ij}=\tfrac12\big(\mathbf{B}_{f^-}+\mathbf{B}_{f^+}\big)^{n+1}$
    & $\mathbf{B}^{n+1}_{ij}=\nabla\times\big(A^{n+1}\mathbf{e}_z\big)$\\[2pt]
  {\scriptsize$\nabla\!\cdot\mathbf{B}=0$ (face-based)}
    & {\scriptsize$\nabla\!\cdot\mathbf{B}=0$ (wide stencil)}
  \end{tabular}
  \end{minipage}};
\draw[arrow] (algbox.south) -- node[right,font=\small,align=left]
  {update $\mathbf{B}$ only} (ct.north);

\node[newtonbox, below=10mm of ct] (nw) {%
  \stitle{nested Newton solver\ \ ($s$-iterate)}\\[3pt]
  \begin{minipage}{8.4cm}\centering
  assemble $V_{\mathrm{mod}}(p)$ and $\underline{\underline{\mathbf{T}}}$:
  \ $\frac{\partial V_{\mathrm{mod}}}{\partial p}=
\begin{cases}
\frac{\partial V}{\partial p}, & p\ge p_{\varepsilon},\\[6pt]
M, & p<p_{\varepsilon},
\end{cases}$
  \\ {\small(positivity + generic EOS)},\\[2pt]
  \ $\underline{\underline{\mathbf{T}}}$ includes the shock-activated dissipation
  $\tilde \kappa$ {\small($\chi>0$)}\\[5pt]
  $J=\dfrac{\partial V_{\mathrm{mod}}}{\partial p}+\underline{\underline{\mathbf{T}}}$
  \ $\to$\ solve $J\,\Delta\mathbf{p}=-\mathbf{g}$ (GMRES)
  \ $\to$\ $\mathbf{p}\leftarrow\mathbf{p}+\Delta\mathbf{p}$
  \ $\circlearrowleft$ until $|\Delta\mathbf{p}|<\delta_\mathrm{{Newton}}$
  \end{minipage}};

\draw[darrow] (s3.west) .. controls ($(s3.west)+(-5.0,0)$) and ($(nw.west)+(-5.0,0.2)$) .. (nw.west)
  node[pos=0.5,left,font=\small,align=center] {first\\solve};
\draw[darrow] (s5.east) .. controls ($(s5.east)+(5.2,0)$) and ($(nw.east)+(5.2,0.2)$) .. (nw.east)
  node[pos=0.5,right,font=\small,align=center] {second\\solve};

\end{tikzpicture}
}
\caption{Schematic illustration of the proposed semi-implicit IMEX scheme. Note
that the divergence-free constraint is enforced separately, after all sub-systems
have been solved, acting only on the magnetic field $\mathbf{B}^{n+1}$. Both
constrained transport variants (see Sect.~\ref{chp:divfree_CT}) are built on the
same flux-consistent electric field, obtained from the interface fluxes of the
implicit magnetic sub-system: the staggered variant evolves the face-centred
components through the discrete Faraday law, while the unstaggered variant
evolves the magnetic potential and recovers $\mathbf{B}$ as its discrete curl.}
\label{fig: Schematic illustration}
\end{figure}

\clearpage
\begin{equation}
    \label{equ:discreteofB}
    \begin{aligned}
        (\rho u) ^ {n+1} &= (\rho u) ^ {*} - \Delta t \frac{\partial}{\partial x}(p^n + \frac{\mathbf{B^n} \cdot \mathbf{B^{n + 1}}}{2} - B_x^n B_x^{n+1}) - \Delta t \frac{\partial}{\partial y} (-B_y ^nB_x^{n + 1}),\\
        (\rho v) ^ {n+1} &= (\rho v) ^ {*} - \Delta t \frac{\partial}{\partial x} (-B_x^n B_y^{n+1}) - \Delta t \frac{\partial}{\partial y}(p^n + \frac{\mathbf{B^n} \cdot \mathbf{B^{n + 1}}}{2} - B_y^n B_y^{n+1}), \\
        (\rho w) ^ {n+1} &= (\rho w) ^ {*} - \Delta t \frac{\partial}{\partial x} (-B_x^n B_z^{n+1}) - \Delta t \frac{\partial}{\partial y} (-B_y ^nB_z^{n + 1}),\\
        B^{n+1}_x &= B^{n}_x - \Delta t \frac{\partial}{\partial y} (\frac{(\rho v)^{n+1}}{\rho ^{n+1}} B_x^n - \frac{(\rho u)^{n+1}}{\rho ^{n+1}} B_y^n),\\
        B^{n+1}_y &= B^{n}_y - \Delta t \frac{\partial}{\partial x} (\frac{(\rho u)^{n+1}}{\rho ^{n+1}} B_y^n - \frac{(\rho v)^{n+1}}{\rho ^{n+1}} B_x^n), \\
       B^{n+1}_z &= B^{n}_z - \Delta t \frac{\partial}{\partial x}\!\left(
                \frac{(\rho u)^{n+1}}{\rho^{n+1}} B_z^n 
                - \frac{(\rho w)^{n+1}}{\rho^{n+1}} B_x^n
              \right) \\
        &\quad \quad \; \; - \Delta t \frac{\partial}{\partial y}\!\left(
                \frac{(\rho v)^{n+1}}{\rho^{n+1}} B_z^n 
                - \frac{(\rho w)^{n+1}}{\rho^{n+1}} B_y^n
              \right).
    \end{aligned}
\end{equation}

Note that the density at the new time level, $\rho^{n+1}$, is already known after the
explicit update of the advective sub-system. Following the approach proposed
in~\cite{BoscheriThomann2024}, the pressure gradient entering the momentum
contribution is evaluated explicitly at the old time level $n$. By formally
substituting the momentum evolution equations into the
magnetic field component equations of~\eqref{equ:magnective sub-system}, one obtains the
following sub-system for the unknowns $B_x^{n+1}$, $B_y^{n+1}$ and $B_z^{n+1}$:

\begin{equation}
    \label{equ:mag-sub}
    \left\{ 
    \begin{aligned}
    B_x^{n+1} &= B_x ^ * - \Delta t^2 \frac{\partial}{\partial y} (\frac{B_x^n}{\rho^{n+1}} [\frac{\partial}{\partial x} (B_x^n B_y^{n+1}) -  \frac{\partial}{\partial y} (\frac{\mathbf{B}^n \cdot \mathbf{B}^{n+1}}{2} - B_y^n B_y^{n+1}) ] \\
    & \quad \quad \quad \quad \quad\quad + \frac{B_y^n}{\rho^{n+1}} [\frac{\partial}{\partial x} (\frac{\mathbf{B}^n \cdot \mathbf{B}^{n+1}}{2} - B_x^n B_x^{n+1}) -  \frac{\partial}{\partial y} ( B_x^{n+1} B_y^{n}) ]),\\
    B_y^{n+1} &= B_y^* - \Delta t^2 \frac{\partial}{\partial x} (\frac{B_y^n}{\rho^{n+1}} [-\frac{\partial}{\partial x} (\frac{\mathbf{B}^n \cdot \mathbf{B}^{n+1}}{2} - B_x^n B_x^{n+1} ) +  \frac{\partial}{\partial y}(B_x^{n} B_y^{n+1})]\\
    & \quad \quad \quad \quad \quad\quad + \frac{B_x^n}{\rho^{n+1}} [- \frac{\partial}{\partial x} (B_x^{n} B_y^{n+1} ) +  \frac{\partial}{\partial y}(\frac{\mathbf{B}^n \cdot \mathbf{B}^{n+1}}{2} - B_y^n B_y^{n+1})]),\\
    B_z^{n+1} &= B_z^* -\Delta t^2 \frac{\partial}{\partial x}(\frac{B_z^n}{\rho^{n+1}}[- \frac{\partial}{\partial x}(\frac{\mathbf{B}^n \cdot \mathbf{B}^{n + 1}}{2} - B_x^n B_x^{n+1}) +  \frac{\partial}{\partial y}(B_x^{n+1} B^n_y)] \\
    & \quad \quad \quad \quad \quad\quad + \frac{B_x^n}{\rho^{n+1}}[- \frac{\partial}{\partial x}(- B_x^n B_z^{n+1}) -  \frac{\partial}{\partial y}(B_y^{n} B_z^{n+1})]) \\
    & \quad \quad \quad -\Delta t^2 \frac{\partial}{\partial y}(\frac{B_z^n}{\rho^{n+1}}[-\frac{\partial}{\partial x}(B_x^n B_y^{n+1}) - \frac{\partial}{\partial y}(\frac{\mathbf{B}^n \cdot \mathbf{B}^{n + 1}}{2} - B_y^n B_y^{n+1})] \\
    & \quad \quad \quad \quad \quad\quad + \frac{B_y^n}{\rho^{n+1}}[-\frac{\partial}{\partial x}(- B_x^n B_z^{n+1}) -  \frac{\partial}{\partial y}(B_y^{n} B_z^{n+1})]),
    \end{aligned}
    \right.
\end{equation}

where
\begin{equation}
    \left\{ 
    \begin{aligned}
    B_x^* &= B_x^n - \Delta t \frac{\partial}{\partial y} (\frac{B_x^n}{\rho ^{n+1}}[(\rho v) ^ * - \Delta t\frac{\partial}{\partial y} (p^n)] - \frac{B_y^n}{\rho^{n+1}}[(\rho u)^* - \Delta t \frac{\partial}{\partial x}(p^n)]),\\
    B_y^* &= B_y^n - \Delta t \frac{\partial}{\partial x} (\frac{B_y^n}{\rho ^{n+1}}[(\rho u) ^ * - \Delta t\frac{\partial}{\partial x} (p^n)] - \frac{B_x^n}{\rho^{n+1}}[(\rho v)^* - \Delta t \frac{\partial}{\partial y}(p^n)]), \\
    B_z^* &= B_z^n - \Delta t \frac{\partial}{\partial x} (\frac{B_z^n}{\rho ^{n+1}}[(\rho u) ^ * - \Delta t\frac{\partial}{\partial x} (p^n)] - \frac{B_x^n}{\rho^{n+1}}(\rho w)^*)\\
    & \quad \quad \quad -\Delta t \frac{\partial}{\partial y} (\frac{B_z^n}{\rho ^{n+1}}[(\rho v) ^ * - \Delta t\frac{\partial}{\partial y} (p^n)] - \frac{B_y^n}{\rho^{n+1}}(\rho w)^*).
    \end{aligned}
    \right.
\end{equation}

Due to the semi-implicit discretisation of the magnetic term $||\mathbf{B}||^2 \approx \mathbf{B}^n \cdot \mathbf{B}^{n+1}$ in the equation (\ref{equ:discreteofB}), the resulting system is linear, see \cite{BoscheriThomann2024} for further details.

After obtaining the magnetic field at time level $t = t^{n+1}$,
the magnetic sub-system~\eqref{equ:magnective sub-system} is updated by evaluating the flux
$F^{\mathbf{B}}(\rho^{n+1}, \rho \mathbf{u}^{*}, E^{*}, \mathbf{B}^{n+1})$.
This update provides the intermediate states of momentum $(\rho \mathbf{u})^{**}$ and energy $E^{**}$.

\subsubsection{Pressure sub-system}
\label{sec:pressure}

After the magnetic field update of Section~\ref{sec:magnetic_subsystem}, the magnetic field
$\mathbf{B}^{n+1}$ is known and the intermediate momentum $(\rho\mathbf{u})^{**}$
and energy $E^{**}$ have been obtained. This section derives the implicit pressure
solve and the subsequent momentum and energy updates in semi-discrete form; the
fully discrete counterparts are presented in Section~\ref{sec:fully_discrete}.

\paragraph{Momentum update:}
The final momentum is obtained by combining the implicit pressure gradient with the
intermediate momentum:
\begin{equation}
\label{equ:finalrhou}
(\rho\mathbf{u})^{n+1} = (\rho\mathbf{u})^{**}
- \Delta t\,\nabla\!\cdot\!\big(p^{n+1}\mathbf{I}\big).
\end{equation}

\paragraph{Derivation of the pressure wave equation:}
From the pressure sub-system~\eqref{equ:pressure sub-system}, the discrete energy
equation reads
\begin{equation}
\label{eq:energy_discrete}
\frac{E^{n+1}-E^{**}}{\Delta t}
+ \nabla\!\cdot\!\big(h^{n}(\rho\mathbf{u})^{n+1}\big) = 0,
\end{equation}
where the total energy at the new time level is
\begin{equation}
\label{eq:Udef}
E^{n+1} = (\rho e)^{n+1}
+ \frac{1}{2}\frac{|(\rho\mathbf{u})^{n+1}|^{2}}{\rho^{n+1}}
+ \frac{1}{2}\,|\mathbf{B}^{n+1}|^{2}.
\end{equation}
In contrast to the ideal-gas case, the internal energy is related to the pressure
through a general, possibly nonlinear, equation of state,
\begin{equation}
\label{eq:eos_V}
(\rho e)^{n+1} = \underline{\mathbf{V}}(p^{n+1}),
\end{equation}
which we retain in the generic form $\underline{\mathbf{V}}(p)$ throughout the
derivation. Since the magnetic field $\mathbf{B}^{n+1}$ is already known from
Section~\ref{sec:magnetic_subsystem}, we introduce the auxiliary quantity
\begin{equation}
\label{eq:pstar}
\mathcal{E}^{**} = E^{**} - \frac{1}{2}\,|\mathbf{B}^{n+1}|^{2},
\end{equation}
which absorbs the known magnetic energy. Substituting~\eqref{eq:Udef},
\eqref{eq:eos_V} and~\eqref{eq:pstar} into~\eqref{eq:energy_discrete} and
multiplying by $\Delta t$ yields
\begin{equation}
\label{eq:intermediate}
\underline{\mathbf{V}}(p^{n+1})
+ \frac{1}{2}\frac{|(\rho\mathbf{u})^{n+1}|^{2}}{\rho^{n+1}}
+ \Delta t\,\nabla\!\cdot\!\big(h^{n}(\rho\mathbf{u})^{n+1}\big)
= \mathcal{E}^{**}.
\end{equation}
Equation~\eqref{eq:intermediate} still couples the pressure to the new momentum
$(\rho\mathbf{u})^{n+1}$, both through the kinetic-energy term and inside the
divergence. The momentum update~\eqref{equ:finalrhou} provides the missing link,
and inserting it into the divergence term eliminates $(\rho\mathbf{u})^{n+1}$, producing an elliptic wave equation for the scalar pressure field:
\begin{equation}
\label{eq:wave}
\underline{\mathbf{V}}(p^{n+1})
- \Delta t^{2}\,\nabla\!\cdot\!\big(h^{n}\nabla p^{n+1}\big)
= \mathcal{E}^{**} - \frac{1}{2}\frac{|(\rho\mathbf{u})^{n+1}|^{2}}{\rho^{n+1}}
- \Delta t\,\nabla\!\cdot\!\big(h^{n}(\rho\mathbf{u})^{**}\big).
\end{equation}
The only difference with respect to the ideal-gas formulation is that the linear
term $p^{n+1}/(\gamma-1)$ is replaced by the EOS relation
$\underline{\mathbf{V}}(p^{n+1})$, which enters~\eqref{eq:wave} exclusively
through the diagonal; the elliptic operator on the left-hand side is unchanged.
Note that the kinetic-energy term on the right-hand side of~\eqref{eq:wave} is
itself a function of $p^{n+1}$ through~\eqref{equ:finalrhou}; this residual
coupling is resolved by the two-stage procedure described below.

\paragraph{Two-stage pressure solve:}
The right-hand side of~\eqref{eq:wave} involves the kinetic energy at the new time
level, which is not yet available since $(\rho\mathbf{u})^{n+1}$ depends on
$p^{n+1}$ through~\eqref{equ:finalrhou}. Following~\cite{FVVEM}, we adopt a
two-stage strategy in which the wave equation~\eqref{eq:wave} is solved twice.

\paragraph{First pressure solve:}
The wave equation is first solved using the intermediate momentum
$(\rho\mathbf{u})^{**}$ as a provisional value for the new kinetic energy:
\begin{equation}
\label{eq:pressurefirst}
\underline{\mathbf{V}}(\tilde{p}^{\,n+1})
- \Delta t^{2}\,\nabla\!\cdot\!\big(h^{n}\nabla\tilde{p}^{\,n+1}\big)
= \mathcal{E}^{**} - \frac{1}{2}\frac{|(\rho\mathbf{u})^{**}|^{2}}{\rho^{n+1}}
- \Delta t\,\nabla\!\cdot\!\big(h^{n}(\rho\mathbf{u})^{**}\big).
\end{equation}

\paragraph{Momentum update:}
The momentum is then advanced using $\tilde{p}^{\,n+1}$:
\begin{equation}
\label{eq:momentum2}
(\rho\mathbf{u})^{n+1} = (\rho\mathbf{u})^{**}
- \Delta t\,\nabla\tilde{p}^{\,n+1}.
\end{equation}

\paragraph{Second pressure solve:}
To guarantee thermodynamic consistency and conservation of the new energy,
$E^{n+1}$, the wave equation is solved once again for the pressure state,
$p^{n+1}$, with the kinetic energy now computed from the updated
momentum~\eqref{eq:momentum2}:
\begin{equation}
\label{eq:pressuresecond}
\underline{\mathbf{V}}(p^{n+1})
- \Delta t^{2}\,\nabla\!\cdot\!\big(h^{n}\nabla p^{n+1}\big)
= \mathcal{E}^{**} - \frac{1}{2}\frac{|(\rho\mathbf{u})^{n+1}|^{2}}{\rho^{n+1}}
- \Delta t\,\nabla\!\cdot\!\big(h^{n}(\rho\mathbf{u})^{**}\big).
\end{equation}
The new total energy, $E^{n+1}$, then follows directly from its
definition~\eqref{equ:defofU}.

\paragraph{Solution of the pressure system for a general equation of state:}
\label{sec:eos_solve}
Both pressure solves require the solution of a system of the compact form
\begin{equation}
\label{eq:compact}
\underline{\mathbf{V}}(\mathbf{p}^{n+1})
+ \underline{\underline{\mathbf{T}}}\,\mathbf{p}^{n+1} = \mathbf{b},
\end{equation}
where $\mathbf{p}^{n+1}$ is the vector of unknown cell-centred pressures,
$\underline{\underline{\mathbf{T}}}$ is the matrix arising from the
discrete elliptic operator, and
$\mathbf{b}$ collects the known right-hand side. The whole dependence on the
equation of state is carried by the diagonal term
$\underline{\mathbf{V}}(\mathbf{p}^{n+1})$, whereas the coupling matrix
$\underline{\underline{\mathbf{T}}}$ is EOS-independent.

\paragraph{Linear equations of state.}
For a linear EOS, such as the ideal gas law, the internal energy is a linear
function of the pressure, $\underline{\mathbf{V}}(p)=p/(\gamma-1)$, so that
$\partial\underline{\mathbf{V}}/\partial p = \mathbf{I}/(\gamma-1)$ is constant.
System~\eqref{eq:compact} is then linear and is solved directly with a single
linear solve.

\paragraph{General nonlinear equations of state.}
For a general EOS, such as the Redlich--Kwong EOS
models, the relation $\underline{\mathbf{V}}(p)$ is nonlinear and~\eqref{eq:compact}
becomes the nonlinear system
\begin{equation}
\mathbf{g}(\mathbf{p}^{n+1})
= \underline{\mathbf{V}}(\mathbf{p}^{n+1})
+ \underline{\underline{\mathbf{T}}}\,\mathbf{p}^{n+1} - \mathbf{b} = \mathbf{0},
\end{equation}
which we solve by a Newton method following Boscheri and Pareschi~\cite{BoscheriPareschi2021}.
Denoting by $s$ the Newton iteration index, the pressure correction
$\Delta\mathbf{p}$ solves
\begin{equation}
\label{eq:newton_update}
\left(\frac{\partial\underline{\mathbf{V}}}{\partial\mathbf{p}}
\bigg|_{\mathbf{p}^{n+1,s}} + \underline{\underline{\mathbf{T}}}\right)
\Delta\mathbf{p}
= \mathbf{b} - \underline{\mathbf{V}}(\mathbf{p}^{n+1,s})
- \underline{\underline{\mathbf{T}}}\,\mathbf{p}^{n+1,s},
\qquad
\mathbf{p}^{n+1,s+1} = \mathbf{p}^{n+1,s} + \Delta\mathbf{p},
\end{equation}
and the iteration is repeated until $|\Delta\mathbf{p}|<\delta_\mathrm{Newton}$, with
$\delta_\mathrm{Newton}=10^{-10}$. The nonlinearity enters only through the diagonal term
$\partial\underline{\mathbf{V}}/\partial p$, which depends on the already-known
density $\rho^{n+1}$ and on the local EOS, while the off-diagonal coupling
$\underline{\underline{\mathbf{T}}}$ is left unchanged. The system is therefore
only mildly nonlinear and the Newton iteration converges in a few steps; the
linear case is recovered as the special situation in which the Jacobian is
constant and a single iteration is exact.

\subsection{Further stabilisation for high-acoustic-Mach and high-Alfvén-Mach number flows}
\label{sec:flattener_mhd}

The central-difference discretisation of the pressure wave equations cannot fully 
suppress over- and undershoots at strong discontinuities, producing localised 
pressure oscillations. Following the a-priori flattener strategy of Boscheri and 
Pareschi~\cite{BoscheriPareschi2021}, we add extra dissipation only in cells where a 
shock is detected, through the cell-centred indicator 
$\chi^{n}\in[0,1]$~\cite{Balsara2012}.

Since the scheme is asymptotic-preserving in both the low-acoustic-Mach and low-Alfvén-Mach 
regimes, the indicator must stay inactive whenever the flow is stiffened either 
acoustically or magnetically.  The flattener reads \cite{Balsara2012}
\begin{equation}
\label{eq:flattener_mhd}
\tilde{\chi}^{\,n} 
= \min\!\left[\,1,\ \max\!\left(0,\ 
-\,\frac{\delta\,(\nabla\cdot\mathbf{u}^{n}) + k_1\,c_{f,\min}^{\,n}}
        {k_1\,c_{f,\min}^{\,n}}\right)\right],
\qquad 
\chi^{n} = \max_{\mathcal{N}} \tilde{\chi}^{\,n},
\end{equation}
with $\delta = \min(\Delta x,\Delta y)$, $k_1 = 10^{-3}$, $c_{f,\min}^{\,n}$ the 
minimum magnetosonic speed over the $3\times3$ stencil $\mathcal{N}$. 
Using $c_f$ rather than $a$ triggers the indicator only on a super-fast-magnetosonic 
compression: a flow stiff in either the acoustic or the magnetic sense yields a 
large $c_f$ and hence $\chi^{n}\to 0$, so no spurious dissipation is introduced in 
either asymptotic limit.

In contrast to~\cite{BoscheriPareschi2021}, which modifies the energy flux explicitly, 
we augment the pressure wave equation with a shock-activated Laplacian treated 
implicitly. Because this equation is assembled from the internal-energy relation 
$\underline{\mathbf{V}}(p)=\rho\,e(\rho,p)$ while the unknown is $p$, the 
dissipation is a regularisation of the internal energy and must act on 
$\underline{\mathbf{V}}$ rather than on $p$ directly. With the thermodynamic 
derivative
\begin{equation}
\label{eq:Vprime}
\underline{\mathbf{V}}'(p)=\frac{\partial(\rho e)}{\partial p}\Big|_{\rho},
\qquad
\underline{\mathbf{V}}'(p)=\frac{1}{\gamma-1}\ \text{(ideal gas),}
\end{equation}
and the chain rule $\nabla\underline{\mathbf{V}}(p)=\underline{\mathbf{V}}'(p)\,
\nabla p$, the augmented system reads
\begin{equation}
\label{equ:kappa_V}
\underline{\mathbf{V}}(p^{n+1})
 - \Delta t^{2}\,\nabla\!\cdot\!\bigl(h^{n}\,\nabla p^{n+1}\bigr)
 - \Delta t\,\nabla\!\cdot\!\bigl(\tilde{\kappa}^{\,n}\,\nabla p^{n+1}\bigr)
 = \mathbf{b},
\qquad
\tilde{\kappa}^{\,x,n} = \underline{\mathbf{V}}'(p^{n})\,\chi^{n}|u^{n}|\Delta x,
\quad
\tilde{\kappa}^{\,y,n} = \underline{\mathbf{V}}'(p^{n})\,\chi^{n}|v^{n}|\Delta y,
\end{equation}
so that the dissipative flux $\tilde{\kappa}^{\,n}\nabla p^{n+1}$ is a consistent
discretisation of $\kappa^{n}\nabla\underline{\mathbf{V}}(p^{n+1})$, with
$\kappa^{x,n}=\chi^{n}|u^{n}|\Delta x$, $\kappa^{y,n}=\chi^{n}|v^{n}|\Delta y$:
the chain rule $\nabla\underline{\mathbf{V}}=\underline{\mathbf{V}}'\nabla p$ is applied inside
the flux, with the factor $\underline{\mathbf{V}}'(p^{n})$ frozen at the explicit state,
consistently with $h^{n}$ and $\chi^{n}$. Keeping $\underline{\mathbf{V}}'$ inside the
divergence preserves the strict flux form of the dissipation, so that the added
term enters the coupling matrix $\underline{\underline{\mathbf{T}}}$ as a conservative
contribution (Lemma~\ref{lem:pos_T}); for an ideal gas $\underline{\mathbf{V}}'=1/(\gamma-1)$ is constant
and the weighting reduces to a constant rescaling of $\kappa^{n}$.

The added term modifies only the linear coupling matrix 
$\underline{\underline{\mathbf{T}}}$ and is solved together with the nested Newton 
iteration of Section~\ref{sec:positivity}. The convective scaling of 
$\kappa^{n}$ and the magnetosonic normalisation of $\chi^{n}$ keep the scheme 
asymptotic-preserving in both limits, with no time step restriction beyond the 
convective CFL condition.

\begin{figure}
\centering
\includegraphics[width=\linewidth]{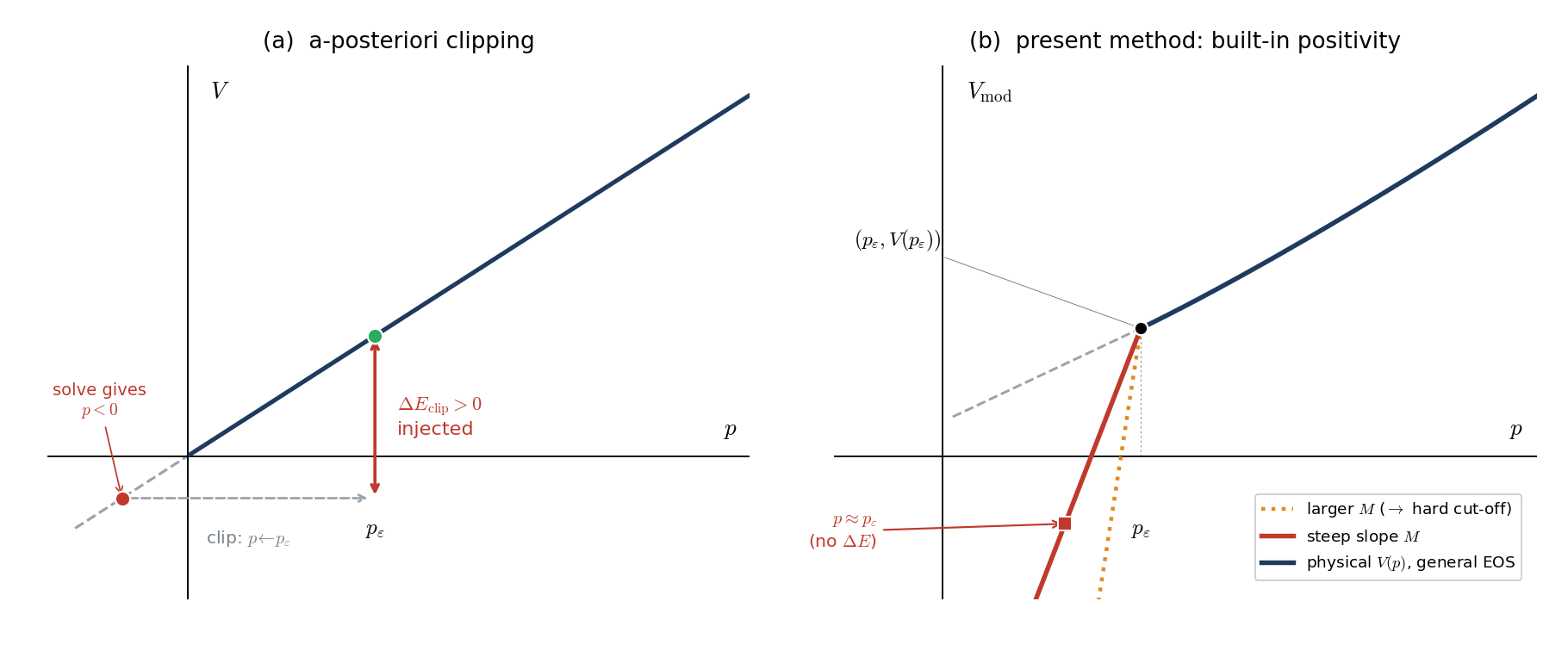}
\caption{Comparison of two strategies for enforcing pressure positivity.
  (a) A-posteriori clipping resets a negative pressure to the floor
  $p_\varepsilon$, injecting a spurious energy $\Delta E_{\mathrm{clip}}>0$ and
  breaking conservation. (b) The present method builds positivity into the
  pressure solve through the steep-slope modification $V_{\mathrm{mod}}$, which
  confines the pressure to $p\approx p_\varepsilon$ without injecting energy.}
\label{fig:positiveper}
\end{figure}

\subsection{Conservative enforcement of pressure positivity}
\label{sec:positivity}

The nested Newton procedure introduced above for a nonlinear equation of state can be
further exploited to enforce a positivity-preserving safeguard on the pressure, without
modifying the structure of the solver. The compact system~\eqref{eq:compact} does not,
by itself, prevent the discrete pressure from becoming negative: in strongly rarefied or
strongly magnetised regions the right-hand side may drive the local solution below zero,
since the internal-energy relation $\underline{\mathbf{V}}(p)$ offers no barrier at
$p=0$. A natural remedy would be to clip the offending cell,
$p\leftarrow\max(p,p_{\varepsilon})$, but this acts on the solution after the
system has been solved and therefore injects energy into that cell without any
compensating flux, breaking conservation. We instead build the floor into the
system, so that a positive solution is obtained as part of the same nested Newton solve.
The key observation is that both the nonlinearity of the equation of state and the
pressure positivity constraint act exclusively on the diagonal term
$\underline{\mathbf{V}}(\tilde{\mathbf{p}})$ in~\eqref{eq:compact}, leaving the flux
coupling matrix $\underline{\underline{\mathbf{T}}}$ untouched.

We recall that in \eqref{eq:newton_update} the Newton
iteration only requires the internal-energy relation
$\underline{\mathbf{V}}(\tilde{\mathbf{p}})$ and its derivative
$\partial\underline{\mathbf{V}}/\partial\tilde{\mathbf{p}}$. For a general equation of
state this derivative is itself a function of the pressure and may become small, or even
degenerate, in low-pressure or thermodynamically unstable regions, which is precisely
where negative pressures are prone to appear. To promote pressure positivity, we therefore bound
the diagonal slope from below by a large prescribed threshold $M>0$,

\begin{equation}
\label{eq:Vmod_slope}
\frac{\partial V_{\mathrm{mod}}}{\partial p}=
\begin{cases}
\frac{\partial V}{\partial p}, & p\ge p_{\varepsilon},\\[6pt]
M, & p<p_{\varepsilon},
\end{cases}
\end{equation}
and recover the modified relation $V_{\mathrm{mod}}(p)$ by integration, so that it
remains continuous. The construction applies to any equation of state: below the floor $p_{\varepsilon}$ the physical slope is replaced by the steep slope $M$, while for $p \geq p_\varepsilon$
the relation is left unchanged. For the ideal gas law, where $\partial V/\partial p=1/(\gamma-1)$ is constant, this reduces to the piecewise-linear
form
\begin{equation}
\label{eq:Vmod}
V_{\mathrm{mod}}(p)=
\begin{cases}
V(p), & p\ge p_{\varepsilon},\\[6pt]
V(p_{\varepsilon})+M\,(p-p_{\varepsilon}), & p<p_{\varepsilon},
\end{cases}
\end{equation}
which switches to the steep slope below a pressure floor $p_{\varepsilon}>0$, the two
branches being joined continuously through the offset $V(p_{\varepsilon})$. This
construction is inspired by the wetting-and-drying treatment of
Casulli~\cite{Casulli2009}, where a similar piecewise relation keeps the water
depth non-negative.

The effect of the steep slope is best understood by solving the local relation for a
cell that the right-hand side would otherwise push below the pressure floor. Isolating cell $i$
and collecting the neighbour coupling into
$\beta_i:=b_i-\sum_{j\neq i}T_{ij}\,\tilde{p}_j$, the modified relation
$V_{\mathrm{mod}}(p_i)+T_{ii}p_i=\beta_i$ gives, to leading order in the steep slope,
\begin{equation}
\label{eq:penalty}
p_{i}\ \approx\ p_{\varepsilon}+\frac{\beta_i-V(p_{\varepsilon})}{M},
\end{equation}
so that the correction term vanishes as $M\to\infty$ and $p_i\to p_{\varepsilon}$.
In other words, for a sufficiently steep slope the modification acts \emph{as if the
pressure had been reset to the floor value} $p_{\varepsilon}$, yet it does so by
selecting the value that the system itself returns under the modified relation, rather
than by overwriting the result afterwards. We emphasise that both the steep slope
\emph{and} the continuity of $V_{\mathrm{mod}}$ are essential: a steep slope through the
origin alone would give $p_i\approx\beta_i/M$, which still admits negative values; it is
the anchoring of $V_{\mathrm{mod}}$ at the positive floor, enforced by the integration
constant, that confines the solution to the admissible range.

Since the modified slope ~\eqref{eq:Vmod_slope} is strictly positive,
$\underline{\mathbf V}_{\mathrm{mod}}$ is monotonically increasing and the
modified system~\eqref{eq:compact} remains uniquely solvable (Proposition~\ref{prop:pos_wellposed}). The nested
Newton iteration~\eqref{eq:newton_update} remains well defined, the only change being that the
diagonal derivative $\partial\underline{\mathbf V}/\partial\tilde{\mathbf p}$
in the Jacobian is replaced by
$\partial\underline{\mathbf V}_{\mathrm{mod}}/\partial\tilde{\mathbf p}$: for
a piecewise-linear relation, such as the ideal gas law together with the positivity modification, it terminates in finitely many steps, while for a
general nonlinear equation of state it converges locally to the unique
solution (Proposition~\ref{prop:pos_newton}). The lower bound $M$ additionally guarantees that
the diagonal of the Newton system stays well conditioned even when the
underlying thermodynamic model yields a vanishing slope, so that the same
procedure applies uniformly across ideal and general nonlinear equations of
state.

Crucially, the modification acts only on the diagonal term $\underline{\mathbf{V}}$, a
purely cell-local algebraic relation, and leaves the inter-cell flux coupling
$\underline{\underline{\mathbf{T}}}$, which is responsible for the telescoping conservation of
the total energy, untouched. The pressure positivity is therefore obtained as the genuine
solution of a modified but still flux-consistent system, in direct contrast to the
a-posteriori clipping $p\leftarrow\max(p,p_{\varepsilon})$, which overwrites a single
cell without a compensating flux and injects spurious energy. The heuristic
estimate~\eqref{eq:penalty}, the unique solvability of the modified system, the global
control of the discrete pressure, and the exact preservation of energy conservation are
made rigorous in Section~\ref{sec:properties}.

\subsection{Properties of the positivity-preserving pressure system}
\label{sec:properties}

We now establish the properties of the modified pressure system
\begin{equation}
\label{eq:pos_sys}
\underline{\mathbf{V}}_{\mathrm{mod}}(\mathbf{p})
+\underline{\underline{\mathbf{T}}}\,\mathbf{p}=\mathbf{b},
\qquad \mathbf{p}\in\mathbb{R}^{N},\quad N=n_x n_y,
\end{equation}
where $\underline{\underline{\mathbf{T}}}$ is the five-point discretisation of the implicit
elliptic operator
$-\Delta t^{2}\,\nabla\!\cdot(h^{n}\nabla\,\cdot\,)
 -\Delta t\,\nabla\!\cdot(\tilde{\kappa}^{\,n}\nabla\,\cdot\,)$,
with face-averaged coefficients
$h^{n}_{i+\frac12,j}=\tfrac12\bigl(h^{n}_{ij}+h^{n}_{i+1,j}\bigr)$ and
$\tilde{\kappa}^{\,x,n}_{i+\frac12,j}
 =\tfrac12\bigl(\tilde{\kappa}^{\,x,n}_{ij}+\tilde{\kappa}^{\,x,n}_{i+1,j}\bigr)$
(and analogously in $y$). Throughout this section we assume $h^{n}_{f}>0$ at
every face $f$ (this holds for the ideal gas and for all states encountered
in the tests of Section~\ref{sec:results}) while $\tilde{\kappa}^{\,n}_{f}\ge 0$ holds by
construction, since $\chi^{n}\ge 0$ and $\underline{\mathbf{V}}'>0$.
$\underline{\mathbf{V}}_{\mathrm{mod}}$ is defined through the modified
slope~\eqref{eq:Vmod_slope}: we write
$\alpha(p)=\partial\underline{\mathbf{V}}/\partial p>0$ for the physical
equation-of-state slope and assume it to be uniformly positive on the
physical branch,
\begin{equation}
\label{equ:alpha0}
\alpha(p)\;\ge\;\alpha_0\;>\;0
\qquad\text{for } p\ge p_{\varepsilon},
\end{equation}
which holds for the ideal gas, where $\alpha=1/(\gamma-1)$, and for the
Redlich--Kwong closure on all states encountered in Section~\ref{sec:results}. Then
$\partial\underline{\mathbf{V}}_{\mathrm{mod}}/\partial p=M$ on the floored branch
$p<p_{\varepsilon}$ and $\alpha(p)$ elsewhere, so that
$\partial\underline{\mathbf{V}}_{\mathrm{mod}}/\partial p\ \ge\
\underline{\alpha}:=\min(M,\alpha_0)>0$ wherever defined, and we recall that
$\underline{\mathbf{V}}_{\mathrm{mod}}$ is continuous with
$\underline{\mathbf{V}}_{\mathrm{mod}}(p_{\varepsilon})
=\underline{\mathbf{V}}(p_{\varepsilon})$. For a general equation of state $\underline{\mathbf{V}}$ acts cell-wise,
$V_{i}(p)=\rho_{i}\,e(\rho_{i},p)$; in the bounds of
Theorem~\ref{thm:pos_global} and Corollary~\ref{cor:pos_strict} the
quantity $V(p_{\varepsilon})$ is understood as
$\max_{i}V_{i}(p_{\varepsilon})$, a constant for the ideal gas.

\begin{lemma}[Structure of $\underline{\underline{\mathbf{T}}}$]
\label{lem:pos_T}
For $h^{n}_{f}>0$, $\tilde{\kappa}^{\,n}_{f}\ge 0$ and impermeable (periodic,
reflective, or closed) boundaries, the matrix $\underline{\underline{\mathbf{T}}}$ is symmetric, has
positive diagonal and non-positive off-diagonal entries, has zero row and
column sums, and is positive semi-definite, with
\begin{equation}
\label{eq:pos_quadform}
\mathbf{q}^{\top}\underline{\underline{\mathbf{T}}}\,\mathbf{q}
 =\sum_{\text{faces }f} w_{f}\,\bigl(q_{R_{f}}-q_{L_{f}}\bigr)^{2}\;\ge\;0,
\qquad
w_{f}:=\frac{\Delta t^{2}h^{n}_{f}+\Delta t\,\tilde{\kappa}^{\,n}_{f}}
            {\Delta x_{f}^{2}},
\qquad\forall\,\mathbf{q}\in\mathbb{R}^{N}.
\end{equation}
\end{lemma}
\begin{proof}
Both dissipative contributions are assembled face-wise, and the two cells
$L_{f}$, $R_{f}$ adjacent to a face $f$ share the single value $w_{f}$: hence
$T_{L_{f}R_{f}}=T_{R_{f}L_{f}}=-w_{f}$ and $\mathbf{T}$ is symmetric. Each
interior face contributes $+w_{f}$ to the diagonal of its two adjacent cells
and $-w_{f}$ to their coupling, which gives $T_{ii}>0$, $T_{ij}\le 0$
($i\ne j$), and makes the contributions telescope so that each row and column
sums to zero. Grouping the quadratic form by faces yields \eqref{eq:pos_quadform}; since
$w_{f}\ge 0$ with the $h$-part strictly positive, $\mathbf{T}\succeq 0$.
\end{proof}

\begin{proposition}[Well-posedness]
\label{prop:pos_wellposed}
For every $\mathbf{b}\in\mathbb{R}^{N}$, the nonlinear system~\eqref{eq:pos_sys} admits a
unique solution.
\end{proposition}
\begin{proof}
Define
\begin{equation*}
\Phi(\mathbf{p})=\sum_i \widehat V_{\mathrm{mod}}(p_i)
+\tfrac12\,\mathbf{p}^{\top}\underline{\underline{\mathbf{T}}}\,\mathbf{p}
-\mathbf{b}^{\top}\mathbf{p},
\end{equation*}
where $\widehat V_{\mathrm{mod}}$ is a primitive of $V_{\mathrm{mod}}$. Its
gradient is
$\nabla\Phi(\mathbf{p})=\underline{\mathbf{V}}_{\mathrm{mod}}(\mathbf{p})
+\underline{\underline{\mathbf{T}}}\,\mathbf{p}-\mathbf{b}$, so the roots of
$\mathbf{G}(\mathbf{p}):=\underline{\mathbf{V}}_{\mathrm{mod}}(\mathbf{p})
+\underline{\underline{\mathbf{T}}}\,\mathbf{p}=\mathbf{b}$ are exactly the
critical points of $\Phi$. Since $V_{\mathrm{mod}}$ is continuous and
piecewise $C^{1}$ with
$\partial V_{\mathrm{mod}}/\partial p\ge\underline{\alpha}=\min(M,\alpha_0)>0$
by \eqref{eq:Vmod_slope} and \eqref{equ:alpha0}, each
$\widehat V_{\mathrm{mod}}$ is $\underline{\alpha}$-strongly convex; by
Lemma~\ref{lem:pos_T} the quadratic term is convex. Hence $\Phi$ is
$\underline{\alpha}$-strongly convex, and therefore coercive, on
$\mathbb{R}^{N}$; it admits a unique minimiser, which is the unique root of
$\mathbf{G}$.
\end{proof}

\begin{theorem}[Global discrete positivity bound]
\label{thm:pos_global}
Let $\mathbf{p}$ solve~\eqref{eq:pos_sys} and set $b_{\min}:=\min_i b_i$. Then the
discrete pressure obeys the global lower bound
\begin{equation}
\label{eq:pos_global_bound}
\min_i p_i\ \ge\ p_{\varepsilon}
-\frac{\big|\,b_{\min}-V(p_{\varepsilon})\,\big|}{M},
\end{equation}
\emph{independently of the number and spatial distribution of the floored cells}. In
particular $\min_i p_i\to p_{\varepsilon}$ as $M\to\infty$.
\end{theorem}
\begin{proof}
Let $i^{\star}=\arg\min_i p_i$ and $p_{\min}=p_{i^{\star}}$. The $i^{\star}$-th equation
of~\eqref{eq:pos_sys} reads
$V_{\mathrm{mod}}(p_{\min})+\sum_j T_{i^{\star}j}\,p_j=b_{i^{\star}}$. Since
$\underline{\underline{\mathbf{T}}}$ has zero row sums (Lemma~\ref{lem:pos_T}),
subtracting $p_{\min}\sum_j T_{i^{\star}j}=0$ gives
$\sum_j T_{i^{\star}j}\,p_j=\sum_{j\neq i^{\star}}T_{i^{\star}j}(p_j-p_{\min})$. For
$j\neq i^{\star}$ one has $T_{i^{\star}j}\le 0$ and $p_j-p_{\min}\ge 0$, so every term is
non-positive and $\sum_j T_{i^{\star}j}\,p_j\le 0$. Therefore
\begin{equation}
V_{\mathrm{mod}}(p_{\min})=b_{i^{\star}}-\sum_j T_{i^{\star}j}\,p_j
\ \ge\ b_{i^{\star}}\ \ge\ b_{\min}.
\end{equation}
If $p_{\min}\ge p_{\varepsilon}$ the bound is trivial. Otherwise $i^{\star}$ lies on the
steep branch, where, by continuity of $V_{\mathrm{mod}}$ at $p_{\varepsilon}$,
$V_{\mathrm{mod}}(p_{\min})=M(p_{\min}-p_{\varepsilon})+V(p_{\varepsilon})$, so
\begin{equation}
M(p_{\min}-p_{\varepsilon})+V(p_{\varepsilon})\ \ge\ b_{\min}
\quad\Longrightarrow\quad
p_{\min}\ \ge\ p_{\varepsilon}+\frac{b_{\min}-V(p_{\varepsilon})}{M}
\ \ge\ p_{\varepsilon}-\frac{|\,b_{\min}-V(p_{\varepsilon})\,|}{M}.
\end{equation}
The right-hand side depends on the data only through $b_{\min}$ and on the fixed
constants $V(p_{\varepsilon}),p_{\varepsilon},M$; it is independent of which cells, or
how many, are floored. Letting $M\to\infty$ gives $p_{\min}\to p_{\varepsilon}$.
\end{proof}

\begin{corollary}[Strict positivity by choice of $M$]
\label{cor:pos_strict}
Given a target floor $\theta p_{\varepsilon}$ with $\theta\in(0,1)$, the choice
\begin{equation}
\label{eq:pos_Mchoice}
M\ \ge\ \frac{\big(V(p_{\varepsilon})-b_{\min}\big)_{+}}{(1-\theta)\,p_{\varepsilon}}
\end{equation}
guarantees $\min_i p_i\ge\theta p_{\varepsilon}>0$. Since $b_{\min}$ is available once
the right-hand side of~\eqref{eq:pos_sys} has been assembled, such an $M$ can be
determined at $O(N)$ cost prior to the solve.
\end{corollary}

\begin{remark}
Theorem~\ref{thm:pos_global} is the structural reason why the present enforcement
prevents a negative pressure from forming and propagating: the lower bound is governed
solely by the global minimum $b_{\min}$ of the right-hand side and by the steep slope
$M$, \emph{uniformly} in the configuration of the floored set. Evaluating the discrete
balance at the minimising cell makes the off-diagonal (diffusive) contribution
sign-definite, so that no dependence on neighbouring pressures remains. Equivalently,
$\mathbf{J}=\mathbf{D}+\underline{\underline{\mathbf{T}}}$ is an M-matrix
($\mathbf{J}^{-1}\ge 0$ entrywise), and~\eqref{eq:pos_global_bound} is a discrete minimum
principle. This makes the heuristic estimate~\eqref{eq:penalty} rigorous and removes its
apparent dependence on the neighbouring pressures.
\end{remark}

\begin{theorem}[Exact discrete energy conservation]
\label{thm:conservation}
Assume periodic boundaries, or impermeable walls on which
$\mathbf{u}\cdot\mathbf{n}=0$ and $\mathbf{B}\cdot\mathbf{n}=0$ hold at every
boundary face, and consider the sub-system updates prior to the
divergence-free correction of Section~\ref{chp:divfree_CT}, so that the magnetic
field is the predictor $\tilde{\mathbf{B}}^{n+1}$. Let the total energy be defined
cell-wise from the converged pressure as
\begin{equation}
\label{eq:Udef_VP}
E^{n+1}_{ij}
 = \underline{\mathbf{V}}_{\mathrm{mod}}\bigl(p^{n+1}_{ij}\bigr)
 + \frac{\bigl|(\rho\mathbf{u})^{n+1}_{ij}\bigr|^{2}}{2\rho^{n+1}_{ij}}
 + \tfrac12\,\bigl\|\tilde{\mathbf{B}}^{n+1}_{ij}\bigr\|^{2},
\end{equation}
where $p^{n+1}$ solves the discrete pressure system \eqref{eq:pos_sys}. Then the total
discrete energy is globally conserved,
\begin{equation}
\label{eq:energy_conservation}
\sum_{ij} E^{n+1}_{ij}\,\Delta x\,\Delta y \;=\; \sum_{ij} E^{n}_{ij}\,\Delta x\,\Delta y.
\end{equation}
\end{theorem}
\begin{proof}
By construction, the discrete pressure system is the cell-wise total-energy balance in
which the internal energy is represented by $V_{\mathrm{mod}}(p^{n+1})$; at convergence of
the Newton iteration the cell-wise definition of $E^{n+1}$ therefore coincides with the
conservative flux-form update
\begin{equation}
\label{eq:fluxform}
E^{n+1}_{ij}=E^{n}_{ij}
-\frac{\Delta t}{\Delta x}\big(\mathcal{F}_{i+\frac12,j}-\mathcal{F}_{i-\frac12,j}\big)
-\frac{\Delta t}{\Delta y}\big(\mathcal{G}_{i,j+\frac12}-\mathcal{G}_{i,j-\frac12}\big),
\end{equation}
whose numerical energy fluxes, $\mathcal{F}$ and $\mathcal{G}$, collect the advective, magnetic
and pressure contributions. These fluxes are \emph{single-valued} at each interface: the
flux leaving cell $(i,j)$ through its right face equals the flux entering cell
$(i+1,j)$ through its left face,
$\mathcal{F}^{\,\mathrm{right}}_{i+\frac12,j}=\mathcal{F}^{\,\mathrm{left}}_{i+\frac12,j}$,
and analogously in $y$. Summing~\eqref{eq:fluxform} over all control volumes, every
interior interface flux therefore appears twice with opposite signs and cancels
(telescoping), so that
\begin{equation}
\sum_{ij}\big(E^{n+1}_{ij}-E^{n}_{ij}\big)\,\Delta x\,\Delta y
=-\Delta t\!\!\sum_{\mathrm{boundary\ faces}}\!\!\mathcal{F}\cdot\mathbf{n}\,\ell_f .
\end{equation}
The advective and pressure energy fluxes are proportional to
$\mathbf{u}\cdot\mathbf{n}$ and vanish since $\mathbf{u}\cdot\mathbf{n}=0$ on
the walls; the magnetic energy flux consists of the terms
$(\mathbf{u}\cdot\mathbf{n})\|\mathbf{\tilde{B}}\|^{2}$ and
$(\mathbf{u}\cdot\mathbf{\tilde{B}})(\mathbf{\tilde{B}}\cdot\mathbf{n})$, which vanish since
$\mathbf{u}\cdot\mathbf{n}=0$ and $\mathbf{\tilde{B}}\cdot\mathbf{n}=0$, respectively;
for periodic boundaries there are no boundary faces. Hence the boundary sum is
zero and \eqref{eq:energy_conservation} follows.
\end{proof}

\begin{remark}
Conservation rests on the single-valuedness of the numerical energy fluxes shared between
adjacent cells and on the convergence of the pressure solve, at which point the cell-wise
energy definition~\eqref{eq:Udef_VP} coincides with the flux-form
update~\eqref{eq:fluxform}. The single-valuedness of the interface fluxes shared between adjacent cells, encoded
in $\underline{\underline{\mathbf{T}}}$ through the shared face weight $w_f$
(Lemma~\ref{lem:pos_T}), is the algebraic manifestation of this property. This is the
discrete analogue of the equivalence established in~\cite{FVVEM} for the
pressure-based semi-implicit formulation.
\end{remark}

\begin{remark}
On floored cells, i.e.\ where $p^{n+1}_{i}<p_{\varepsilon}$ at convergence, the
stored internal energy is $\underline{\mathbf{V}}_{\mathrm{mod}}(p^{n+1}_{i})$ rather than
the equation-of-state value $\underline{\mathbf{V}}(p^{n+1}_{i})$: the pressure returned by
the solve acts as a floored surrogate that exceeds the value the equation of
state would associate with the stored internal energy. This is the exact
counterpart of Proposition~\ref{prop:pos_clip}: a-posteriori clipping keeps the equation-of-state
relation and creates the spurious energy $\Delta E_{\mathrm{clip}}$, whereas
the present construction conserves the total energy exactly and relaxes the
equation-of-state relation on the floored cells only. Away from the floored set
$\underline{\mathbf{V}}_{\mathrm{mod}}\equiv\underline{\mathbf{V}}$, so the two definitions coincide,
and thermodynamic consistency is restored automatically as soon as the
discrete balance lifts a cell above $p_{\varepsilon}$.
\end{remark}

\begin{proposition}[Contrast with a-posteriori clipping]
\label{prop:pos_clip}
The a-posteriori clipping $p_i\leftarrow\max(p_i,p_{\varepsilon})$ injects a spurious
positive energy
$\Delta E_{\mathrm{clip}}=\sum_{i\in\mathcal{C}}\big(V(p_{\varepsilon})-V(p_i^{\mathrm{old}})\big)>0$
over the set $\mathcal{C}$ of clipped cells, whereas the present construction yields
$\Delta E=0$ by Theorem~\ref{thm:conservation}.
\end{proposition}
\begin{proof}
Since $V$ is strictly increasing ($\alpha=\partial V/\partial p>0$) and $p_i^{\mathrm{old}}
<p_{\varepsilon}$ on $\mathcal{C}$, clipping overwrites the internal energy in cell $i$
from $V(p_i^{\mathrm{old}})$ to $V(p_{\varepsilon})>V(p_i^{\mathrm{old}})$ without a
compensating flux, adding the stated positive amount to the total energy. The present
method instead obtains $p_i$ as the genuine solution of~\eqref{eq:pos_sys} and advances
the energy by its definition~\eqref{eq:Udef_VP}, so that no energy is created.
\end{proof}

\begin{proposition}[Solvability and convergence of the nested Newton iteration]
\label{prop:pos_newton}
Let $\mathbf{D}(\mathbf{p})
=\operatorname{diag}\bigl(\partial V_{\mathrm{mod}}/\partial p\,(p_i)\bigr)$
be evaluated branch-wise according to \eqref{eq:Vmod_slope}. Then every such
Jacobian $\mathbf{J}=\mathbf{D}(\mathbf{p}^{s})+\underline{\underline{\mathbf{T}}}
\succeq\underline{\alpha}\,\mathbf{I}\succ0$ is invertible, so the update
$\mathbf{J}\,\Delta\mathbf{p}=-(\mathbf{G}(\mathbf{p}^{s})-\mathbf{b})$ is well
defined at every iteration. When $V_{\mathrm{mod}}$ is piecewise linear (a linear
EOS together with the pressure positivity modification), the iteration is a semi-smooth
Newton method for a piecewise-linear system and terminates in finitely many
steps~\cite{BrugnanoCasulli2009}; if all cells share one branch of
$V_{\mathrm{mod}}$, $\mathbf{G}$ is affine and a single step is exact, in
particular, when no cell is floored this recovers the single linear solve of
the unmodified scheme. For a general nonlinear EOS the iteration converges
locally to the unique solution of Proposition~\ref{prop:pos_wellposed}.
\end{proposition}
\begin{proof}
Both one-sided slopes of $V_{\mathrm{mod}}$ at $p_{\varepsilon}$, and its
derivative elsewhere, are bounded below by
$\underline{\alpha}=\min(M,\alpha_0)>0$ by \eqref{eq:Vmod_slope} and
\eqref{equ:alpha0}; hence $\mathbf{D}\succeq\underline{\alpha}\,\mathbf{I}$ for
any branch-wise evaluation and, by Lemma~\ref{lem:pos_T},
$\mathbf{J}\succeq\underline{\alpha}\,\mathbf{I}\succ0$ at every iterate, so
each update is well defined. In the piecewise-linear case, $\mathbf{G}$ is a
piecewise-linear, strongly monotone map whose selection Jacobians share the positive semi-definite matrix $\underline{\underline{\mathbf{T}}}$
with non-positive off-diagonal entries (Lemma~\ref{lem:pos_T}) and differ only
in the positive diagonal; this is precisely the class for which the iteration
of Brugnano and Casulli~\cite{BrugnanoCasulli2009} terminates in finitely many
steps. If all cells share one branch, the branch pattern is constant,
$\mathbf{G}$ is affine there, and one step solves the system exactly. For a
general nonlinear EOS, $\mathbf{G}$ is Lipschitz continuous, strongly monotone
by Proposition~\ref{prop:pos_wellposed}, and piecewise $C^{1}$, the only
non-smoothness being the kink of $V_{\mathrm{mod}}$ at $p_{\varepsilon}$. If no
component of the solution lies exactly on the floor, $\mathbf{G}$ is $C^{1}$
in a neighbourhood of the solution with
$\mathbf{J}\succeq\underline{\alpha}\,\mathbf{I}$ nonsingular, and Newton's
method converges locally; the degenerate configuration
$p_i=p_{\varepsilon}$ falls into the semismooth setting, for which local
superlinear convergence is classical~\cite{QiSun1993}.
\end{proof}

\subsection{Enforcement of the divergence-free constraint}
\label{chp:divfree_CT}

As anticipated in Section~\ref{sec:governing}, the solenoidal condition
$\nabla\cdot\mathbf B=0$ is a constraint of the ideal MHD system: if it holds
for the initial data, the evolution of magnetic field preserve it for all
later times at the continuous level. At the discrete level this property is not
inherited automatically, and the divergence errors generated by a generic
discretisation are known to produce spurious oscillations, incorrect shock
speeds and unphysical states~\cite{BrackbillBarnes1980,Toth2000}. Several strategies
have been devised to control or eliminate these errors. The eight-wave
formulation of Powell et al.~\cite{Powell1999} augments the system with source
terms proportional to the divergence error, at the price of losing strict
conservation. In the projection method of Brackbill and
Barnes~\cite{BrackbillBarnes1980} the magnetic field produced by the base scheme is
projected onto the space of solenoidal fields through the solution of an
auxiliary Poisson problem. Hyperbolic divergence cleaning~\cite{Dedner2002}
couples the constraint to the conservation laws through a generalised Lagrange
multiplier, which damps and advects the divergence error towards the
boundaries. The most widely adopted alternative is the constrained transport
(CT) approach originally introduced by Evans and Hawley~\cite{EvansHawley1988}, in
which the magnetic field is represented on the cell interfaces and the electric
field at the cell corners; a discrete Faraday law then preserves a face-based
divergence operator down to machine accuracy. We refer to the review of
T\'oth~\cite{Toth2000} and to~\cite{BalsaraSpicer1999,GardinerStone2005,%
Rossmanith2006} for a detailed comparison of these families.

In the present work we explore two different constrained transport strategies:
a staggered adaptation of the method of Gardiner and Stone~\cite{GardinerStone2005},
and an unstaggered, potential-based adaptation of the technique introduced by
Rossmanith~\cite{Rossmanith2006}. Two aspects of their use within the present
algorithm are worth stating explicitly. First, since $\nabla\cdot\mathbf B=0$ is
a constraint rather than a member of any of the three flux sub-systems of
Section~\ref{sec:3-split}, it is enforced after all sub-systems have been
solved, and the correction acts on the magnetic field alone: the cell-centred
field $\tilde{\mathbf B}^{n+1}$ delivered by the implicit magnetic sub-system of
Section~\ref{sec:magnetic_subsystem} is regarded as a \emph{predictor}, used to advance the
momentum and the energy, and is subsequently discarded. Second, because the
correction is applied at the end of the time step, the density $\rho^{n+1}$ and
the velocity field $\mathbf u^{n+1}$ resulting from the pressure sub-system are
already available at that stage, so that the electric field entering the
correction is evaluated directly at the new time level. In the following we
describe the two adaptations in turn.

\subsubsection{Staggered adaptation of the constrained transport method}
\label{sec:CT-staggered}

Following Gardiner and Stone~\cite{GardinerStone2005}, Balsara and
Spicer~\cite{BalsaraSpicer1999} and Dumbser et al.~\cite{Dumbser2019}, we
introduce staggered magnetic field components $(B_x)_{i+1/2,j}$ and
$(B_y)_{i,j+1/2}$, located at the cell faces, which are evolved by a discrete
version of Faraday's induction law,
\begin{equation}
\label{eq:CT-faraday}
\begin{aligned}
(B_x)^{n+1}_{i+1/2,j}
  &= (B_x)^{n}_{i+1/2,j}
   - \frac{\Delta t}{\Delta y}\Big[(E_z)_{i+1/2,j+1/2}-(E_z)_{i+1/2,j-1/2}\Big],\\[2pt]
(B_y)^{n+1}_{i,j+1/2}
  &= (B_y)^{n}_{i,j+1/2}
   + \frac{\Delta t}{\Delta x}\Big[(E_z)_{i+1/2,j+1/2}-(E_z)_{i-1/2,j+1/2}\Big],
\end{aligned}
\end{equation}
where the corner values of the $z$-component of the electric field are detailed
below. It is readily verified that~\eqref{eq:CT-faraday} preserves the
face-based discrete divergence
\begin{equation}
\label{eq:CT-div-staggered}
[\nabla\cdot\mathbf B]_{i,j}
= \frac{(B_x)_{i+1/2,j}-(B_x)_{i-1/2,j}}{\Delta x}
+ \frac{(B_y)_{i,j+1/2}-(B_y)_{i,j-1/2}}{\Delta y} = 0
\end{equation}
exactly, to machine accuracy, at every time level, provided it vanishes
initially: the four corner contributions entering the two increments cancel in
pairs. It is therefore essential that the staggered field be initialised in a
way that is compatible with~\eqref{eq:CT-div-staggered}, which is achieved by
evaluating the initial condition for each component at the corresponding face
centre at $t=0$. The staggered components are the primary representation
of the magnetic field: they are stored and evolved for the entire simulation,
and are never reset from cell-centred data, since doing so would reintroduce
divergence errors.

The corner electric field in~\eqref{eq:CT-faraday} is assembled from
face-centred values through the gradient-corrected averaging of Gardiner and
Stone~\cite{GardinerStone2005},
\begin{equation}
\label{eq:CT-corner-E}
\begin{aligned}
(E_z)_{i+1/2,j+1/2}
= (\bar E_z)_{i+1/2,j+1/2}
& + \frac{\Delta x}{8}\left[
   \left(\frac{\partial E_z}{\partial x}\right)_{i+1/4,j+1/2}
 - \left(\frac{\partial E_z}{\partial x}\right)_{i+3/4,j+1/2}\right]
\\& + \frac{\Delta y}{8}\left[
   \left(\frac{\partial E_z}{\partial y}\right)_{i+1/2,j+1/4}
 - \left(\frac{\partial E_z}{\partial y}\right)_{i+1/2,j+3/4}\right],
 \end{aligned}
\end{equation}
with the arithmetic average of the four adjacent face-centred values
\begin{equation}
\label{eq:CT-Ebar}
(\bar E_z)_{i+1/2,j+1/2}
= \tfrac14\Big[(E_z)_{i+1/2,j}+(E_z)_{i+1/2,j+1}
             +(E_z)_{i,j+1/2}+(E_z)_{i+1,j+1/2}\Big].
\end{equation}
Rather than recomputing the face-centred values in~\eqref{eq:CT-Ebar} from an
auxiliary reconstruction and an independent approximate Riemann solve, we
exploit the duality between the electric field and the numerical fluxes of the
induction equation~\cite{Dematte2024},
\begin{equation}
\label{eq:CT-duality}
(E_z)_{i+1/2,j} = -\big(F_{[B_y]}\big)_{i+1/2,j},
\qquad
(E_z)_{i,j+1/2} = +\big(G_{[B_x]}\big)_{i,j+1/2},
\end{equation}
where $\big(F_{[B_y]}\big)_{i+1/2,j}$ and $\big(G_{[B_x]}\big)_{i,j+1/2}$ denote
the interface fluxes of the $B_y$ and $B_x$ rows employed in the cell-centred
update of the magnetic sub-system. We stress that, in the present three-way
splitting~\eqref{equ:3_split}, the induction rows of the advective and pressure
fluxes vanish identically, so that the whole numerical flux of the
magnetic field is carried by the implicit magnetic sub-system; no additional
contribution has to be accumulated from the other two sub-systems. These fluxes
are assembled from the density $\rho^{n+1}$, the velocity $\mathbf u^{n+1}$
resulting from the pressure sub-system, and the predictor
$\tilde{\mathbf B}^{n+1}$. As a consequence, the divergence-free field is
evolved by the same discrete operator that advances the cell-centred predictor,
and the constrained transport step acts as a projection of the implicit magnetic
update onto a divergence-free representation.

The gradients appearing in~\eqref{eq:CT-corner-E} are centred at the midpoints
between a face centre and the cell corner under consideration, and are
approximated by averaging one-sided differences between the face-centred values
and a cell-centred reference field. For instance, at the location
$(i+1/4,\,j+1/2)$,
\begin{equation}
\label{eq:CT-grad}
\left(\frac{\partial E_z}{\partial x}\right)_{i+1/4,j+1/2}
= \frac12\left[
  \frac{(E_z)_{i+1/2,j}-(\tilde E_z)_{i,j}}{\Delta x/2}
+ \frac{(E_z)_{i+1/2,j+1}-(\tilde E_z)_{i,j+1}}{\Delta x/2}\right],
\end{equation}
where the cell-centred reference value is evaluated from
$\mathbf E=-\mathbf u\times\mathbf B$ as
\begin{equation}
\label{eq:CT-Ecell}
(\tilde E_z)_{i,j}
= v^{n+1}_{i,j}\,(\tilde B_x)^{n+1}_{i,j}
- u^{n+1}_{i,j}\,(\tilde B_y)^{n+1}_{i,j}.
\end{equation}
The remaining gradients are obtained analogously. Once the staggered components
have been evolved through~\eqref{eq:CT-faraday}, the corrected cell-centred
field is recovered by simple averaging,
\begin{equation}
\label{eq:CT-average}
(B_x)^{n+1}_{i,j}=\tfrac12\Big[(B_x)^{n+1}_{i-1/2,j}+(B_x)^{n+1}_{i+1/2,j}\Big],
\qquad
(B_y)^{n+1}_{i,j}=\tfrac12\Big[(B_y)^{n+1}_{i,j-1/2}+(B_y)^{n+1}_{i,j+1/2}\Big],
\end{equation}
which overwrites the predictor; in two space dimensions the out-of-plane
component is left unchanged, $(B_z)^{n+1}_{i,j}=(\tilde B_z)^{n+1}_{i,j}$. As
discussed by Rossmanith~\cite{Rossmanith2006} and T\'oth~\cite{Toth2000}, the
cell-centred field obtained from~\eqref{eq:CT-average} satisfies, in addition to
the face-based constraint~\eqref{eq:CT-div-staggered}, the discrete
divergence-free condition
\begin{equation}
\label{eq:CT-div-toth}
[\nabla\cdot\mathbf B]_{i+1/2,j+1/2}
= \frac{(B_x)_{i+1,j}+(B_x)_{i+1,j+1}-(B_x)_{i,j}-(B_x)_{i,j+1}}{2\Delta x}
+ \frac{(B_y)_{i,j+1}+(B_y)_{i+1,j+1}-(B_y)_{i,j}-(B_y)_{i+1,j}}{2\Delta y}=0 .
\end{equation}

\subsubsection{Unstaggered adaptation of the constrained transport method}
\label{sec:CT-unstaggered}

The main drawback of the staggered formulation described above, when embedded
in a genuinely cell-centred algorithm, is that it requires a dual
representation of the magnetic field at both cell centres and cell faces. This
makes the method cumbersome to combine with adaptive mesh refinement, as noted
by Rossmanith~\cite{Rossmanith2006}, and awkward at sharp material interfaces,
where mixed-material Riemann solvers require the full set of variables to be
collocated at a single point. 
To retain a fully cell-centred
discretisation we therefore consider a second variant, adapted from the
unstaggered CT algorithm of Rossmanith~\cite{Rossmanith2006}, in which the
out-of-plane magnetic potential $A$ is evolved in place of the magnetic field
itself, with
\begin{equation}
\label{eq:CT-BcurlA}
\mathbf B = \nabla\times\big(A\,\mathbf e_z\big).
\end{equation}

Since the induction equation is the curl of the evolution equation for the
potential, $\partial_t A = -E_z$, and since the ideal Ohm law gives
$\mathbf E=-\mathbf u\times\mathbf B$, the substitution
of~\eqref{eq:CT-BcurlA} yields
$E_z = v B_x - u B_y = u\,\partial_x A + v\,\partial_y A$, so that the potential
equation reduces to the scalar transport equation
\begin{equation}
\label{eq:CT-transport}
\frac{\partial A}{\partial t}
+ u\,\frac{\partial A}{\partial x}
+ v\,\frac{\partial A}{\partial y} = 0 .
\end{equation}
The two formulations are equivalent at the continuous level, but not once
discretised. We deliberately discretise the transport form~\eqref{eq:CT-transport}
directly, rather than reconstructing corner electric fields as
in~\eqref{eq:CT-corner-E}: in this way the limiting acts on the smooth scalar
$A$ instead of on the magnetic field or on the electric field, so that the
numerical dissipation is of high order in smooth regions while spurious
oscillations in $A$, and hence in $\mathbf B$, are suppressed near steep
gradients. No corner reconstruction and no gradient estimates are required.
Given the cell-centred field $A^n_{i,j}$ and the velocity $\mathbf u^{n+1}$
obtained from the pressure sub-system, the correction consists of the following
steps.

\begin{itemize}
\item[\textbf{Step I}] Solve the implicit magnetic sub-system~\eqref{equ:discreteofB}
      for the predictor $\tilde{\mathbf B}^{n+1}$, which does not satisfy any
      discrete divergence constraint.
\item[\textbf{Step II}] Advance momentum, pressure and energy through the
      magnetic and pressure sub-systems of Sections~\ref{sec:magnetic_subsystem}
      and~\ref{sec:pressure} using $\tilde{\mathbf B}^{n+1}$, obtaining in
      particular the velocity field $\mathbf u^{n+1}$.
\item[\textbf{Step III}] Advance the magnetic potential
      by solving~\eqref{eq:CT-transport} with the velocity $\mathbf u^{n+1}$
      of Step~II, using the high-resolution wave-propagation method
      of~\cite{Rossmanith2006,leVeque2002}, in which first-order upwind
      fluctuations are combined with limited second-order correction fluxes
      acting directly on the increments of the potential:
      \begin{equation}
      \label{eq:CT-wp}
      A^{n+1}_{i,j} = A^{n}_{i,j}
      - \frac{\Delta t}{\Delta x}\Big[\mathcal A^{+}\!\Delta A_{i-1/2,j}
                                    + \mathcal A^{-}\!\Delta A_{i+1/2,j}
                                    + \tilde F_{i+1/2,j}-\tilde F_{i-1/2,j}\Big]
      - \frac{\Delta t}{\Delta y}\Big[\mathcal B^{+}\!\Delta A_{i,j-1/2}
                                    + \mathcal B^{-}\!\Delta A_{i,j+1/2}
                                    + \tilde G_{i,j+1/2}-\tilde G_{i,j-1/2}\Big],
      \end{equation}
      with the fluctuations and the correction fluxes in the $x$-direction
      given by
      \begin{equation}
      \label{eq:CT-wp-fluct}
      \mathcal A^{\pm}\!\Delta A_{i+1/2,j}
        = u^{\pm}_{i+1/2,j}\big(A_{i+1,j}-A_{i,j}\big),
      \qquad
      \tilde F_{i+1/2,j}
        = \tfrac12\,\big|u_{i+1/2,j}\big|
          \left(1-\frac{\Delta t}{\Delta x}\big|u_{i+1/2,j}\big|\right)
          \widetilde{\Delta A}_{i+1/2,j},
      \end{equation}
      where $u^{\pm}=\tfrac12(u\pm|u|)$, the interface velocity is
      $u_{i+1/2,j}=\tfrac12\big(u^{n+1}_{i,j}+u^{n+1}_{i+1,j}\big)$, and
      $\widetilde{\Delta A}_{i+1/2,j}$ denotes the limited increment obtained
      from the neighbouring increments through a standard TVD limiter; the
      $y$-direction contributions $\mathcal B^{\pm}$ and $\tilde G$ are defined
      analogously. Since only the material velocity enters~\eqref{eq:CT-wp},
      this update is subject to the same convective CFL
      condition~\eqref{eq:cfl_advective} that governs the explicit sub-system, and
      introduces no additional time-step restriction.
\item[\textbf{Step IV}] Recover the corrected, divergence-free cell-centred
      magnetic field as the discrete curl of the potential,
      \begin{equation}
      \label{eq:CT-curl}
      (B_x)^{n+1}_{i,j} = \frac{A^{n+1}_{i,j+1}-A^{n+1}_{i,j-1}}{2\Delta y},
      \qquad
      (B_y)^{n+1}_{i,j} = -\,\frac{A^{n+1}_{i+1,j}-A^{n+1}_{i-1,j}}{2\Delta x},
      \end{equation}
      which is the field used from this point onwards; the in-plane components
      of the predictor $\tilde{\mathbf B}^{n+1}$ are discarded, while the
      out-of-plane component is unchanged,
      $(B_z)^{n+1}_{i,j}=(\tilde B_z)^{n+1}_{i,j}$.
\end{itemize}

By construction, the field~\eqref{eq:CT-curl} satisfies the wide-stencil
discrete divergence-free condition
\begin{equation}
\label{eq:CT-div-unstaggered}
[\nabla\cdot\mathbf B]_{i,j}
= \frac{(B_x)_{i+1,j}-(B_x)_{i-1,j}}{2\Delta x}
+ \frac{(B_y)_{i,j+1}-(B_y)_{i,j-1}}{2\Delta y} = 0
\end{equation}
exactly and at machine accuracy, since the mixed second differences of $A$
cancel identically. As in the staggered variant, the auxiliary variable is the
primary representation of the magnetic field: the potential $A$ is stored and
evolved throughout the simulation, and is never overwritten from the
cell-centred field. Its initialisation is problem-specific: whenever the initial
data are prescribed directly in terms of a magnetic potential, that expression
is used; when only the magnetic field $\mathbf B^0$ is given, the initial
potential is obtained by solving the discrete Poisson problem
\begin{equation}
\label{eq:CT-poisson}
\Delta A^{0} = -\left(\frac{\partial B^0_y}{\partial x}
                     -\frac{\partial B^0_x}{\partial y}\right),
\end{equation}
as described in~\cite{Rossmanith2006}. In either case the initial cell-centred
magnetic field is re-initialised through
\begin{equation}
\label{eq:CT-reinit}
(B_x)^{0}_{i,j} = \frac{A^{0}_{i,j+1}-A^{0}_{i,j-1}}{2\Delta y},
\qquad
(B_y)^{0}_{i,j} = -\,\frac{A^{0}_{i+1,j}-A^{0}_{i-1,j}}{2\Delta x},
\end{equation}
so that~\eqref{eq:CT-div-unstaggered} is satisfied from the outset. The
boundary conditions for $A$ follow those imposed on $\mathbf B$: periodic
boundaries for the magnetic field induce periodic boundaries for the potential,
whereas a zeroth-order extrapolation of $\mathbf B$ requires a first-order
extrapolation of $A$.

\subsection{Second Order Time Discretisation}

The semi-discrete scheme introduced in the previous section is
first-order accurate in time.
To achieve second-order temporal accuracy, we employ a class of IMEX time integrators originally
proposed in~\cite{RKIMEX}.
These Runge--Kutta methods are particularly well suited to the present
framework, since a time linearisation has been applied to the three-way
flux splitting in \eqref{equ:3_split}, which naturally leads to a semi-implicit
formulation.

Without loss of generality, we consider the one-dimensional
system \eqref{equ:1DMHD} and the corresponding semi-discrete formulation
introduced in Sect. \ref{sec:scheme}.
Following~\cite{RKIMEX}, the governing equations can be written in autonomous
form as
\begin{equation}
\frac{\partial \mathbf{U}}{\partial t}
=
\mathcal{H}\bigl(\mathbf{U}_E(t), \mathbf{U}_I(t)\bigr),
\label{equ:RK}
\end{equation}
with initial condition $\mathbf{U}_0 = \mathbf{U}(t=0)$.
The operator $\mathcal{H}$ represents the spatial discretisation of the
flux terms in \eqref{equ:3_split} and depends on two arguments, namely the explicitly
treated state vector $\mathbf{U}_E$ and the implicitly treated state
vector $\mathbf{U}_I$.

To integrate system \eqref{equ:RK} in time, we adopt Implicit--Explicit
Runge--Kutta (IMEX-RK) schemes; see, for example,~\cite{RKIMEX2} for a review
of stiffly accurate IMEX methods commonly used in multiscale problems.
An IMEX-RK scheme is characterized by a double Butcher tableau,
\begin{equation}
\begin{array}{c|c}
\tilde{\mathbf{c}} & \tilde{\mathbf{A}} \\
\hline
& \tilde{\mathbf{b}}^{\mathrm T}
\end{array}
\qquad
\begin{array}{c|c}
\mathbf{c} & \mathbf{A} \\
\hline
& \mathbf{b}^{\mathrm T}
\end{array},
\end{equation}
where $\tilde{\mathbf{A}}, \mathbf{A} \in \mathbb{R}^{s \times s}$ and
$\tilde{\mathbf{c}}, \mathbf{c}, \tilde{\mathbf{b}}, \mathbf{b}
\in \mathbb{R}^s$.
The tilde notation refers to the explicit scheme, for which
$\tilde{\mathbf{A}}$ is strictly lower triangular with vanishing diagonal
entries, while $\mathbf{A}$ corresponds to the implicit scheme and is
lower triangular with non-zero diagonal entries.

In this work, we employ the second-order stiffly accurate LSDIRK2 scheme
introduced in~\cite{LSDIRK2}, which is defined by the following Butcher
tableaux:
\begin{equation}
\label{eq:butcher}
\begin{array}{c|cc}
0 & 0 & 0 \\
\tilde{c} & \tilde{c} & 0 \\
\hline
& 1-\alpha & \alpha
\end{array}
\qquad
\begin{array}{c|cc}
\alpha & \alpha & 0 \\
1 & 1-\alpha & \alpha \\
\hline
& 1-\alpha & \alpha
\end{array},
\end{equation}
with
\[
\alpha = 1 - \frac{1}{\sqrt{2}},
\qquad
\tilde{c} = \frac{1}{2\alpha}.
\]
Since $\tilde{\mathbf{b}} = \mathbf{b}$ and the weight vector
$\mathbf{b}$ coincides with the last row of $\mathbf{A}$, the scheme is
stiffly accurate (SA), which is a crucial property for ensuring asymptotic
consistency and accuracy~\cite{BOSCARINO2019594}.

To construct the semi-implicit IMEX-RK update, we initialize
$\mathbf{U}_E^n = \mathbf{U}_I^n = \mathbf{U}^n$.
The stage values for $i = 1, \dots, s$ are then computed as
\begin{equation}
\label{equ:stage_flux}
    \begin{aligned}
\mathbf{U}_E^{\,i}
&=
\mathbf{U}_E^n
+
\Delta t
\sum_{j=1}^{i-1}
\tilde{a}_{ij}\,\mathbf{k}_j,
\qquad 2 \le i \le s,
\\
\tilde{\mathbf{U}}_I^{\,i}
&=
\mathbf{U}_I^n
+
\Delta t
\sum_{j=1}^{i-1}
a_{ij}\,\mathbf{k}_j,
\qquad 2 \le i \le s,
\\
\mathbf{k}_i
&=
\mathcal{H}
\!\left(
\mathbf{U}_E^{\,i},
\tilde{\mathbf{U}}_I^{\,i}
+
\Delta t\, a_{ii}\,\mathbf{k}_i
\right),
\qquad 1 \le i \le s.
\end{aligned}
\end{equation}

The evaluation of the stage fluxes \eqref{equ:stage_flux} requires the solution of an
implicit system, which corresponds to the implicit magnetic and pressure
sub-systems defined in \eqref{equ:magnective sub-system} and \eqref{equ:pressure sub-system}, respectively.

Finally, owing to the stiffly accurate property of the IMEX scheme,
the numerical solution at the new time level is obtained as $\mathbf{U}^{n+1} = \mathbf{U}^{\,s}$,
i.e.\ the solution coincides with the last Runge--Kutta stage.

\subsection{Fully Discrete Scheme in Two Dimensions}
\label{sec:fully_discrete}
A single time step of the proposed semi-implicit method, advancing the solution
from time level $t^n$ to $t^{n+1}$, consists of the following steps. Throughout,
$\mathcal{F}^{x}$ and $\mathcal{F}^{y}$ denote the conservative finite-difference
flux-divergence operators in the $x$- and $y$-directions.

\begin{enumerate}

\item \textbf{Explicit advective sub-system.}
The advective sub-system is solved explicitly for density, momentum and energy,
based on the advective flux $\mathbf{F}^{c}$:
\begin{equation}
\begin{aligned}
\rho_{ij}^{\star}
&= \rho_{ij}^{n}
 - \Delta t\left[\mathcal{F}^{x}\!\big((\rho u)^{n}\big)
 + \mathcal{F}^{y}\!\big((\rho v)^{n}\big)\right],\\
(\rho u)_{ij}^{\star}
&= (\rho u)_{ij}^{n}
 - \Delta t\left[\mathcal{F}^{x}\!\big((\rho u^{2})^{n}\big)
 + \mathcal{F}^{y}\!\big((\rho u v)^{n}\big)\right],\\
(\rho v)_{ij}^{\star}
&= (\rho v)_{ij}^{n}
 - \Delta t\left[\mathcal{F}^{x}\!\big((\rho v u)^{n}\big)
 + \mathcal{F}^{y}\!\big((\rho v^{2})^{n}\big)\right],\\
(\rho w)_{ij}^{\star}
&= (\rho w)_{ij}^{n}
 - \Delta t\left[\mathcal{F}^{x}\!\big((\rho w u)^{n}\big)
 + \mathcal{F}^{y}\!\big((\rho w v)^{n}\big)\right],\\
E_{ij}^{\star}
&= E_{ij}^{n}
 - \Delta t\left[\mathcal{F}^{x}\!\big((E)^{n}\big)
 + \mathcal{F}^{y}\!\big((E)^{n}\big)\right].
\end{aligned}
\end{equation}
We note that $\rho_{ij}^{\star}=\rho_{ij}^{n+1}$, since the density equation does
not involve the pressure or magnetic fluxes and is therefore already fully updated
at this stage. The explicit momentum contribution is collected as
$(\rho\mathbf{u})_{ij}^{\star}=\big((\rho u)_{ij}^{\star},(\rho v)_{ij}^{\star},
(\rho w)_{ij}^{\star}\big)^{\mathrm{T}}$.

\item \textbf{Implicit magnetic field update.}
The implicit magnetic sub-system~\eqref{equ:mag-sub} is solved for the magnetic
field at $t^{n+1}$. To enhance robustness in regimes of large acoustic and
Alfvén Mach numbers, a Laplacian dissipation is added. The
Laplace operator $\mathbb{L}(\kappa,q)$ provides a second-order central
approximation of $\nabla\!\cdot(\kappa\nabla q)$ with a spatially varying
coefficient $\kappa$; in two dimensions,
\begin{equation}
\label{eq:laplacian}
\mathbb{L}(\kappa,q)_{ij}
= \frac{1}{\Delta x^{2}}\Big(
   \kappa_{i+\frac12 j}(q_{i+1\,j}-q_{ij})
 - \kappa_{i-\frac12 j}(q_{ij}-q_{i-1\,j})\Big)
+ \frac{1}{\Delta y^{2}}\Big(
   \kappa_{i\,j+\frac12}(q_{i\,j+1}-q_{ij})
 - \kappa_{i\,j-\frac12}(q_{ij}-q_{i\,j-1})\Big).
\end{equation}
The dissipation coefficient is built from the maximal eigenvalue of the magnetic
sub-system~\eqref{equ:magnective sub-system} $\lambda^{B}$ ,
\begin{equation}
\label{eq:kappa_B}
\kappa_{\mathbf{B},ij}^{x,n}=|\lambda^{\mathbf B}_{ij}|\,\Delta x,
\qquad
\kappa_{\mathbf{B},ij}^{y,n}=|\lambda^{\mathbf B}_{ij}|\,\Delta y,
\end{equation}
so that each Cartesian component of~\eqref{equ:mag-sub} is augmented on its
left-hand side with the implicit dissipation
$-\Delta t\,\mathbb{L}\bigl(\kappa_{B}^{\,n},\tilde B^{\,n+1}_{d}\bigr)$,
$d\in\{x,y,z\}$.
 As the Laplacian is discretised implicitly, it imposes no time step
restriction beyond the material CFL condition~\eqref{eq:cfl_advective}. The
resulting linear system is solved with a GMRES algorithm~\cite{GMRES} to a
prescribed tolerance, typically $10^{-14}$.

\item \textbf{Magnetic sub-system update (momentum and energy).}
Using $\mathbf{B}^{n+1}$, the momentum and energy are advanced with the magnetic
flux $\mathbf{F}^{\mathbf{B}}$:
\begin{equation}
\begin{aligned}
(\rho u)_{ij}^{**}
&= (\rho u)_{ij}^{\star}
 - \Delta t\left[\mathcal{F}^{x}\!\Big(\tfrac{\|\mathbf{B}^{n+1}\|^{2}}{2}-(B_x^{n+1})^{2}\Big)
 + \mathcal{F}^{y}\!\big(-B_x^{n+1}B_y^{n+1}\big)\right],\\
(\rho v)_{ij}^{**}
&= (\rho v)_{ij}^{\star}
 - \Delta t\left[\mathcal{F}^{x}\!\big(-B_x^{n+1}B_y^{n+1}\big)
 + \mathcal{F}^{y}\!\Big(\tfrac{\|\mathbf{B}^{n+1}\|^{2}}{2}-(B_y^{n+1})^{2}\Big)\right],\\
(\rho w)_{ij}^{**}
&= (\rho w)_{ij}^{\star}
 - \Delta t\left[\mathcal{F}^{x}\!\big(-B_x^{n+1}B_z^{n+1}\big)
 + \mathcal{F}^{y}\!\big(-B_y^{n+1}B_z^{n+1}\big)\right],\\
E_{ij}^{**}
&= E_{ij}^{\star}
 - \Delta t\left[\mathcal{F}^{x}\!\big(\|\mathbf{B}\|^{2}u-B_x(\mathbf{u}\!\cdot\!\mathbf{B})\big)
 + \mathcal{F}^{y}\!\big(\|\mathbf{B}\|^{2}v-B_y(\mathbf{u}\!\cdot\!\mathbf{B})\big)\right],
\end{aligned}
\end{equation}
with $(\rho\mathbf{u})_{ij}^{**}=\big((\rho u)_{ij}^{**},(\rho v)_{ij}^{**},
(\rho w)_{ij}^{**}\big)^{\mathrm{T}}$.

\item \textbf{First pressure solve.}
Substituting the momentum update into the discrete energy equation yields the
pressure wave equation, augmented with the shock-activated Laplacian dissipation.
The indicator $\chi^{n}$ and the dissipation coefficient are evaluated from the
known state at $t^{n}$ as in~\eqref{eq:flattener_mhd} and \eqref{equ:kappa_V},
\begin{equation}
\tilde{\kappa}^{\,x,n}_{ij}
 = \underline{\mathbf{V}}'(p^{n}_{ij})\,\chi^{n}_{ij}\,|u^{n}_{ij}|\,\Delta x,
\qquad
\tilde{\kappa}^{\,y,n}_{ij}
 = \underline{\mathbf{V}}'(p^{n}_{ij})\,\chi^{n}_{ij}\,|v^{n}_{ij}|\,\Delta y.
\end{equation}
Solved with the intermediate momentum $(\rho\mathbf{u})^{**}$ in the kinetic
energy, the augmented pressure wave equation reads, for a general equation of
state,
\begin{equation}
\label{eq:fd_pressurefirst}
\underline{\mathbf{V}}_{\mathrm{mod}}\bigl(\tilde p^{\,n+1}_{ij}\bigr)
 - \Delta t^{2}\,\mathbb{L}\bigl(h^{n},\tilde p^{\,n+1}\bigr)_{ij}
 - \Delta t\,\mathbb{L}\bigl(\tilde{\kappa}^{\,n},\tilde p^{\,n+1}\bigr)_{ij}
 = b_{ij},
\end{equation}
with the right-hand side
\begin{equation}
b_{ij} = E^{**}_{ij} - \tfrac{1}{2}\|\mathbf{B}^{n+1}_{ij}\|^{2}
- \frac{|(\rho\mathbf{u})^{**}_{ij}|^{2}}{2\rho^{n+1}_{ij}}
- \Delta t\left[\mathcal{F}^{x}\!\big(h^{n}(\rho u)^{**}\big)
+ \mathcal{F}^{y}\!\big(h^{n}(\rho v)^{**}\big)\right].
\end{equation}
Because $\underline{\mathbf{V}}(p)$ is nonlinear for a general EOS, the system is
mildly nonlinear and is solved by the nested Newton procedure, within which the positivity-preserving
modification of Section~\ref{sec:positivity} keeps the discrete pressure in the
admissible positive range without affecting conservation. Each Newton inner solve
is performed with GMRES to a tolerance of $10^{-14}$. For a linear EOS the
procedure reduces to a single linear solve.

\item \textbf{Momentum update.}
The momentum is advanced using the pressure flux $\tilde{p}^{\,n+1}$:
\begin{equation}
(\rho u)_{ij}^{n+1}=(\rho u)_{ij}^{**}-\Delta t\,\mathcal{G}^{x}\!\big(\tilde{p}^{\,n+1}\big),
\qquad
(\rho v)_{ij}^{n+1}=(\rho v)_{ij}^{**}-\Delta t\,\mathcal{G}^{y}\!\big(\tilde{p}^{\,n+1}\big),
\end{equation}
so that the momentum is fully updated at this stage.

\item \textbf{Second pressure solve.}
To guarantee thermodynamic consistency and conservation of the new energy, the
augmented pressure wave equation is solved once more, with the kinetic energy now
evaluated from the updated momentum and using the same shock-activated dissipation
coefficients:
\begin{equation}
\label{eq:fd_pressuresecond}
\underline{\mathbf{V}}_{\mathrm{mod}}\bigl( p^{\,n+1}_{ij}\bigr)
 - \Delta t^{2}\,\mathbb{L}\bigl(h^{n}, p^{\,n+1}\bigr)_{ij}
 - \Delta t\,\mathbb{L}\bigl(\tilde{\kappa}^{\,n}, p^{\,n+1}\bigr)_{ij}
 = b_{ij},
\end{equation}
where $b_{ij}$ is now evaluated with the kinetic energy
$|(\rho\mathbf{u})^{n+1}_{ij}|^{2}/(2\rho^{n+1}_{ij})$. This solve employs the same
nested Newton procedure and pressure positivity-preserving
modification (Section~\ref{sec:positivity}).

\item \textbf{Energy update.}
Once the updated pressure $p^{n+1}$, momentum $(\rho\mathbf{u})^{n+1}$ and magnetic
field $\mathbf{B}^{n+1}$ are available, the total energy is assembled directly from
its definition rather than advanced through a separate flux update. The internal
energy is obtained from the updated pressure through the general equation of state,
\begin{equation}
\label{eq:fd_internal_energy}
(\rho e)^{n+1}_{ij}=\underline{\mathbf{V}}_\mathrm{mod}(p^{n+1}_{ij}),
\end{equation}
and the total energy is recovered as the sum of its internal, kinetic and magnetic
contributions,
\begin{equation}
\label{eq:fd_energy_def}
E^{n+1}_{ij}
=\underline{\mathbf{V}}_{\mathrm{mod}}(p^{n+1}_{ij})
+\frac{1}{2}\frac{\big\|(\rho\mathbf{u})^{n+1}_{ij}\big\|^{2}}{\rho^{n+1}_{ij}}
+\frac{1}{2}\big\|\mathbf{B}^{n+1}_{ij}\big\|^{2}.
\end{equation}

\item \textbf{Divergence-free enforcement (constrained transport).}
Finally, the magnetic field is corrected to satisfy the discrete divergence-free
constraint. Since $\nabla\!\cdot\mathbf{B}=0$ is a constraint rather than part of any
flux sub-system, it is enforced through one of the two constrained transport variants of
Section~\ref{chp:divfree_CT}, which act only on the magnetic field.

For the staggered constrained transport, the face-centred electric field is given by the
duality relation~\eqref{eq:CT-duality} with the interface fluxes of the implicit magnetic
sub-system, and the corner values $(E_z)_{i+1/2,j+1/2}$ follow from~\eqref{eq:CT-corner-E}.
The face-centred magnetic field components are evolved by the discrete Faraday law
\begin{equation}
\begin{aligned}
(B_x)^{n+1}_{i+1/2,j} &= (B_x)^{n}_{i+1/2,j}
 - \frac{\Delta t}{\Delta y}\Big[(E_z)_{i+1/2,j+1/2}-(E_z)_{i+1/2,j-1/2}\Big],\\
(B_y)^{n+1}_{i,j+1/2} &= (B_y)^{n}_{i,j+1/2}
 + \frac{\Delta t}{\Delta x}\Big[(E_z)_{i+1/2,j+1/2}-(E_z)_{i-1/2,j+1/2}\Big],
\end{aligned}
\end{equation}
and the cell-centred field is recovered by averaging the adjacent face values, so that
the face-based divergence~\eqref{eq:CT-div-toth} vanishes to machine accuracy.

For the unstaggered constrained transport, the out-of-plane magnetic potential is
advanced in transport form,
\begin{equation}
\label{eq:A_transport_step}
\frac{\partial A}{\partial t}
 + u^{\,n+1}\frac{\partial A}{\partial x}
 + v^{\,n+1}\frac{\partial A}{\partial y} = 0 ,
\end{equation}
discretised with the high-resolution wave-propagation method
of~\cite{leVeque2002,Rossmanith2006}, and the divergence-free magnetic field is
recovered as the discrete curl
\begin{equation}
(B_x)^{n+1}_{ij}=\frac{A^{n+1}_{i\,j+1}-A^{n+1}_{i\,j-1}}{2\Delta y},
\qquad
(B_y)^{n+1}_{ij}=-\frac{A^{n+1}_{i+1\,j}-A^{n+1}_{i-1\,j}}{2\Delta x},
\end{equation}
which satisfies the wide-stencil discrete divergence-free
condition~\eqref{eq:CT-div-unstaggered} exactly.

\end{enumerate}
The operator $\mathcal{H}(\mathbf{U}_E,\mathbf{U}_I)$ in~\eqref{equ:stage_flux}
is precisely the one-step update given by steps~1--8 above: evaluating a single stage of
the IMEX-RK scheme~\eqref{eq:butcher} amounts to carrying out these steps, with the
explicit convective flux $\mathbf{F}^c$ contributing to $\mathbf{U}_E$ and the implicit
magnetic, pressure and energy sub-systems contributing to $\mathbf{U}_I$. At stage $k$,
step~1 advances the explicit part with the tableau
$(\tilde{\mathbf{A}},\tilde{\mathbf{c}})$, while steps~2--8 solve the implicit
sub-systems over the increment $\alpha\,\Delta t$ using the stage right-hand side
assembled from the previous stages; the reuse of the same magnetic and pressure solves
and of the constrained transport step requires no additional machinery. The
construction is identical in one and two space dimensions, since the two-stage
integrator acts on the spatial operator $\mathcal{H}$ (steps~1--8) unchanged.

In summary, the numerical scheme is stabilised through three complementary sources of
artificial viscosity. The first contribution stems from the dissipative component of
the explicit Rusanov flux, which scales with the local flow velocity $\mathbf{u}$. The
second stabilisation mechanism is provided by an implicit Laplacian operator acting on
the magnetic sub-system, with diffusion coefficients proportional to the characteristic
magnetic eigenvalue $\lambda^\mathbf{B}$. This dissipation plays a particularly important role in regimes characterized by large acoustic Mach and Alfv\'en Mach numbers. The third source of dissipation is introduced through a
Laplacian-based stabilisation of the pressure field, with diffusion coefficients chosen
proportional to the local flow speed. Owing to the implicit treatment of the Laplacian
terms, neither the magnetic nor the pressure stabilisation imposes additional
restrictions on the material time step required for numerical stability.

\section{Numerical results}

In this section we validate the proposed semi-implicit scheme on a broad suite of test
problems, following the order in which the ingredients of the method are introduced. We begin by assessing the accuracy of the full scheme on a smooth magnetohydrodynamic
vortex, measuring the experimental order of convergence as the Alfv\'en number is
decreased. We then turn to one-dimensional Riemann problems, presented together: first
with the magnetic field switched off, so that the system reduces to the compressible
Euler equations and a set of Redlich--Kwong Riemann problems verifies that the nested
Newton solution of the nonlinear pressure system reproduces the correct wave structure,
by comparison with the exact solution; then a set of ideal-gas MHD Riemann problems, on
which the shock-capturing capabilities of the scheme are examined against classical,
stringent one-dimensional benchmarks. We next consider two-dimensional benchmarks for
ideal MHD featuring strong compressible structures, after which we turn to low acoustic
Mach number, pressure- and magnetic-pressure-dominated configurations that probe the
all-acoustic-Mach and all-Alfv\'en-Mach behaviour of the method. Finally, a series of
strongly magnetised, low-plasma-beta tests is used to demonstrate the robustness and
positivity-preserving character of the scheme under extreme conditions.

Unless otherwise stated, all test cases employ the ideal gas equation of state
with a ratio of specific heats $\gamma = 5/3$. The Redlich--Kwong tests of
Section~\ref{sec:rk_riemann} instead use the parameters $\mathfrak{a}_0 = 0.5$,
$b = 0.5$, $c_v = 1$ and $R = 0.4$. The time step
is built from the advective eigenvalues $\lambda^{c}$, i.e.\ the fluid velocity
alone, according to~\eqref{eq:cfl_advective}, with
$\mathrm{CFL} = 0.8$ and the first time step is taken as a magnetosonic explicit step. The positivity safeguard of Sections~\ref{sec:positivity}
and~\ref{sec:properties} is active in every run, with
$p_{\varepsilon}=10^{-7}$ and $\theta=0.9$; the steep slope is chosen at each
solve as
$M=\max\!\bigl(M_{0},\,(V(p_{\varepsilon})-b_{\min})_{+}/((1-\theta)\,
p_{\varepsilon})\bigr)$ with $M_{0}=10^{6}/(\gamma-1)$, in accordance with
Corollary~\ref{cor:pos_strict}; the baseline $M_{0}$ merely keeps the floored branch
well defined when the bound of the corollary is inactive. The
divergence-free constraint on the magnetic field is enforced by unstaggered constrained
transport; we have verified for the relevant test cases that the discrete
divergence error stays at machine-precision level throughout the simulation.
\label{sec:results}

\subsection{Accuracy of the numerical scheme}
\label{sec:convergence}

\begin{table}[t]
\centering
\caption{MHD traveling vortex. Maximum acoustic Mach number $M_c$ and Alfv\'en
speed $c_a$ at the initial time.}
\label{tab:vortex_mach}
\begin{tabular}{l cccccc}
\toprule
 & \multicolumn{6}{c}{Background density $\rho_0 = 10^{-k}$} \\
 & $k=0$ & $k=1$ & $k=2$ & $k=3$ & $k=4$ & $k=5$ \\
\midrule
$M_c$ & 1.606E$+$00 & 4.489E$-$01 & 1.529E$-$01 & 4.832E$-$02 & 1.529E$-$02 & 4.832E$-$03 \\
$c_a$   & 1.549E$-$01 & 4.900E$-$01 & 1.549E$+$00 & 4.900E$+$00 & 1.549E$+$01 & 4.900E$+$01 \\
\bottomrule
\end{tabular}
\end{table}

\begin{table}[t]
\centering
\footnotesize
\setlength{\tabcolsep}{4pt}
\caption{MHD traveling vortex: $L^2$ errors and EOC of the proposed second-order
scheme for varying background density $\rho_0$ at $t_f=1$.}
\label{tab:vortex_eoc}
\begin{tabular}{ll cc cc cc cc cc}
\toprule
$\rho_0$ & $N_x{=}N_y$
 & \multicolumn{2}{c}{$\rho$}
 & \multicolumn{2}{c}{$u$}
 & \multicolumn{2}{c}{$p$}
 & \multicolumn{2}{c}{$B_x$}
 & \multicolumn{2}{c}{$A_z$} \\
 & & $L^2$ & EOC & $L^2$ & EOC & $L^2$ & EOC & $L^2$ & EOC & $L^2$ & EOC \\
\midrule
\multirow[t]{4}{*}{$10^{0}$}
 & 32  & 2.65e-02 & --   & 4.53e-02 & --   & 1.52e-02 & --   & 5.47e-02 & --   & 3.77e-02 & --   \\
 & 64  & 8.04e-03 & 1.72 & 8.67e-03 & 2.39 & 5.06e-03 & 1.59 & 1.52e-02 & 1.85 & 9.84e-03 & 1.94 \\
 & 128 & 2.00e-03 & 2.00 & 1.80e-03 & 2.27 & 1.48e-03 & 1.78 & 3.90e-03 & 1.96 & 2.48e-03 & 1.99 \\
 & 256 & 4.80e-04 & 2.06 & 4.21e-04 & 2.10 & 4.00e-04 & 1.88 & 9.82e-04 & 1.99 & 6.22e-04 & 2.00 \\
\multirow[t]{4}{*}{$10^{-1}$}
 & 32  & 2.06e-03 & --   & 4.70e-02 & --   & 8.81e-03 & --   & 5.45e-02 & --   & 3.76e-02 & --   \\
 & 64  & 6.00e-04 & 1.78 & 1.04e-02 & 2.17 & 2.57e-03 & 1.78 & 1.51e-02 & 1.85 & 9.81e-03 & 1.94 \\
 & 128 & 1.54e-04 & 1.96 & 2.41e-03 & 2.12 & 6.75e-04 & 1.93 & 3.88e-03 & 1.96 & 2.47e-03 & 1.99 \\
 & 256 & 3.88e-05 & 1.99 & 5.85e-04 & 2.04 & 1.75e-04 & 1.95 & 9.77e-04 & 1.99 & 6.20e-04 & 2.00 \\
\multirow[t]{4}{*}{$10^{-2}$}
 & 32  & 2.78e-04 & --   & 1.32e-01 & --   & 7.54e-03 & --   & 5.06e-02 & --   & 3.46e-02 & --   \\
 & 64  & 6.42e-05 & 2.12 & 3.60e-02 & 1.87 & 2.09e-03 & 1.85 & 1.36e-02 & 1.90 & 8.75e-03 & 1.98 \\
 & 128 & 1.56e-05 & 2.04 & 9.20e-03 & 1.97 & 5.53e-04 & 1.92 & 3.44e-03 & 1.98 & 2.19e-03 & 2.00 \\
 & 256 & 3.88e-06 & 2.00 & 2.31e-03 & 1.99 & 1.47e-04 & 1.91 & 8.61e-04 & 2.00 & 5.47e-04 & 2.00 \\
\multirow[t]{4}{*}{$10^{-3}$}
 & 32  & 7.65e-05 & --   & 4.80e-01 & --   & 3.76e-03 & --   & 2.54e-02 & --   & 1.63e-02 & --   \\
 & 64  & 1.50e-05 & 2.35 & 1.23e-01 & 1.97 & 9.75e-04 & 1.95 & 5.84e-03 & 2.12 & 3.49e-03 & 2.22 \\
 & 128 & 2.45e-06 & 2.61 & 3.22e-02 & 1.93 & 2.35e-04 & 2.05 & 1.42e-03 & 2.04 & 8.33e-04 & 2.07 \\
 & 256 & 4.50e-07 & 2.45 & 8.26e-03 & 1.96 & 6.20e-05 & 1.92 & 3.53e-04 & 2.01 & 2.06e-04 & 2.02 \\
\multirow[t]{4}{*}{$10^{-4}$}
 & 32  & 4.51e-05 & --   & 1.10e+00 & --   & 7.57e-03 & --   & 3.95e-02 & --   & 3.00e-02 & --   \\
 & 64  & 1.21e-05 & 1.90 & 2.25e-01 & 2.28 & 2.42e-03 & 1.64 & 1.16e-02 & 1.76 & 7.59e-03 & 1.98 \\
 & 128 & 1.68e-06 & 2.85 & 5.25e-02 & 2.10 & 5.10e-04 & 2.25 & 2.30e-03 & 2.34 & 1.49e-03 & 2.35 \\
 & 256 & 2.21e-07 & 2.93 & 1.36e-02 & 1.95 & 1.02e-04 & 2.32 & 4.43e-04 & 2.38 & 3.24e-04 & 2.20 \\
\multirow[t]{4}{*}{$10^{-5}$}
 & 32  & 2.00e-05 & --   & 8.99e+00 & --   & 1.91e-02 & --   & 1.21e-01 & --   & 1.15e-01 & --   \\
 & 64  & 7.19e-06 & 1.47 & 5.19e-01 & 4.11 & 9.56e-03 & 1.00 & 4.92e-02 & 1.29 & 4.15e-02 & 1.46 \\
 & 128 & 2.12e-06 & 1.76 & 1.19e-01 & 2.12 & 3.26e-03 & 1.55 & 1.52e-02 & 1.70 & 1.17e-02 & 1.83 \\
 & 256 & 2.13e-07 & 3.31 & 2.13e-02 & 2.48 & 5.83e-04 & 2.48 & 2.59e-03 & 2.55 & 1.61e-03 & 2.86 \\
\bottomrule
\end{tabular}
\end{table}
To quantify the accuracy of the proposed scheme we adopt a variant of the
analytic magnetohydrodynamic vortex of~\cite{accuracyvortex, BoscheriThomann2024}, solved on the square
domain $\Omega = [-5,5]^2$. The exact solution superimposes a smooth
perturbation $\delta\mathbf{q}$ onto a uniform background
$\mathbf{q}_0 = (\rho_0, \mathbf{u}_0, p_0, \mathbf{0})$. With the polar angle
$\theta$ and the radius $r=\sqrt{x^2+y^2}$, the velocity and magnetic-field
perturbations are prescribed as
\begin{align}
(\delta u,\, \delta v)
  &= \frac{\tilde{v}}{2\pi}\,
     \exp\!\left(\frac{1-r^{2}}{2}\right)\,
     \bigl(-r\sin\theta,\ r\cos\theta\bigr), \label{eq:vortex_vel}\\[2pt]
(\delta B_x,\, \delta B_y)
  &= \frac{\tilde{B}}{2\pi}\,
     \exp\!\left(\frac{1-r^{2}}{2}\right)\,
     \bigl(-r\sin\theta,\ r\cos\theta\bigr), \label{eq:vortex_mag}
\end{align}
with the amplitude constants $\tilde{v}=\sqrt{2\pi}$ and $\tilde{B}=1$.
 The magnetic perturbation is generated
by the out-of-plane vector potential
\begin{equation}
\delta A_z = \frac{\tilde{B}}{2\pi}\,\exp\!\left(\frac{1-r^{2}}{2}\right).
\label{eq:vortex_Az}
\end{equation}
The pressure is set so as to enforce radial equilibrium between the centrifugal
force of the swirling motion and the combined magnetic tension and
magnetic-pressure gradient produced by the curved field lines. Written in polar
form, this balance reads
\begin{equation}
\frac{\mathrm{d}\,\delta p}{\mathrm{d}r}
  = \frac{\rho\,\delta v_\theta^{2}}{r}
  - \frac{\delta B_\theta^{2}}{r}
  - \frac{1}{2}\,\frac{\mathrm{d}\bigl(\delta B_\theta^{2}\bigr)}{\mathrm{d}r},
\label{eq:vortex_equil}
\end{equation}
where the azimuthal perturbations are
\begin{equation}
\delta v_\theta = \frac{\tilde{v}}{2\pi}\,r\,\exp\!\left(\frac{1-r^{2}}{2}\right),
\qquad
\delta B_\theta = \frac{\tilde{B}}{2\pi}\,r\,\exp\!\left(\frac{1-r^{2}}{2}\right).
\label{eq:vortex_theta}
\end{equation}
Integrating~\eqref{eq:vortex_equil} inward from infinity yields the closed-form
pressure perturbation
\begin{equation}
\delta p = \frac{1}{2}\,e^{\,1-r^{2}}
  \left[\frac{\tilde{B}^{2}}{(2\pi)^{2}}\bigl(1-r^{2}\bigr)
  - \rho_0\left(\frac{\tilde{v}}{2\pi}\right)^{\!2}\right].
\label{eq:vortex_dp}
\end{equation}
The vortex is advected diagonally through the domain by taking
$\mathbf{u}_0=(1,1,0)$, together with $p_0=1$. Since the salient feature of the
scheme is a stability limit governed by a material CFL condition that is
independent of the acoustic and Alfv\'en wave speeds, we probe the experimental
order of convergence (EOC) under a variation of the Alfv\'en speed. A uniform
background field cannot be added without spoiling the analytic character of the
solution; we therefore exploit the $\rho^{-1/2}$ scaling of the Alfv\'en speed
and lower the background density $\rho_0$, which raises the Alfv\'en speed and
hence reduces the Alfv\'en Mach number. Decreasing $\rho_0$ at the same time
increases the sound speed and attenuates the kinetic contribution to the
pressure perturbation~\eqref{eq:vortex_dp}, whereas the magnetic contribution is
left unchanged. As the background field vanishes, a domain-wide Alfv\'en Mach
number is not defined; Table~\ref{tab:vortex_mach} therefore reports the peak
acoustic Mach number $M_c$ and the Alfv\'en speed $c_a$
for the sequence $\rho_0=10^{-k}$, $k=0,\dots,5$.

The vortex is evolved up to the final time $t_f=1$ on a hierarchy of
successively refined grids. Table~\ref{tab:vortex_eoc} summarises the $L^2$
errors and the associated EOC for $\rho$, $u$, $p$, $B_x$ and $A_z$. The design order of two is recovered for every variable in all regimes. Finally, to underline
the advantage of a scale-independent stability constraint, Fig.~\ref{fig:convergence_dt_ratio}
displays the ratio between a fully explicit step, limited by the fastest
characteristic speeds, and the step admitted by the CFL condition of the present
scheme. Particularly at low acoustic and Alfv\'en Mach numbers, the method
sustains time steps several orders of magnitude larger than the explicit limit
while preserving second-order accuracy in space and time.

\begin{figure}
    \centering
    \includegraphics[width=0.6\linewidth]{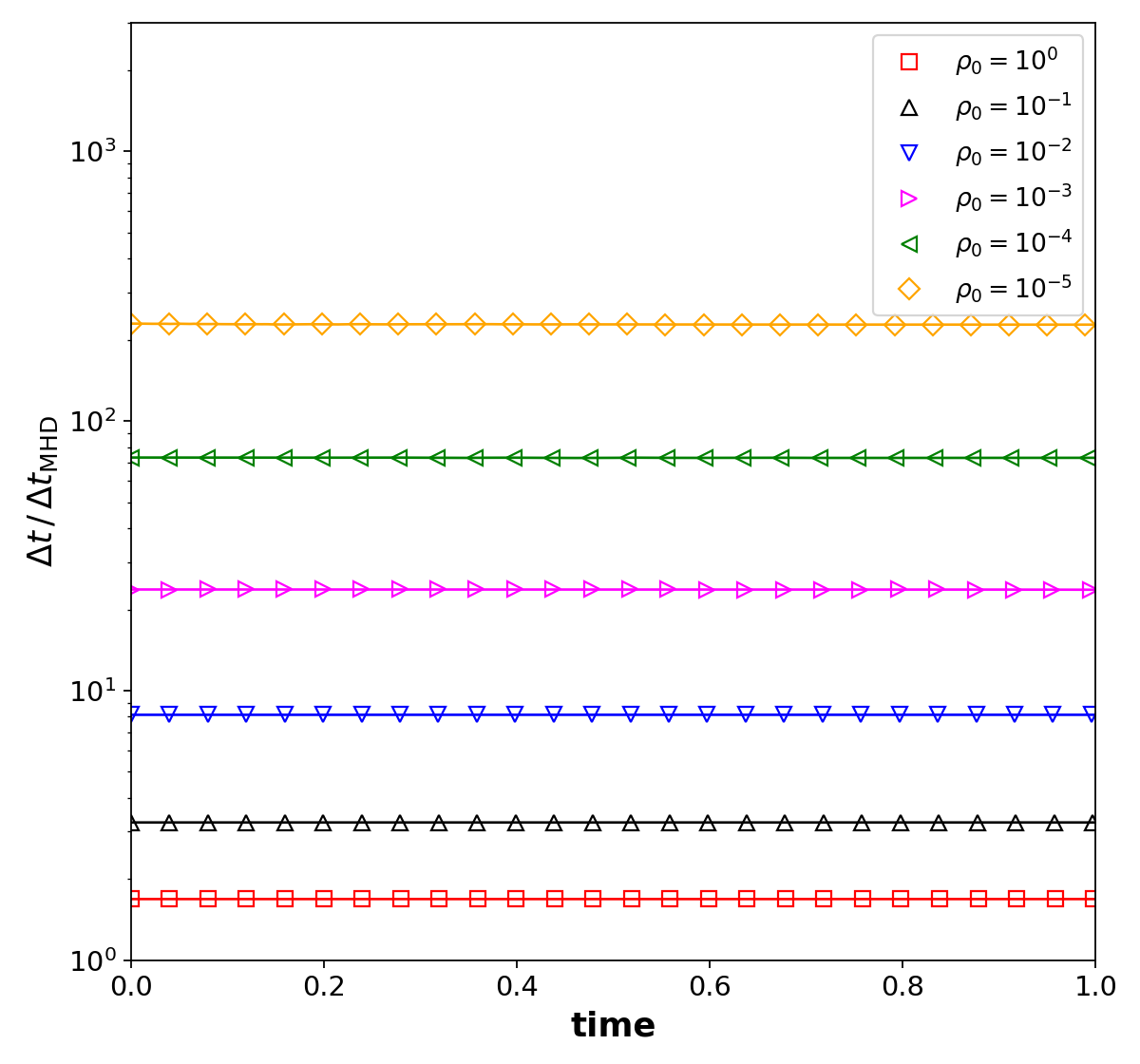}
    \caption{MHD traveling vortex. Time evolution of the ratio between the material time step
    $\Delta t$ of the present semi-implicit IMEX scheme and the magnetosonic time step
    $\Delta t_{\mathrm{MHD}}$ of a fully explicit finite volume scheme, for the background
    densities $\rho_0 = 10^{-k}$, $k = 0,\dots,5$ of Table~\ref{tab:vortex_mach}.}
    \label{fig:convergence_dt_ratio}
\end{figure}

\subsection{One-dimension Riemann test}
Having established the convergence order in Section~\ref{sec:convergence}, we now assess
the shock-capturing behaviour of the scheme on one-dimensional Riemann problems, in
two steps. First, in Section~\ref{sec:rk_riemann}, we isolate and validate the
Redlich--Kwong EOS branch: the magnetic field is switched off, so that the governing
equations reduce to the compressible Euler system, and the solution is compared
against the exact Redlich--Kwong Riemann solution. Second, in Section~\ref{sec:MHD_1d}
the magnetic field is restored and the full ideal-MHD shock-capturing capabilities
are examined on a set of classical, stringent Riemann problems for an ideal gas.

\subsubsection{One-dimensional validation of the generic equation of state}
\label{sec:rk_riemann}
\begin{table}[htbp]
\centering
\caption{Initial left (L) and right (R) states for the Redlich--Kwong Euler
Riemann test problems. The discontinuity is located at $x_d$. In all cases
$c_v = 1$, $R = 0.4$ and $\mathfrak{a}_0 = b = 0.5$.}
\label{tab:rk_riemann}
\setlength{\tabcolsep}{3.5pt}
\renewcommand{\arraystretch}{1.15}
\begin{tabular}{llllllll}   
\toprule
Test & & $\rho$ & $u$ & $p$ & $x_d$ & $t_{\mathrm{out}}$ \\
\midrule
\multirow{2}{*}{Redlich--Kwong EOS test 1}
 & L: & $1.0$ & $+1.0$ & $2.0$  & \multirow{2}{*}{$0.0$} & \multirow{2}{*}{$0.1$} \\
 & R: & $1.0$ & $-1.0$ & $1.0$  & & \\
\addlinespace
\multirow{2}{*}{Redlich--Kwong EOS test 2}
 & L: & $1.0$   & $0.0$ & $1.0$  & \multirow{2}{*}{$0.0$} & \multirow{2}{*}{$0.2$} \\
 & R: & $0.125$ & $0.0$ & $0.1$  & & \\
\addlinespace
\multirow{2}{*}{Redlich--Kwong EOS test 3}
 & L: & $1.0$ & $0.0$ & $1000$ & \multirow{2}{*}{$0.1$} & \multirow{2}{*}{$0.008$} \\
 & R: & $1.0$ & $0.0$ & $0.01$ & & \\
\addlinespace
\multirow{2}{*}{Redlich--Kwong EOS test 4}
 & L: & $1.0$ & $0.0$ & $2.0$  & \multirow{2}{*}{$0.0$} & \multirow{2}{*}{$0.1$} \\
 & R: & $1.5$ & $0.0$ & $1.0$  & & \\
\bottomrule
\end{tabular}
\vspace{2pt}
\end{table}

\begin{figure}
    \centering
    \begin{minipage}{0.32\linewidth}
        \includegraphics[width=\linewidth]{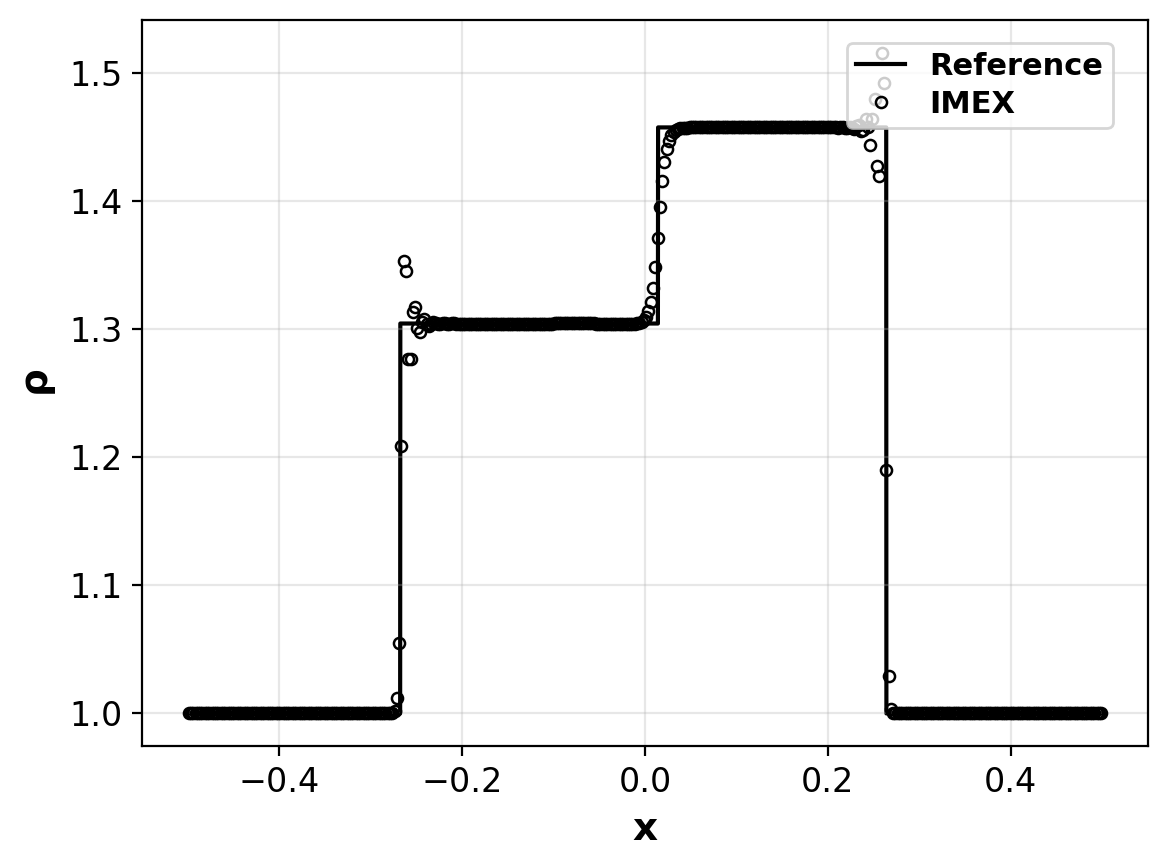}
    \end{minipage}
    \hspace{0.01\linewidth}%
    \begin{minipage}{0.32\linewidth}
        \includegraphics[width=\linewidth]{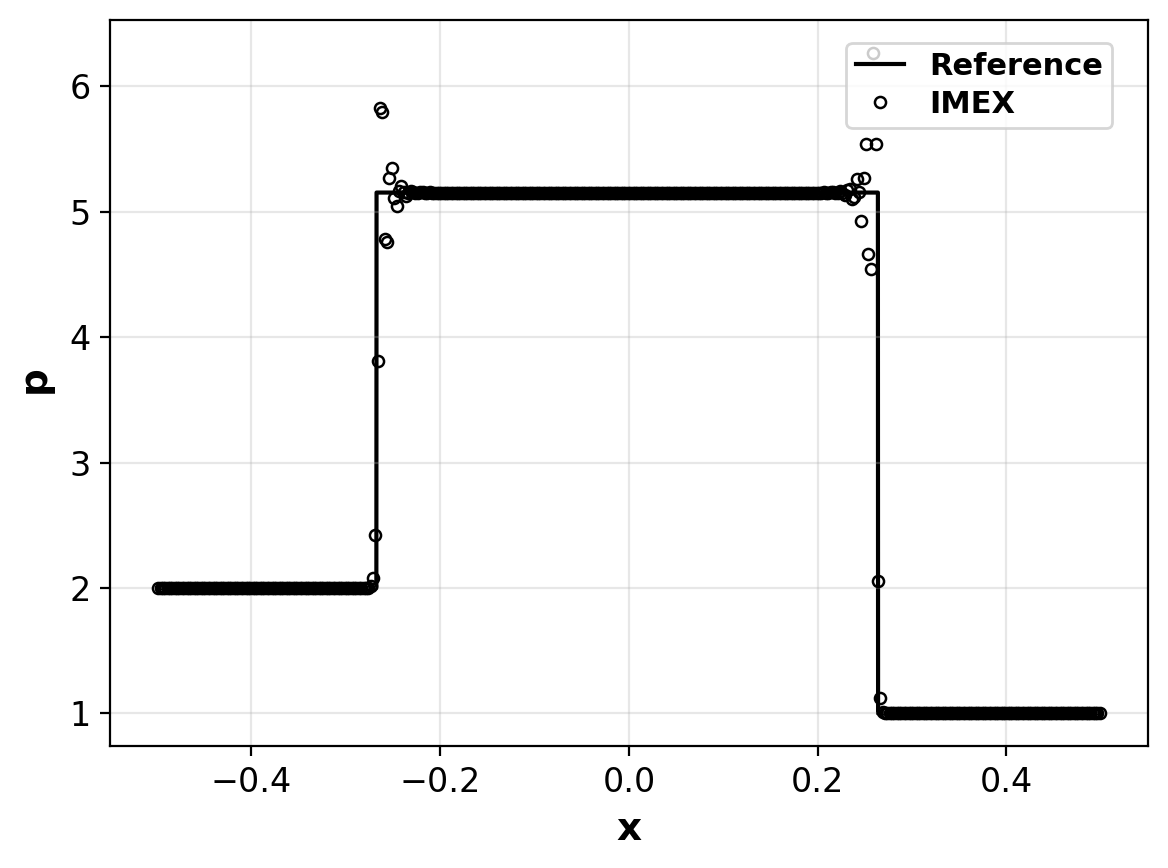}
    \end{minipage}
    \hspace{0.01\linewidth}%
    \begin{minipage}{0.32\linewidth}
        \includegraphics[width=\linewidth]{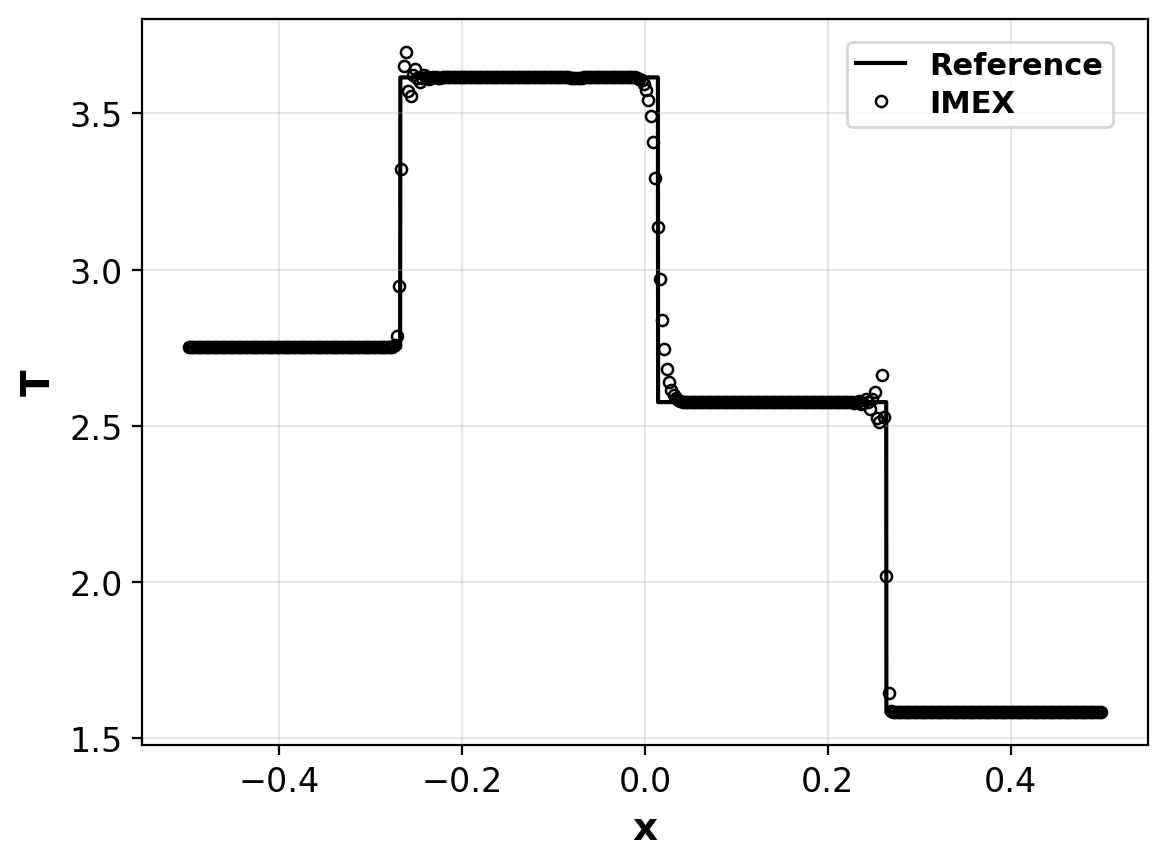}
    \end{minipage}
    \vspace{0.5em}

    \begin{minipage}{0.32\linewidth}
        \includegraphics[width=\linewidth]{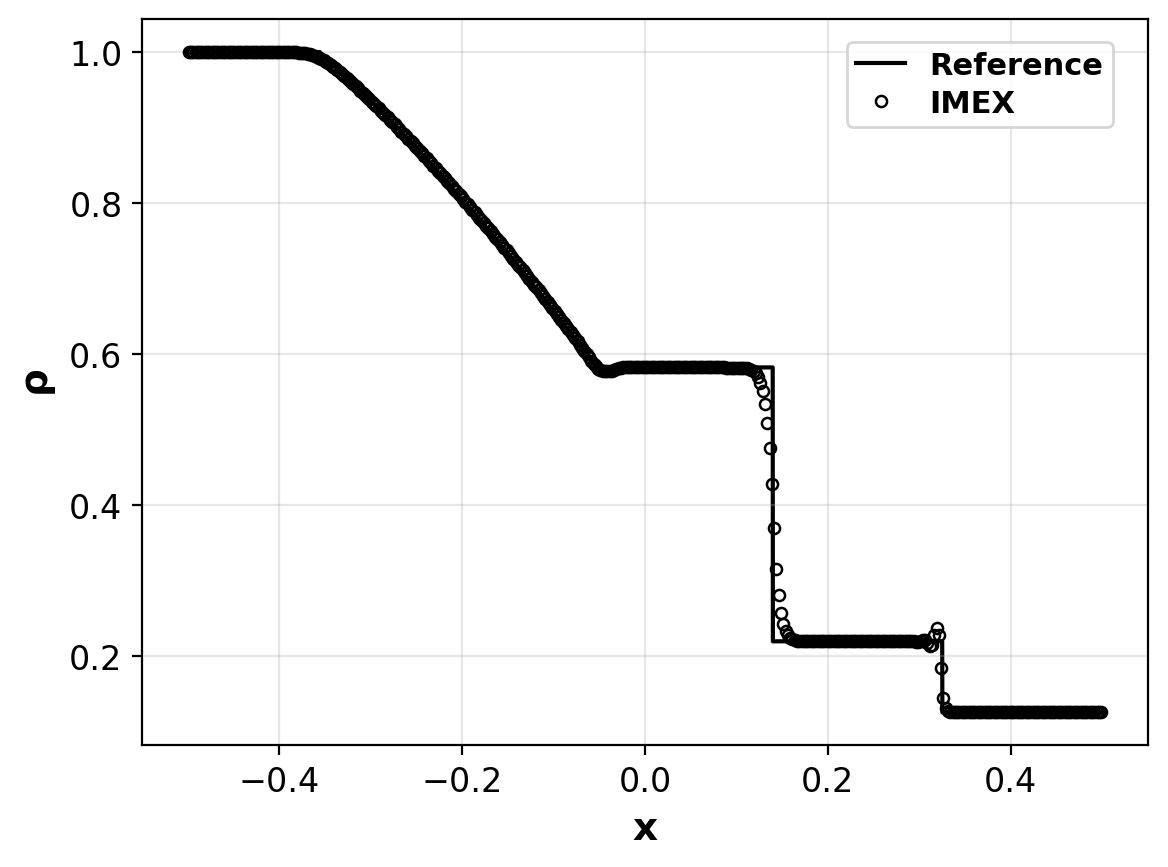}
    \end{minipage}
    \hspace{0.01\linewidth}%
    \begin{minipage}{0.32\linewidth}
        \includegraphics[width=\linewidth]{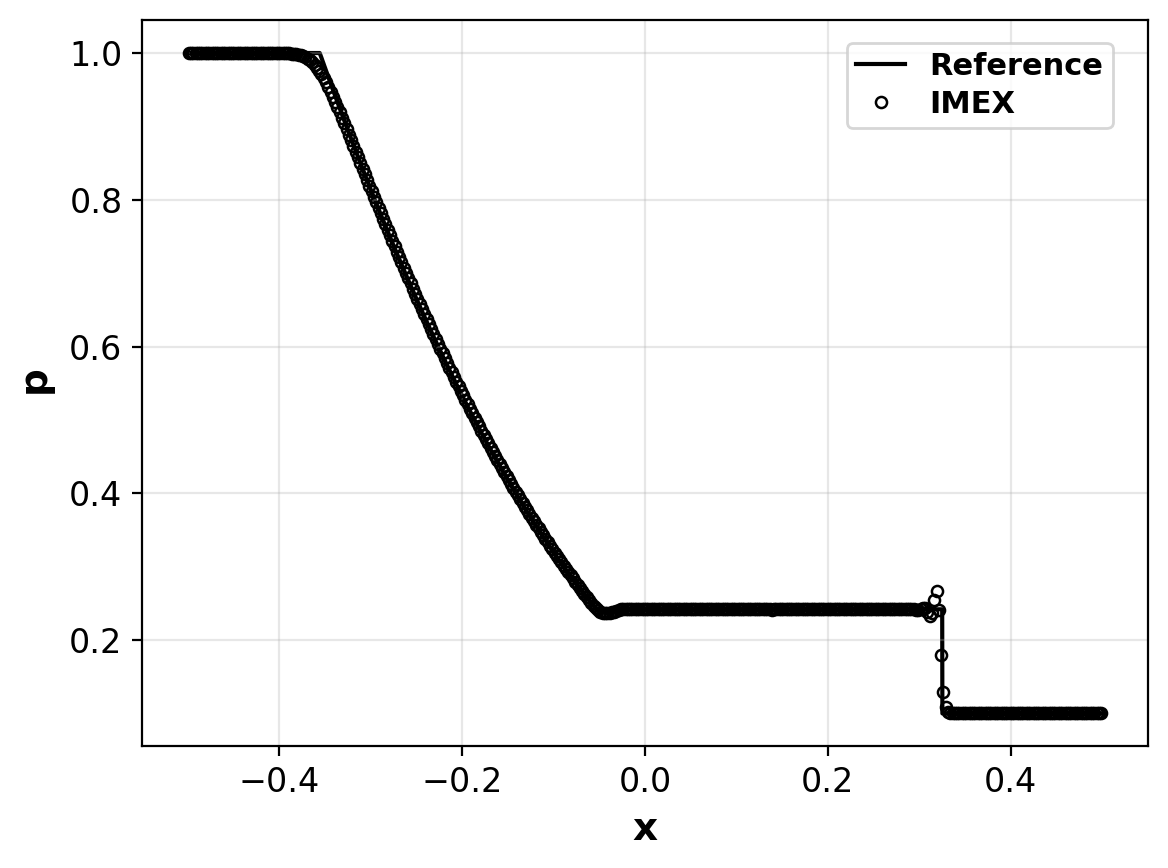}
    \end{minipage}
    \hspace{0.01\linewidth}%
    \begin{minipage}{0.32\linewidth}
        \includegraphics[width=\linewidth]{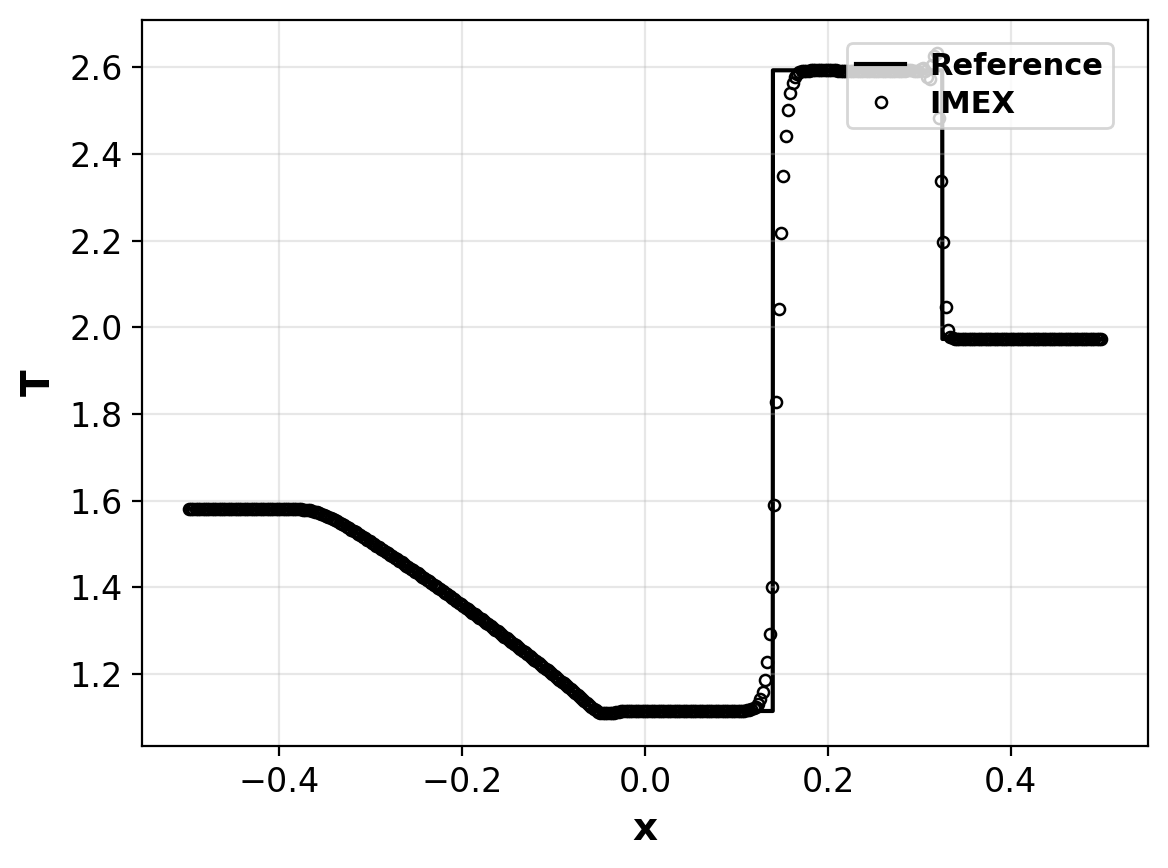}
    \end{minipage}
    \vspace{0.5em}
    \begin{minipage}{0.32\linewidth}
        \includegraphics[width=\linewidth]{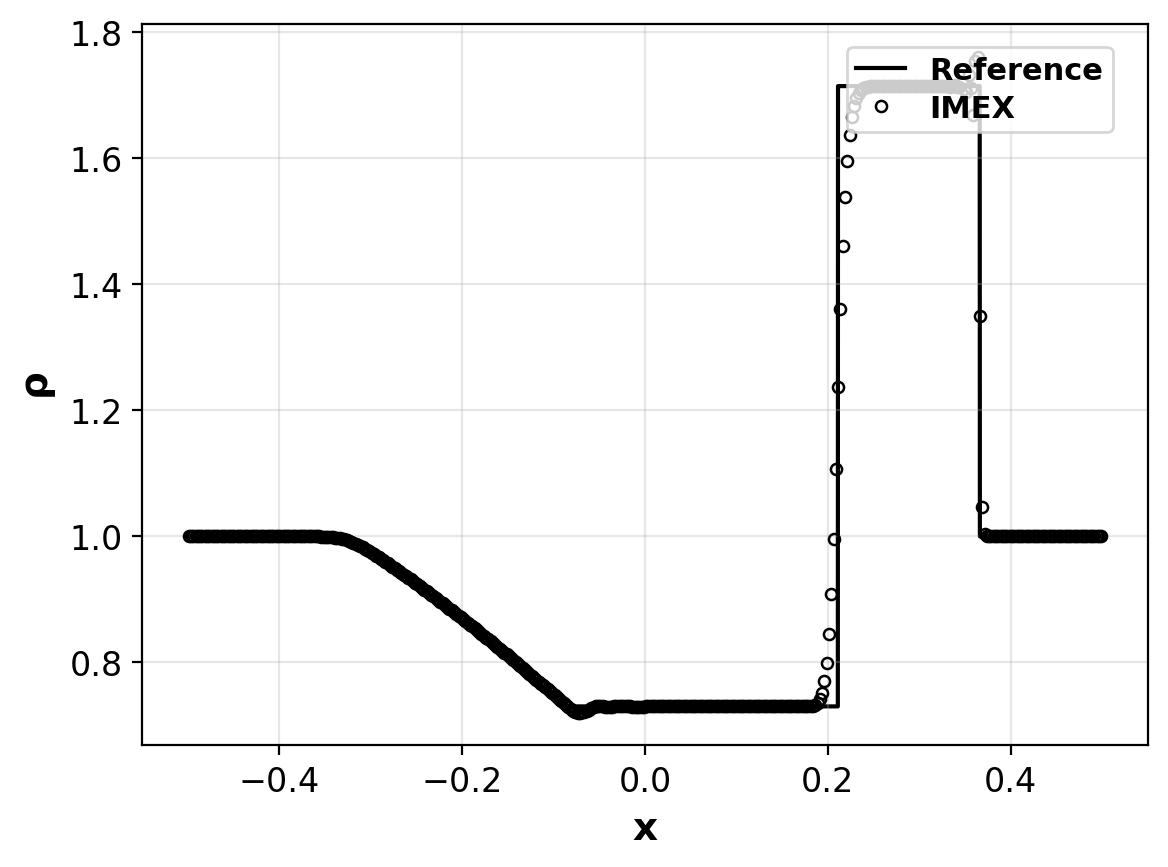}
    \end{minipage}
    \hspace{0.01\linewidth}%
    \begin{minipage}{0.32\linewidth}
        \includegraphics[width=\linewidth]{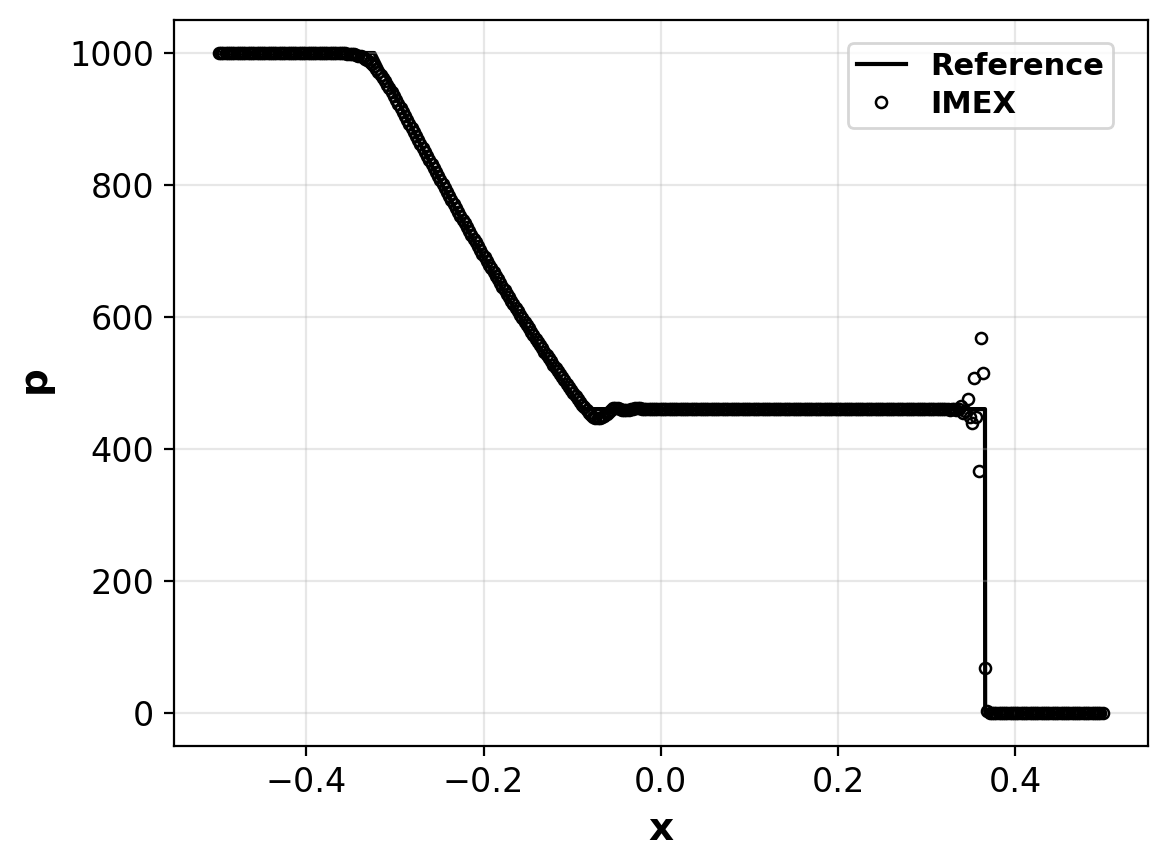}
    \end{minipage}
    \hspace{0.01\linewidth}%
    \begin{minipage}{0.32\linewidth}
        \includegraphics[width=\linewidth]{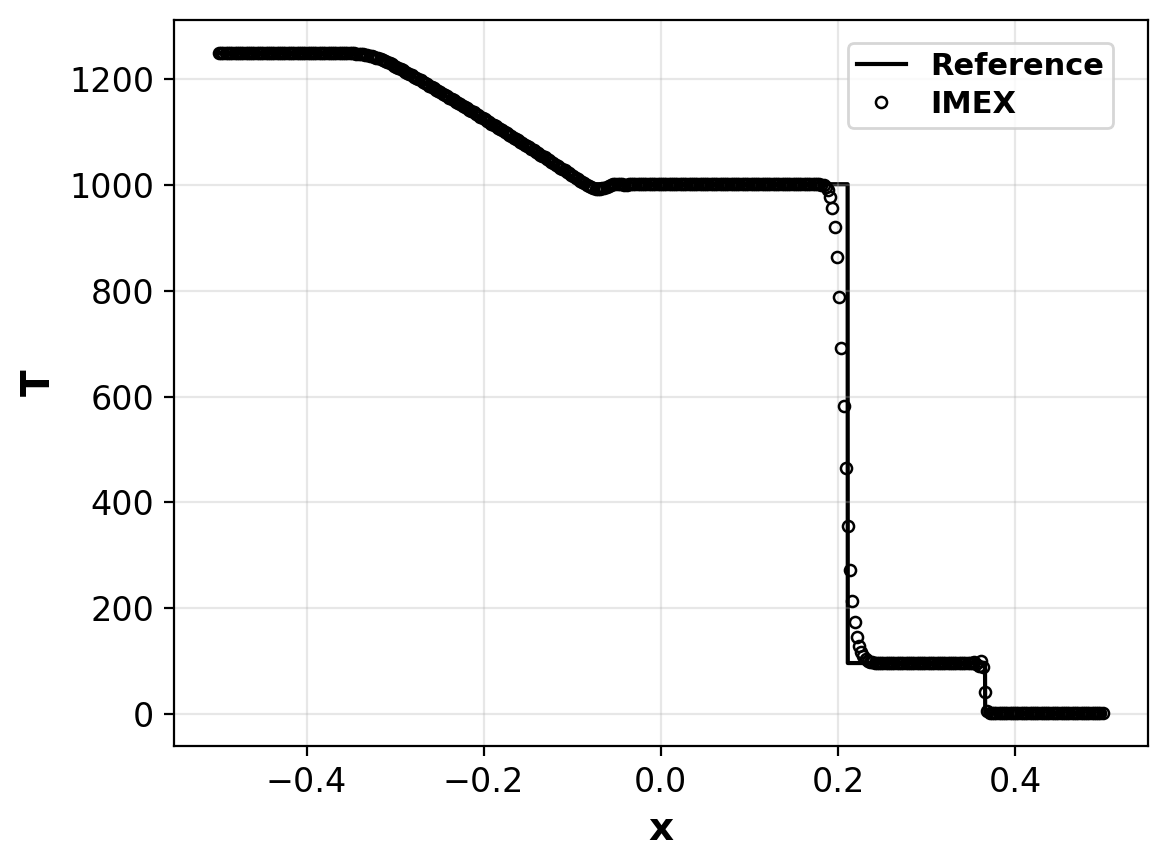}
    \end{minipage}
    
    \vspace{0.5em}

    \begin{minipage}{0.32\linewidth}
        \includegraphics[width=\linewidth]{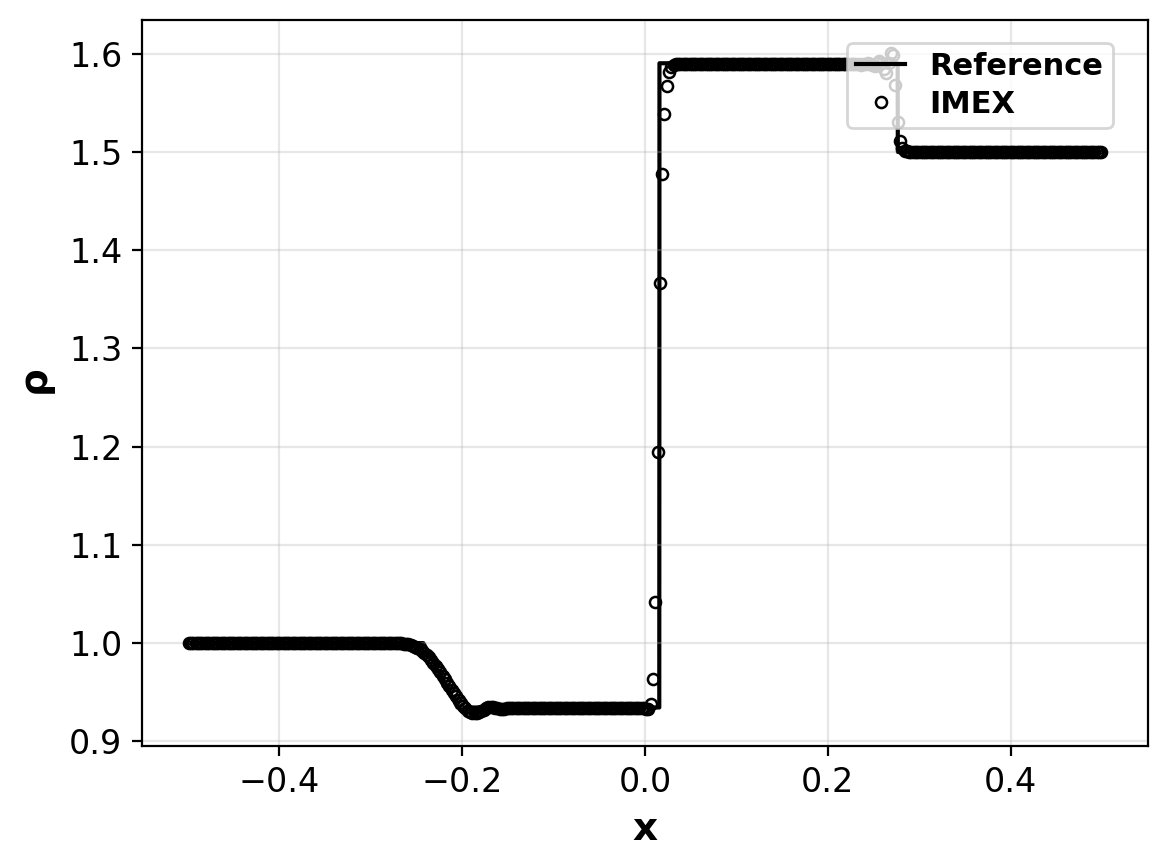}
    \end{minipage}
    \hspace{0.01\linewidth}%
    \begin{minipage}{0.32\linewidth}
        \includegraphics[width=\linewidth]{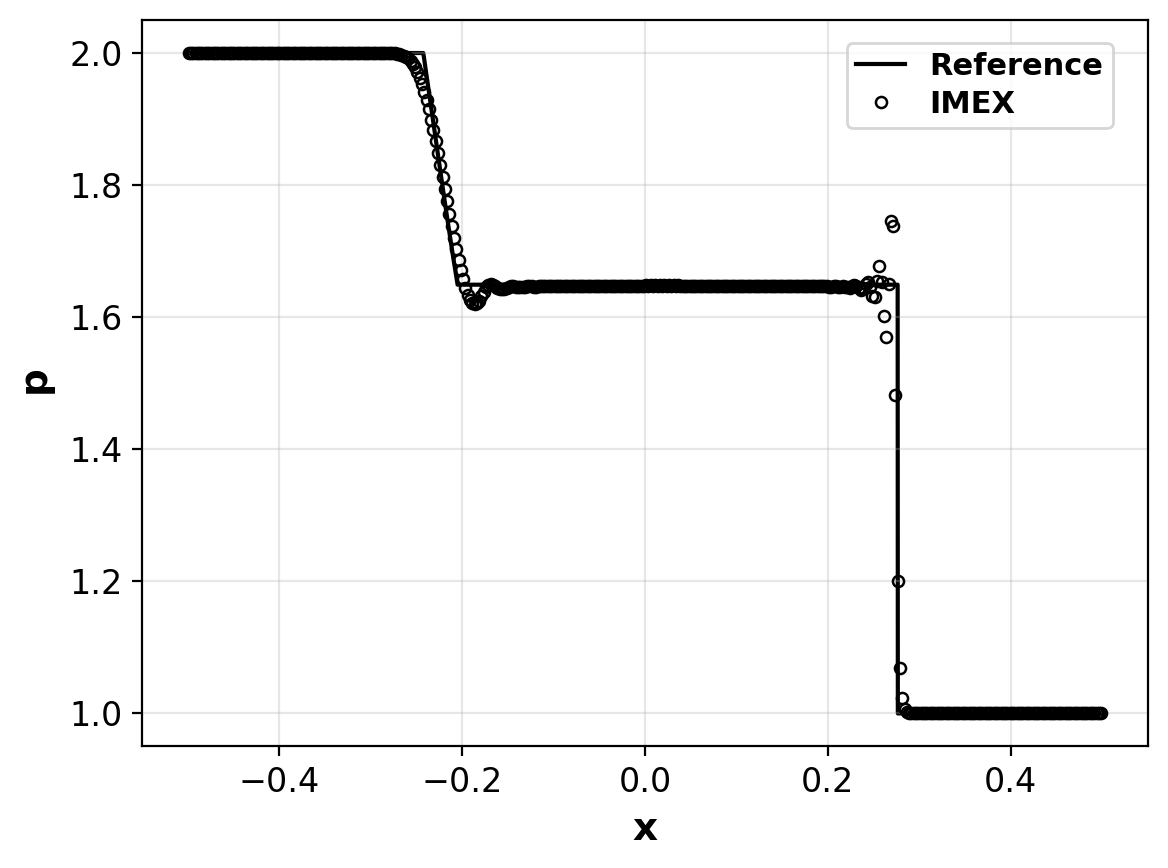}
    \end{minipage}
    \hspace{0.01\linewidth}%
    \begin{minipage}{0.32\linewidth}
        \includegraphics[width=\linewidth]{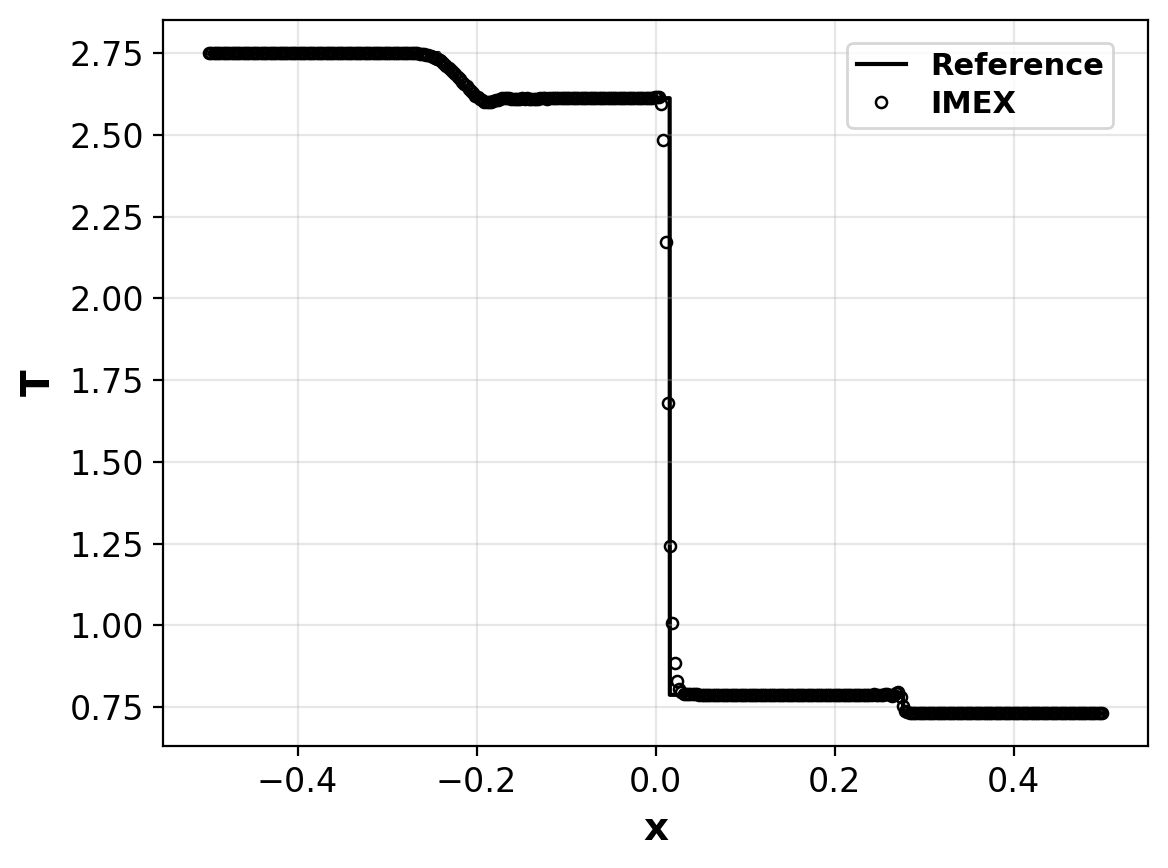}
    \end{minipage}
    \caption{Numerical results for the Riemann problems Redlich--Kwong EOS test 1, test 2, test 3, and test 4 (from top to bottom) obtained using the new semi-implicit scheme (black symbols)  compared to exact solutions (black line). The fluid density (left), the fluid pressure (middle) and the temperature (right) are depicted at the chosen output time for each test.}
    \label{fig:RKEOSRPs}
\end{figure}

We consider the four Riemann problems whose left and right states, together with
the discontinuity position $x_d$, are listed
in Table~\ref{tab:rk_riemann}. Each test is discretised with $N = 400$ cells on the domain $\Omega = [-0.5, 0.5]$, and the
numerical solution is compared against the exact solution of the corresponding
Redlich--Kwong Riemann problem. As shown in
Figs.~\ref{fig:RKEOSRPs}, the computed profiles are in good agreement with the reference throughout, confirming that the generic-EOS machinery reproduces the correct wave structure.

\subsubsection{One-dimension MHD Riemann test of the ideal gas equation of state}
\label{sec:MHD_1d}
We now assess the shock-capturing capabilities of the proposed semi-implicit 
scheme against a set of classical, non-trivial Riemann problems for the ideal 
MHD equations. For each test the initial data consist of two constant states 
$\mathbf{W}_L = [\rho_L, \mathbf{u}_L, p_L, \mathbf{B}_L]$ and 
$\mathbf{W}_R = [\rho_R, \mathbf{u}_R, p_R, \mathbf{B}_R]$ separated by a 
discontinuity located at $x = x_0$. The seven configurations RP1--RP7 are 
collected in Table~\ref{tab:mhd_riemann}, together with the position of the initial 
discontinuity and the final simulation time. The computational domain $[-0.5,0.5]$ 
is discretised with $1000$ cells with transmissive boundaries. These tests span moderate to high acoustic Mach number flows, and 
thus probe the robustness of the scheme across a broad range of regimes.

In Figs.~\ref{fig:MHDRPs1} and~\ref{fig:MHDRPs2} we report the density, the 
pressure and the $y$-component of the magnetic field, compared against the exact 
solution of the Riemann problem, computed with the exact ideal-MHD Riemann solver 
kindly provided by Falle and Komissarov~\cite{exact1,exact2}; an alternative 
exact solver is described in~\cite{exact3}. For every configuration the 
proposed scheme reproduces the correct wave pattern, capturing the shock positions 
and the intermediate states accurately, which confirms its ability to resolve 
shock waves and flow discontinuities. Good agreement is observed with results 
previously reported in the literature for both fully-explicit and other 
semi-implicit schemes. Two 
configurations deserve a specific comment. RP1 is the well-known Brio--Wu 
test~\cite{BrioWu1988}, whose density profile exhibits an artificial compound wave; 
the associated overshoot in density and pressure is not specific to the present 
method, but a common feature of standard finite volume schemes for ideal MHD, 
also arising in approximate solutions produced with the 
HLLEM~\cite{DumbserHLLEM}, Osher-type~\cite{DumbserToro} and Rusanov schemes. 
RP7 contains an isolated, steady Alfvén wave, which is well resolved; the small 
perturbations of density and pressure visible around the jump position are of 
the same order of magnitude as those produced by the explicit scheme and by 
other IMEX methods~\cite{BoscheriThomann2024}.

Being a semi-implicit method, the present scheme naturally introduces more 
numerical dissipation than an explicit approximate Riemann solver. The reason is 
structural: its time step is dictated solely by the advective (material) 
velocity $\mathbf{u}$, rather than by the fast magnetosonic speeds that would 
otherwise restrict a fully explicit update. It is precisely this decoupling that 
lets the step size grow well past the explicit stability limit. To quantify the 
gain, Figs.~\ref{fig:timeofRPs} compares the material step $\Delta t$ effectively 
employed by the scheme with the explicit magnetosonic step 
$\Delta t_{\mathrm{MHD}}$ fixed by the fastest signal speeds. Although most of the seven 
configurations lie in a genuinely compressible regime, the admissible $\Delta t$ 
still exceeds $\Delta t_{\mathrm{MHD}}$ by a clear margin; the trade-off is a 
mild smearing of the fast waves, while the material contact and shear structures 
remain sharply resolved. We finally observe that, whenever the velocity satisfies 
a discrete maximum principle, this ratio stays essentially constant over the 
entire simulation.
\begin{table}[htbp]
\centering
\caption{Initial left (L) and right (R) states for the MHD Riemann test
problems. The discontinuity is located at $x_0$. }
\label{tab:mhd_riemann}
\setlength{\tabcolsep}{3.5pt}
\renewcommand{\arraystretch}{1.15}
\resizebox{\textwidth}{!}{%
\begin{tabular}{llllllllllll}   
\toprule
Test & & $\rho$ & $u$ & $v$ & $w$ & $B_x$ & $B_y$ & $B_z$ & $p$ & $x_0$ & $t_{\mathrm{out}}$ \\
\midrule
\multirow{2}{*}{RP1}
 & L: & $1.0$   & $0.0$ & $0.0$ & $0.0$ & $0.75$ & $1.0$  & $0.0$ & $1.0$ & \multirow{2}{*}{$0.0$} & \multirow{2}{*}{$0.1$} \\
 & R: & $0.125$ & $0.0$ & $0.0$ & $0.0$ & $0.75$ & $-1.0$ & $0.0$ & $0.1$ & & \\
\addlinespace
\multirow{2}{*}{RP2}
 & L: & $1.08$   & $1.2$     & $0.01$   & $0.5$      & $0.564190$ & $1.015541$ & $0.564190$ & $0.95$    & \multirow{2}{*}{$-0.1$} & \multirow{2}{*}{$0.2$} \\
 & R: & $0.9891$ & $-0.0131$ & $0.0269$ & $0.010037$ & $0.564190$ & $1.135262$ & $0.564923$ & $0.97159$ & & \\
\addlinespace
\multirow{2}{*}{RP3}
 & L: & $1.7$ & $0.0$ & $0.0$ & $0.0$       & $1.1$ & $1.0$      & $0.0$      & $1.7$ & \multirow{2}{*}{$-0.1$} & \multirow{2}{*}{$0.15$} \\
 & R: & $0.2$ & $0.0$ & $0.0$ & $-1.496891$ & $1.1$ & $0.785887$ & $0.618370$ & $0.2$ & & \\
\addlinespace
\multirow{2}{*}{RP4}
 & L: & $1.0$ & $0.0$ & $0.0$ & $0.0$ & $1.3$ & $1.0$  & $0.0$ & $1.0$ & \multirow{2}{*}{$0.0$} & \multirow{2}{*}{$0.16$} \\
 & R: & $0.4$ & $0.0$ & $0.0$ & $0.0$ & $1.3$ & $-1.0$ & $0.0$ & $0.4$ & & \\
\multirow{2}{*}{RP5}
 & L: & $0.15$ & $21.55$  & $1.0$ & $1.0$ & $0.05/\sqrt{4\pi}$ & $-2.0/\sqrt{4\pi}$ & $-1.0/\sqrt{4\pi}$ & $0.28$ & \multirow{2}{*}{$0.0$} & \multirow{2}{*}{$0.04$} \\
 & R: & $0.1$  & $-26.45$ & $0.0$ & $0.0$ & $0.05/\sqrt{4\pi}$ & $2.0/\sqrt{4\pi}$  & $1.0/\sqrt{4\pi}$  & $0.1$  & & \\
\addlinespace
\multirow{2}{*}{RP6}
 & L: & $1.0$ & $36.87$  & $-0.115$ & $-0.0386$ & $4/\sqrt{4\pi}$ & $4/\sqrt{4\pi}$ & $1/\sqrt{4\pi}$ & $1.0$ & \multirow{2}{*}{$0.0$} & \multirow{2}{*}{$0.03$} \\
 & R: & $1.0$ & $-36.87$ & $0.0$    & $0.0$     & $4/\sqrt{4\pi}$ & $4/\sqrt{4\pi}$ & $1/\sqrt{4\pi}$ & $1.0$ & & \\
\addlinespace
\multirow{2}{*}{RP7}
 & L: & $1/(4\pi)$ & $-1.0$ & $1.0$  & $-1.0$ & $1/\sqrt{4\pi}$ & $-1/\sqrt{4\pi}$ & $1/\sqrt{4\pi}$ & $1.0$ & \multirow{2}{*}{$0.0$} & \multirow{2}{*}{$0.25$} \\
 & R: & $1/(4\pi)$ & $-1.0$ & $-1.0$ & $-1.0$ & $1/\sqrt{4\pi}$ & $1/\sqrt{4\pi}$  & $1/\sqrt{4\pi}$ & $1.0$ & & \\
\bottomrule
\end{tabular}}

\vspace{2pt}
\end{table}

\begin{figure}
    \centering
    \begin{minipage}{0.32\linewidth}
        \includegraphics[width=\linewidth]{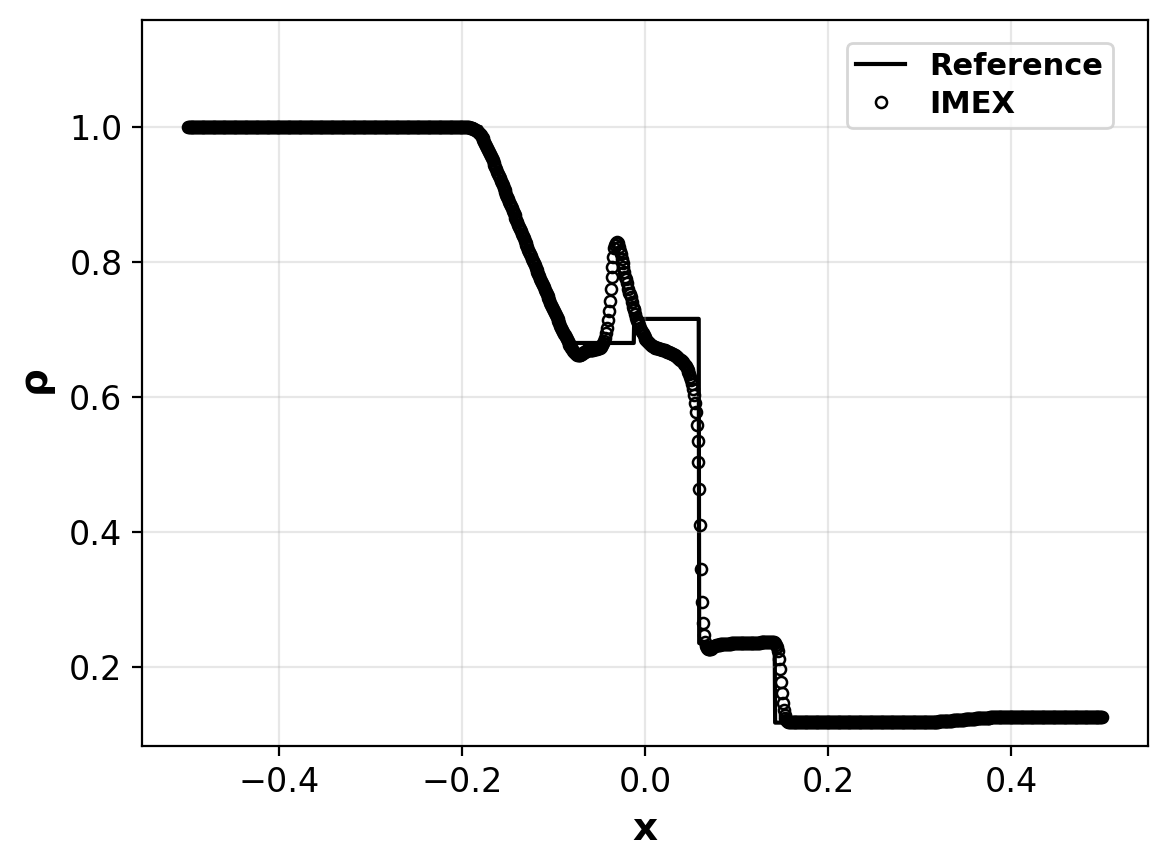}
    \end{minipage}
    \hspace{0.01\linewidth}%
    \begin{minipage}{0.32\linewidth}
        \includegraphics[width=\linewidth]{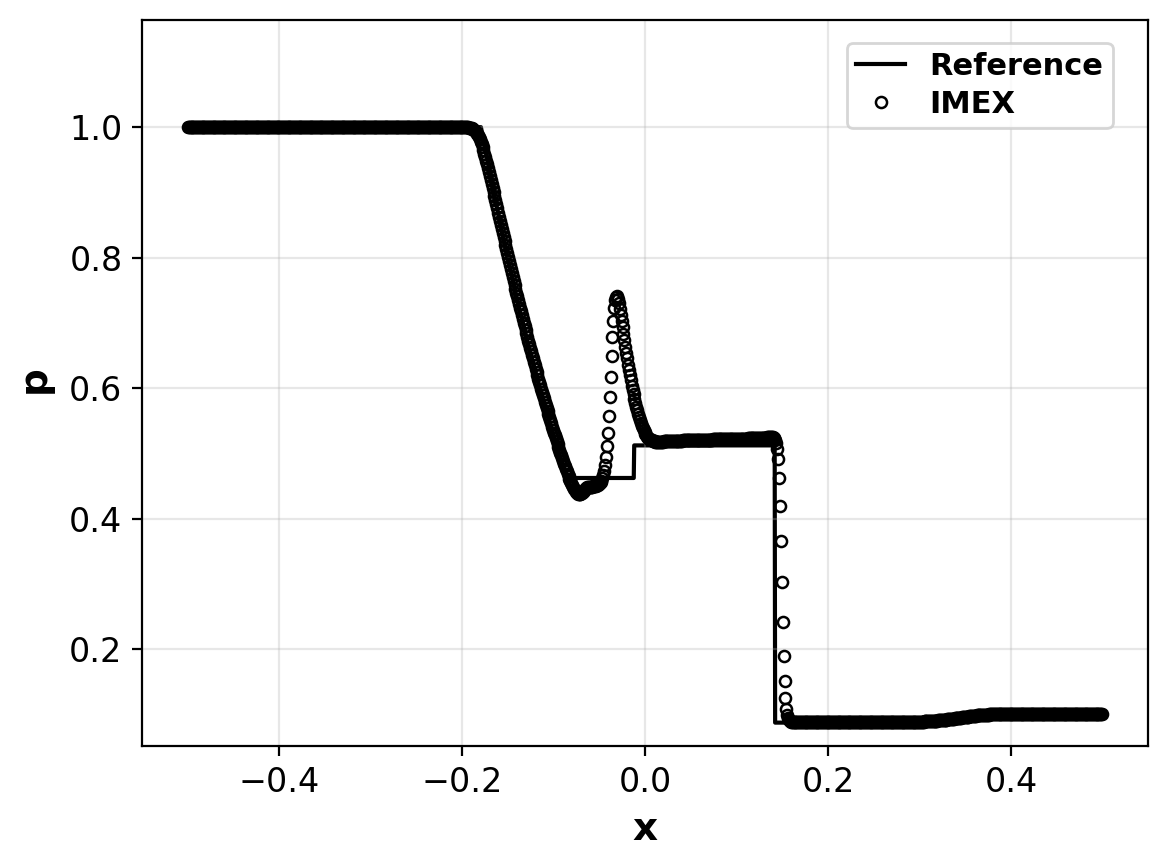}
    \end{minipage}
    \hspace{0.01\linewidth}%
    \begin{minipage}{0.32\linewidth}
        \includegraphics[width=\linewidth]{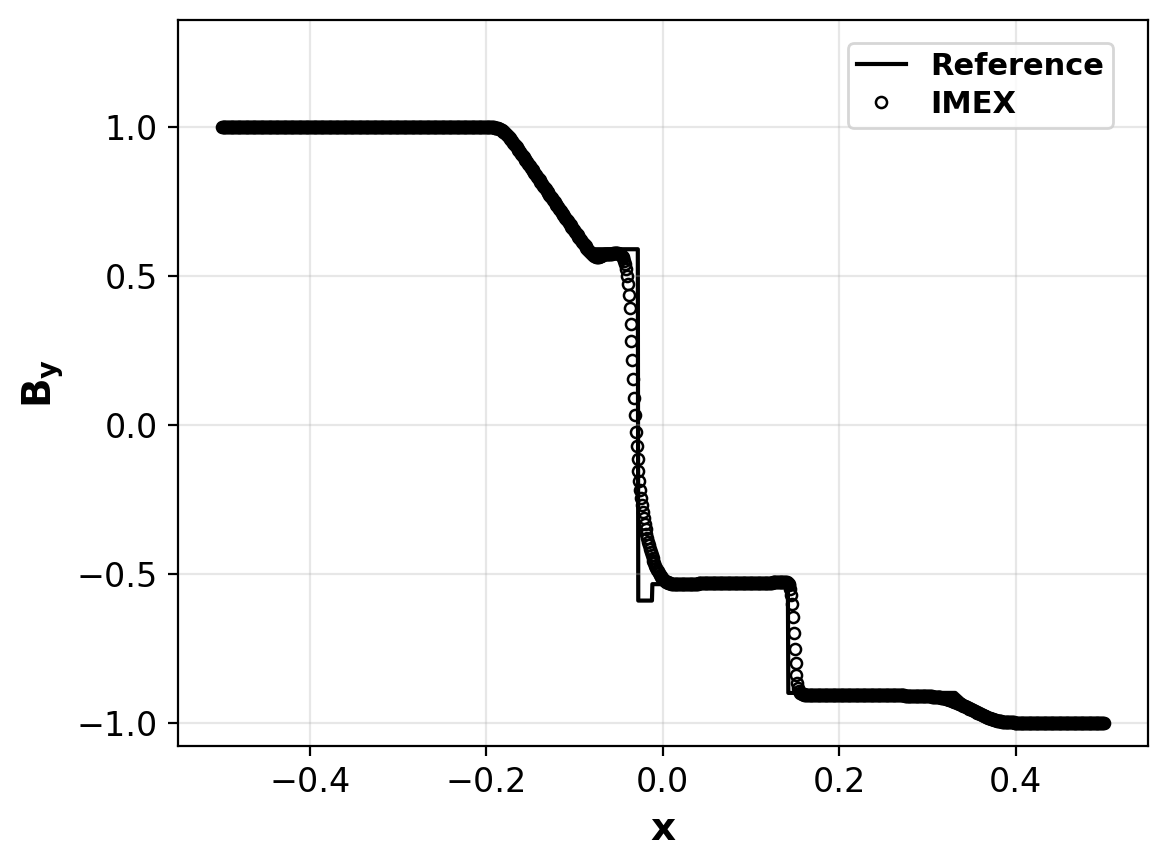}
    \end{minipage}
    \vspace{0.5em}

    \begin{minipage}{0.32\linewidth}
        \includegraphics[width=\linewidth]{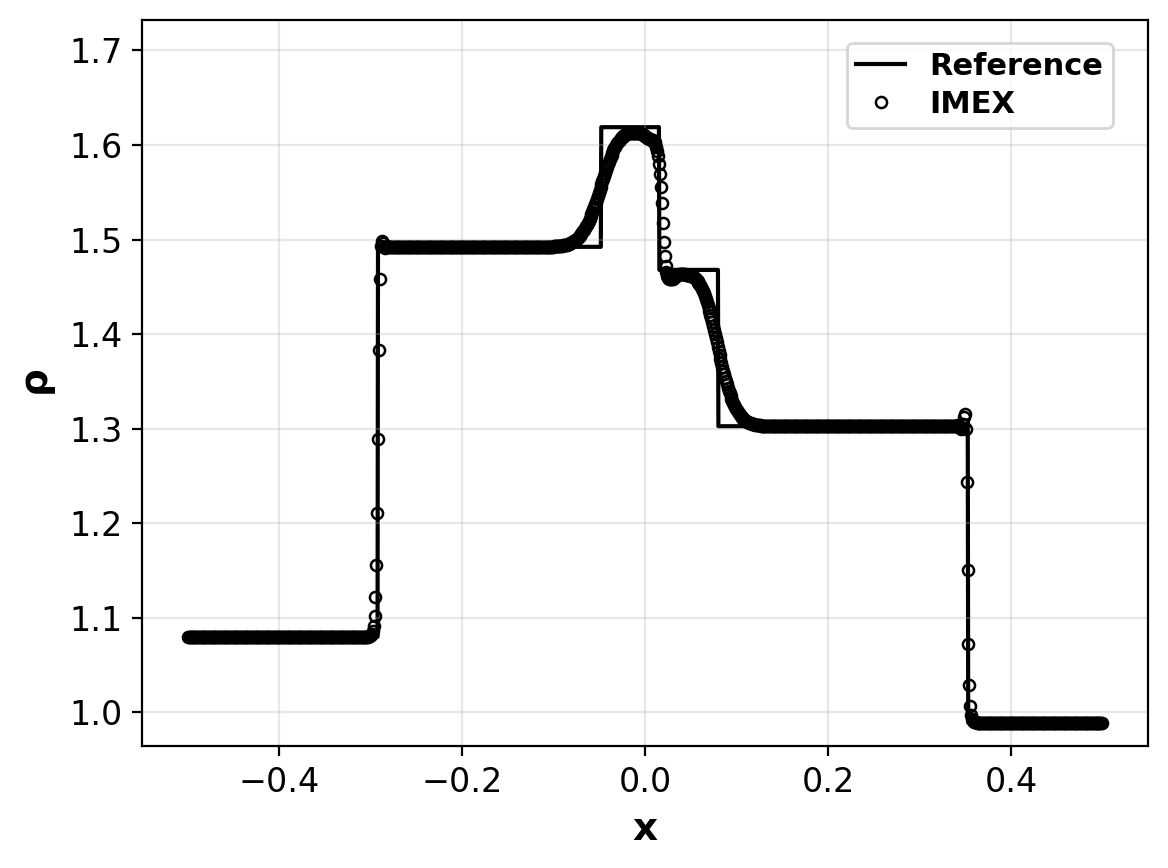}
    \end{minipage}
    \hspace{0.01\linewidth}%
    \begin{minipage}{0.32\linewidth}
        \includegraphics[width=\linewidth]{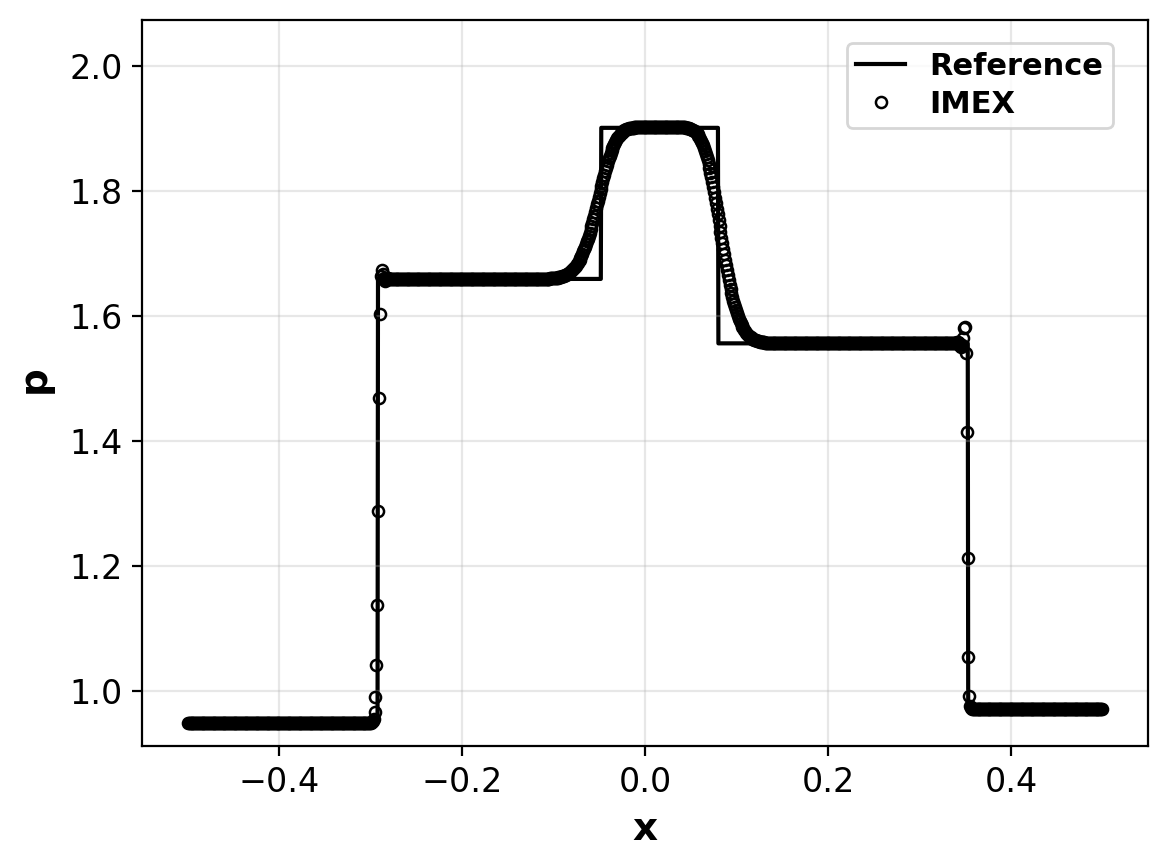}
    \end{minipage}
    \hspace{0.01\linewidth}%
    \begin{minipage}{0.32\linewidth}
        \includegraphics[width=\linewidth]{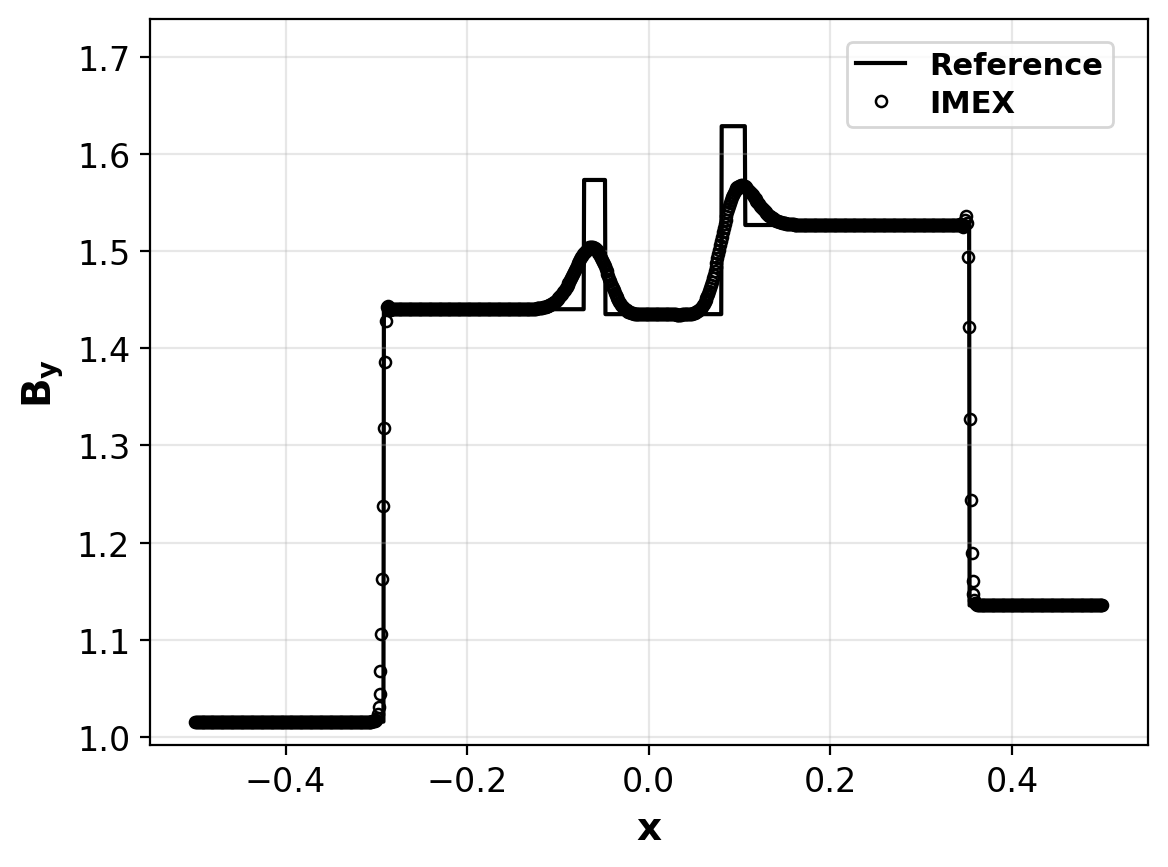}
    \end{minipage}
    \vspace{0.5em}
    \begin{minipage}{0.32\linewidth}
        \includegraphics[width=\linewidth]{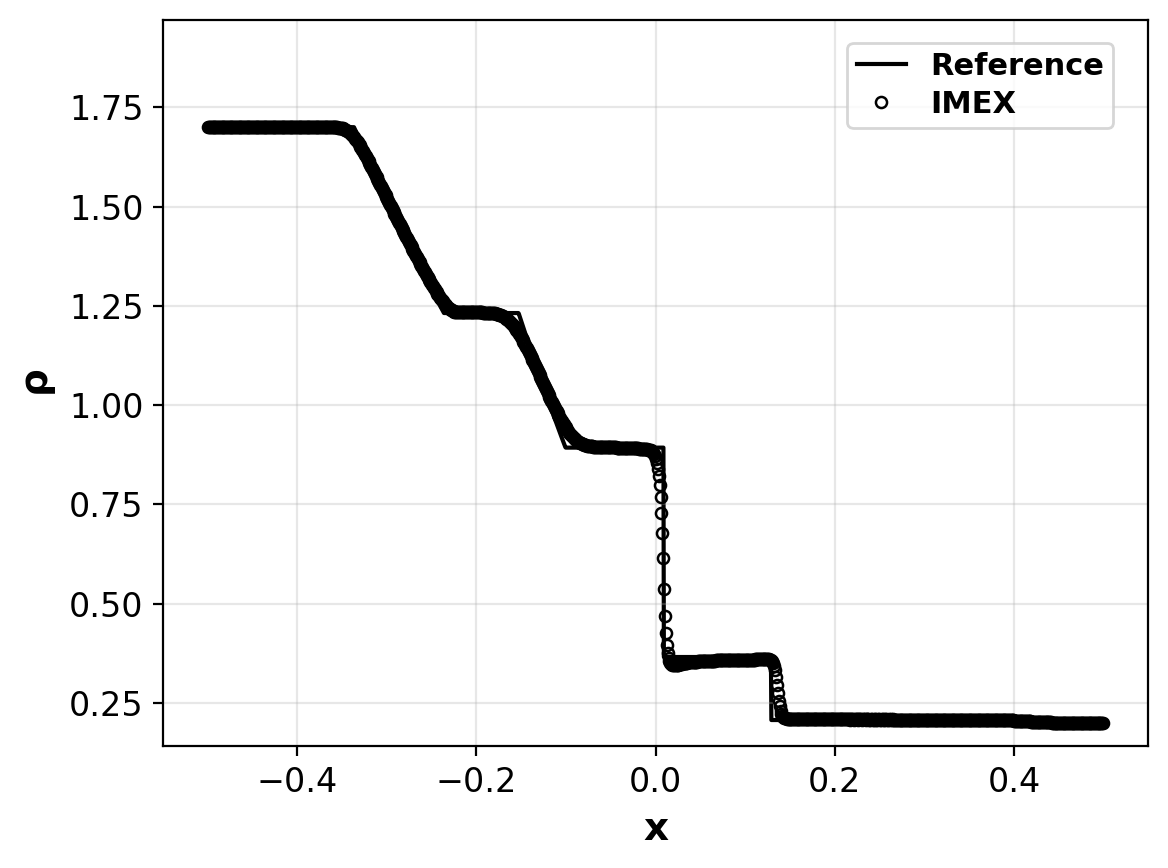}
    \end{minipage}
    \hspace{0.01\linewidth}%
    \begin{minipage}{0.32\linewidth}
        \includegraphics[width=\linewidth]{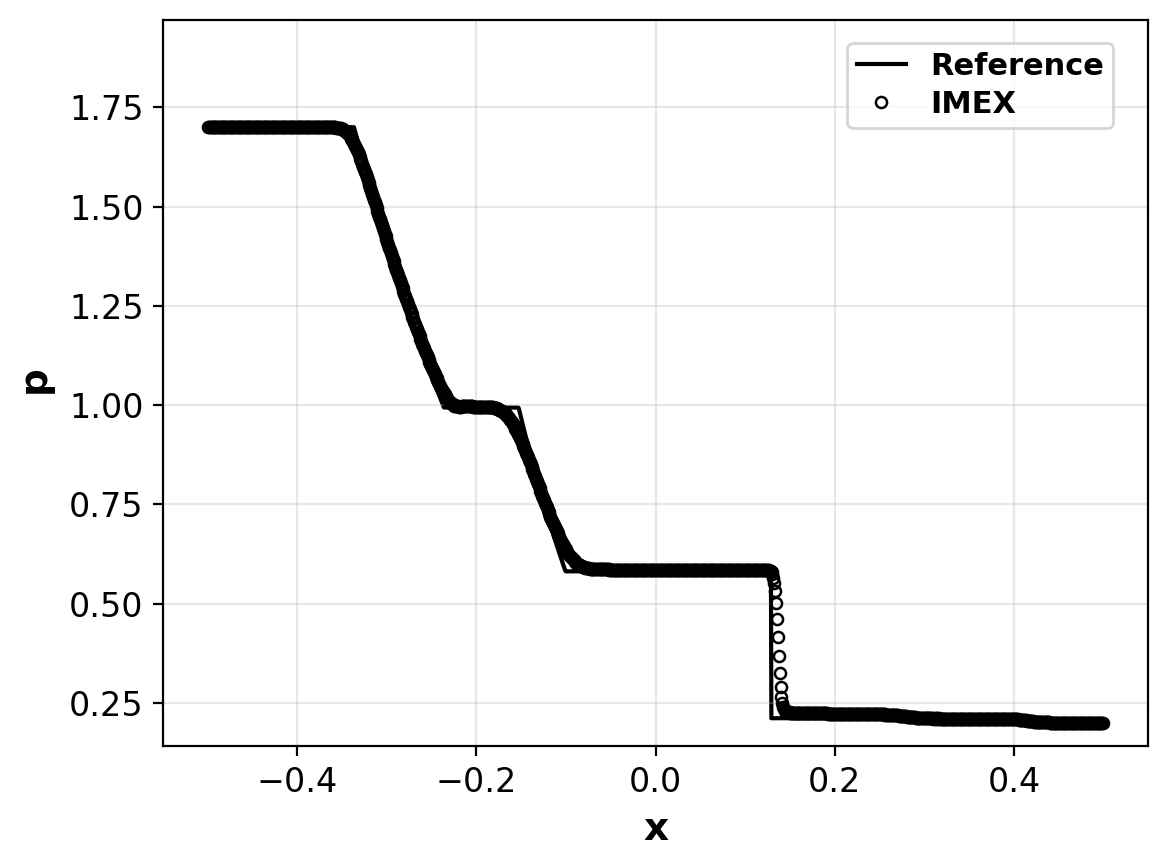}
    \end{minipage}
    \hspace{0.01\linewidth}%
    \begin{minipage}{0.32\linewidth}
        \includegraphics[width=\linewidth]{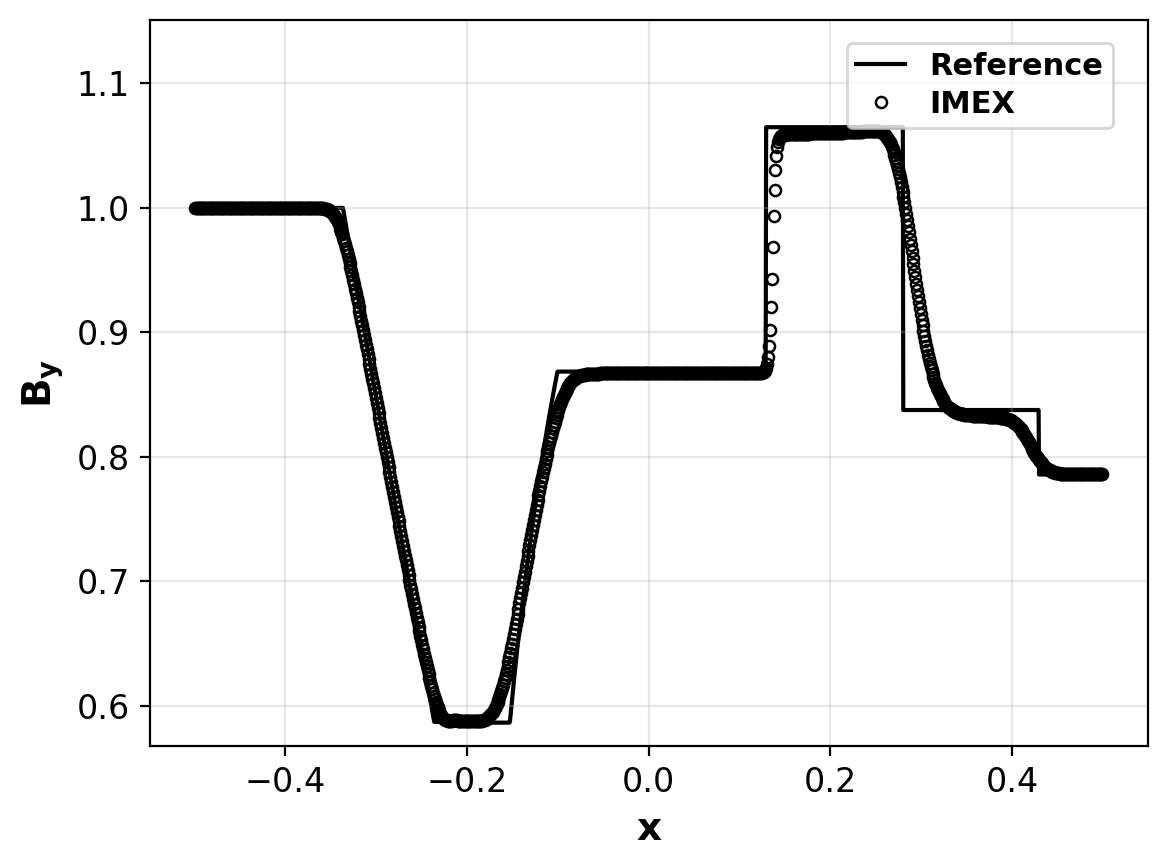}
    \end{minipage}
    
    \vspace{0.5em}

    \begin{minipage}{0.32\linewidth}
        \includegraphics[width=\linewidth]{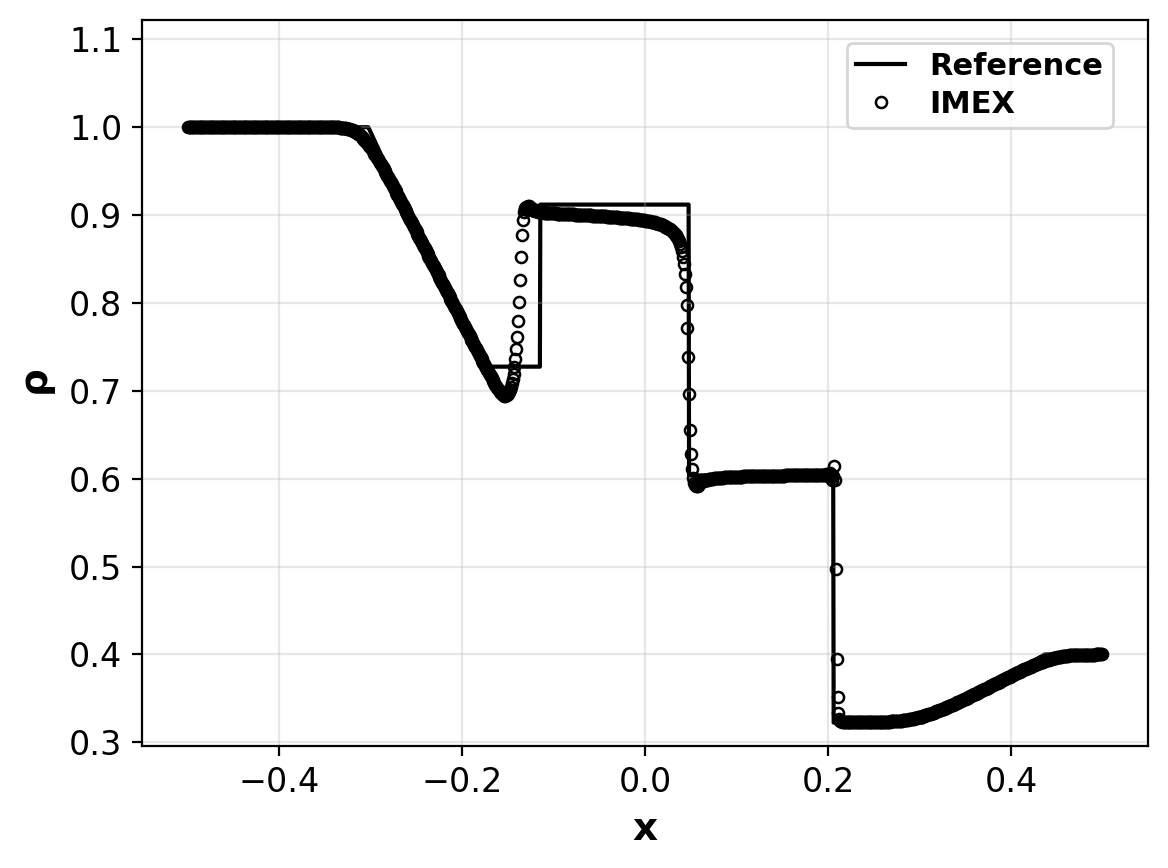}
    \end{minipage}
    \hspace{0.01\linewidth}%
    \begin{minipage}{0.32\linewidth}
        \includegraphics[width=\linewidth]{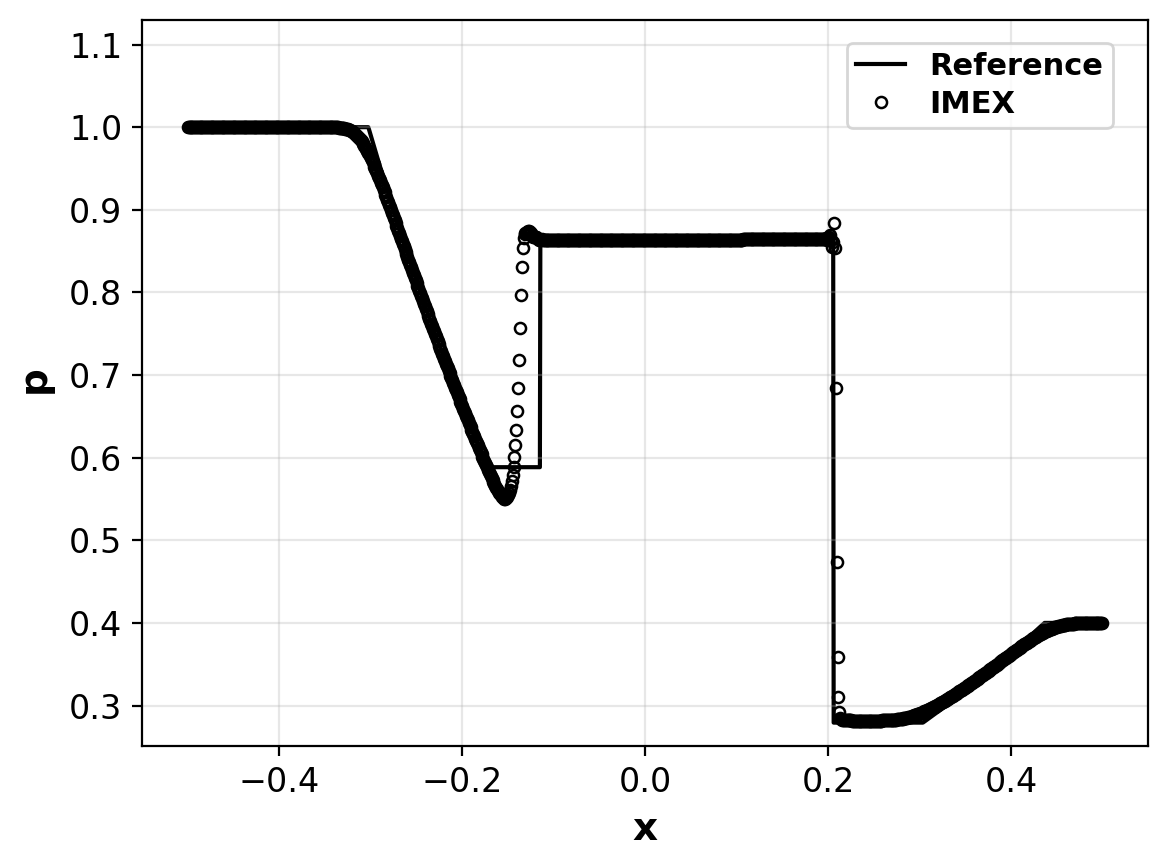}
    \end{minipage}
    \hspace{0.01\linewidth}%
    \begin{minipage}{0.32\linewidth}
        \includegraphics[width=\linewidth]{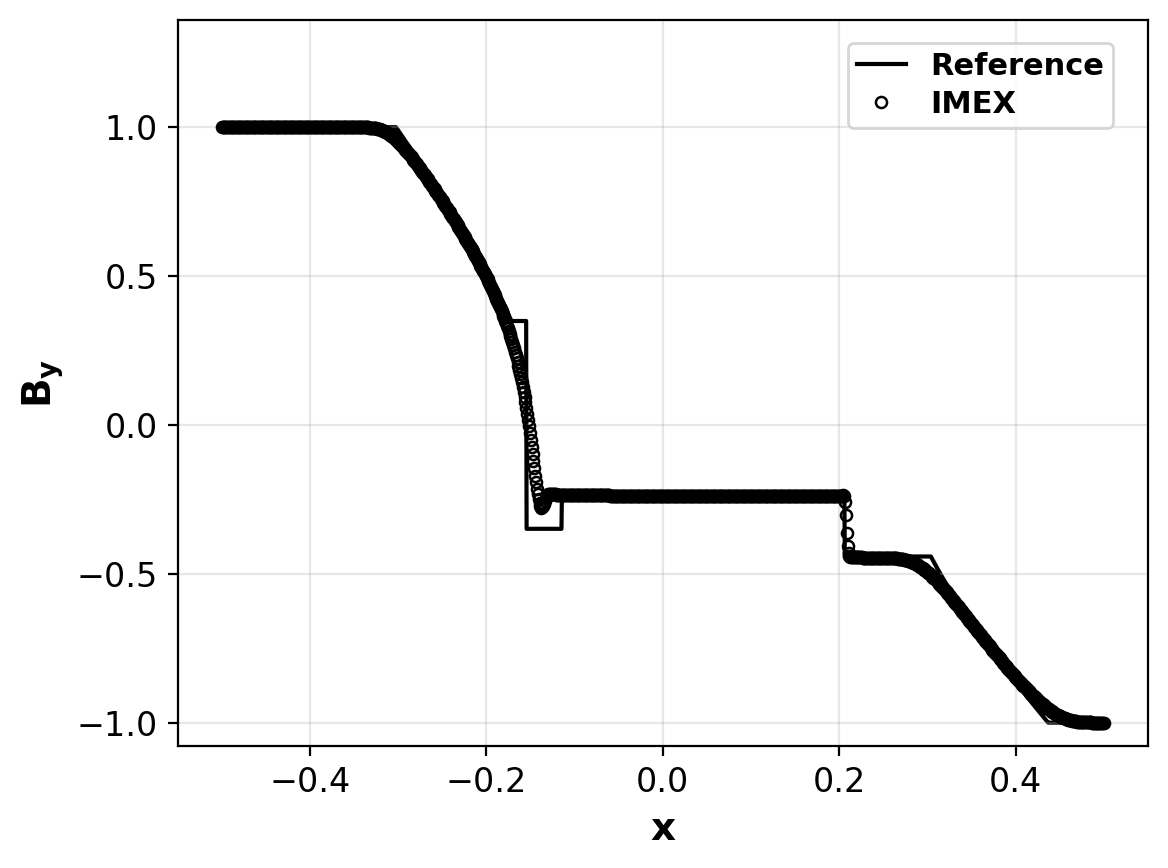}
    \end{minipage}

    \caption{Numerical results for the Riemann problems RP1, RP2, RP3 and RP4 (from top to bottom) obtained using the new semi-implicit scheme compared to exact Riemann solutions (black line). The black symbols indicate results obtained by proposed IMEX scheme.
    The fluid density (left), the fluid pressure (middle) and the component of the magnetic field (right) are depicted at the chosen output time for each test.}
    \label{fig:MHDRPs1}
\end{figure}

\begin{figure}
    \centering
    \begin{minipage}{0.32\linewidth}
        \includegraphics[width=\linewidth]{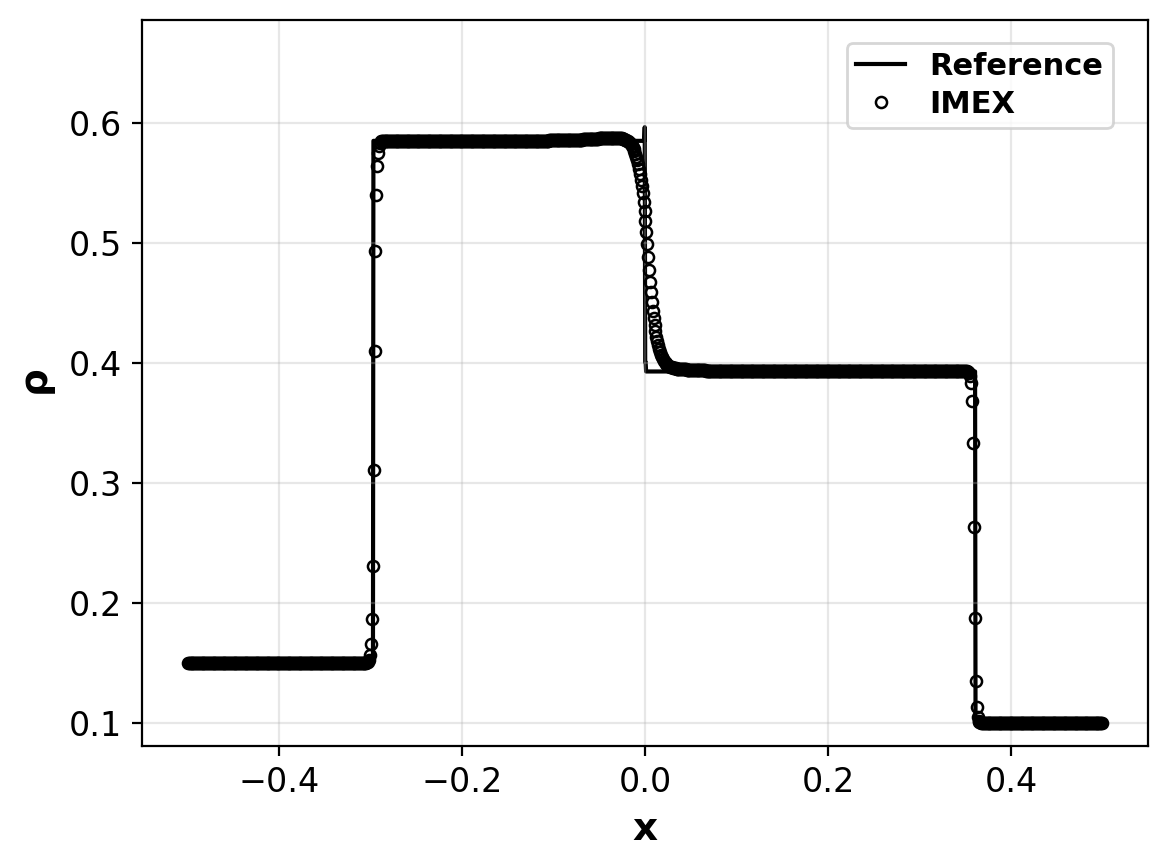}
    \end{minipage}
    \hspace{0.01\linewidth}%
    \begin{minipage}{0.32\linewidth}
        \includegraphics[width=\linewidth]{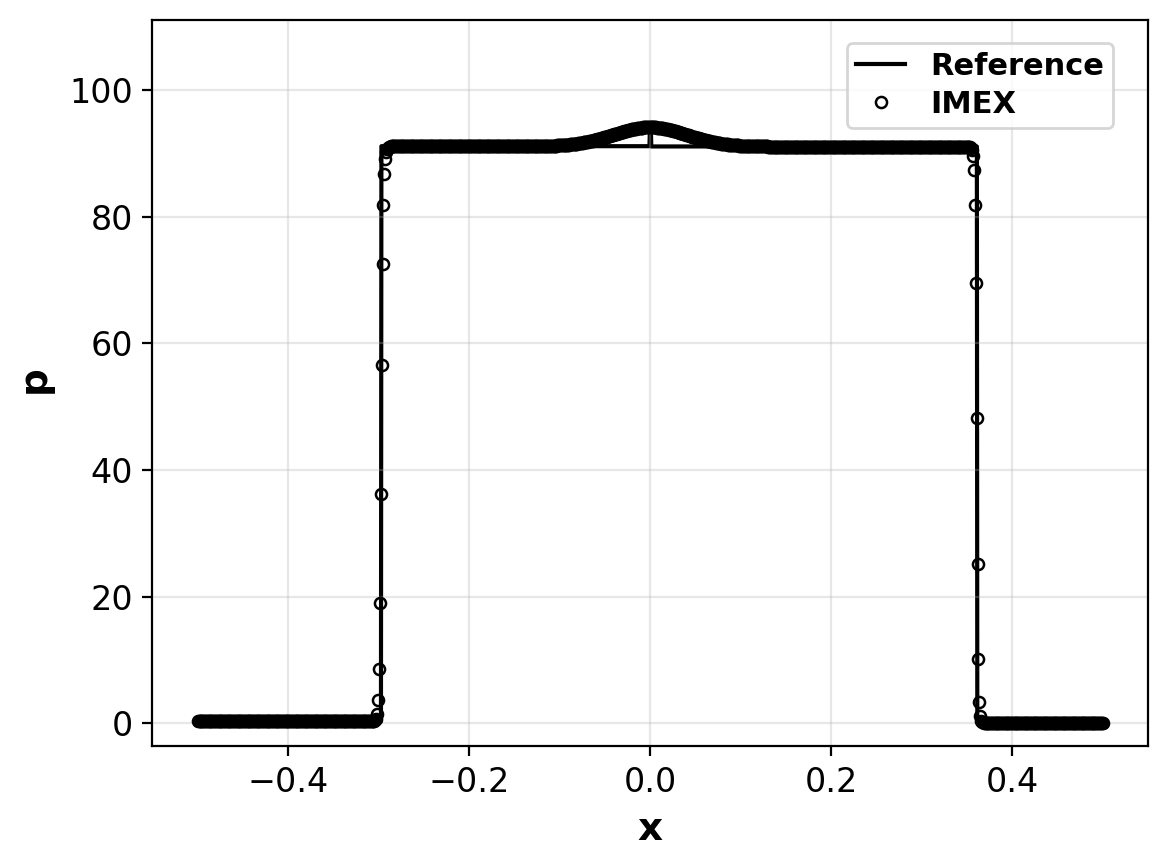}
    \end{minipage}
    \hspace{0.01\linewidth}%
    \begin{minipage}{0.32\linewidth}
        \includegraphics[width=\linewidth]{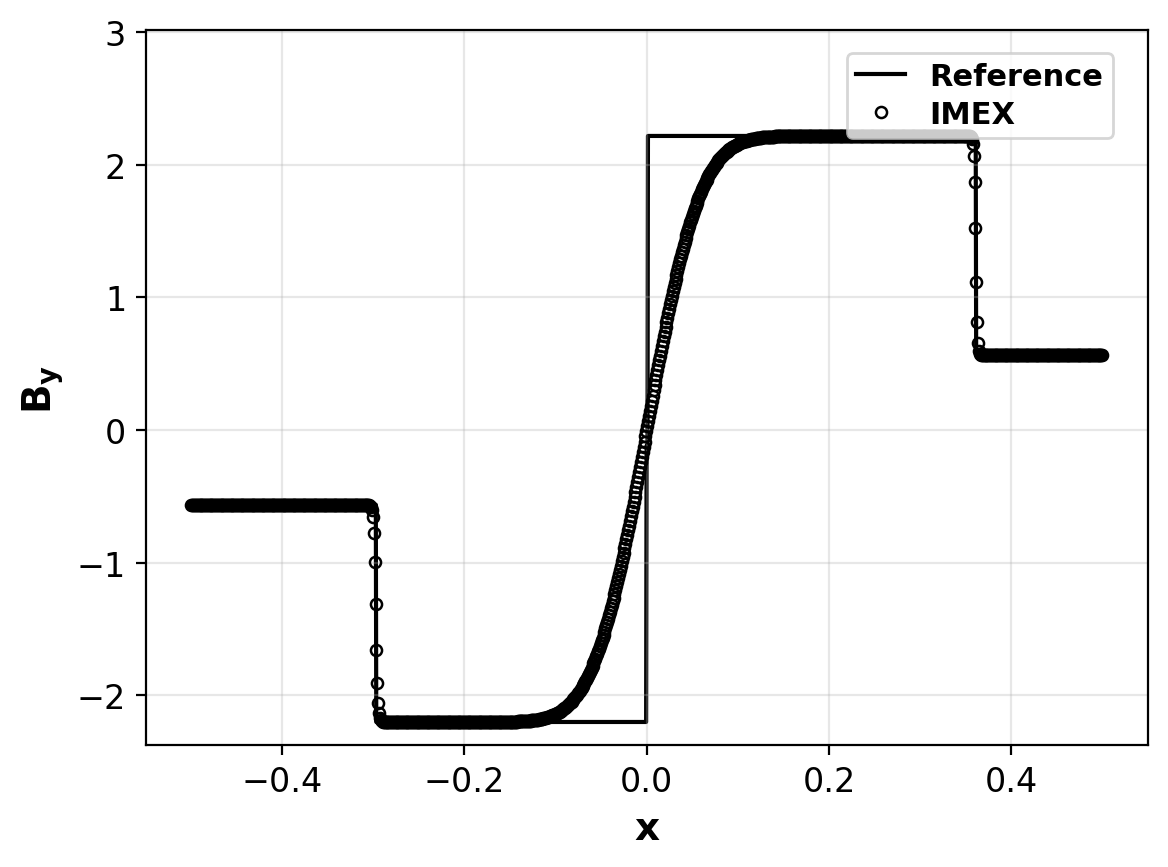}
    \end{minipage}
    \vspace{0.5em}

    \begin{minipage}{0.32\linewidth}
        \includegraphics[width=\linewidth]{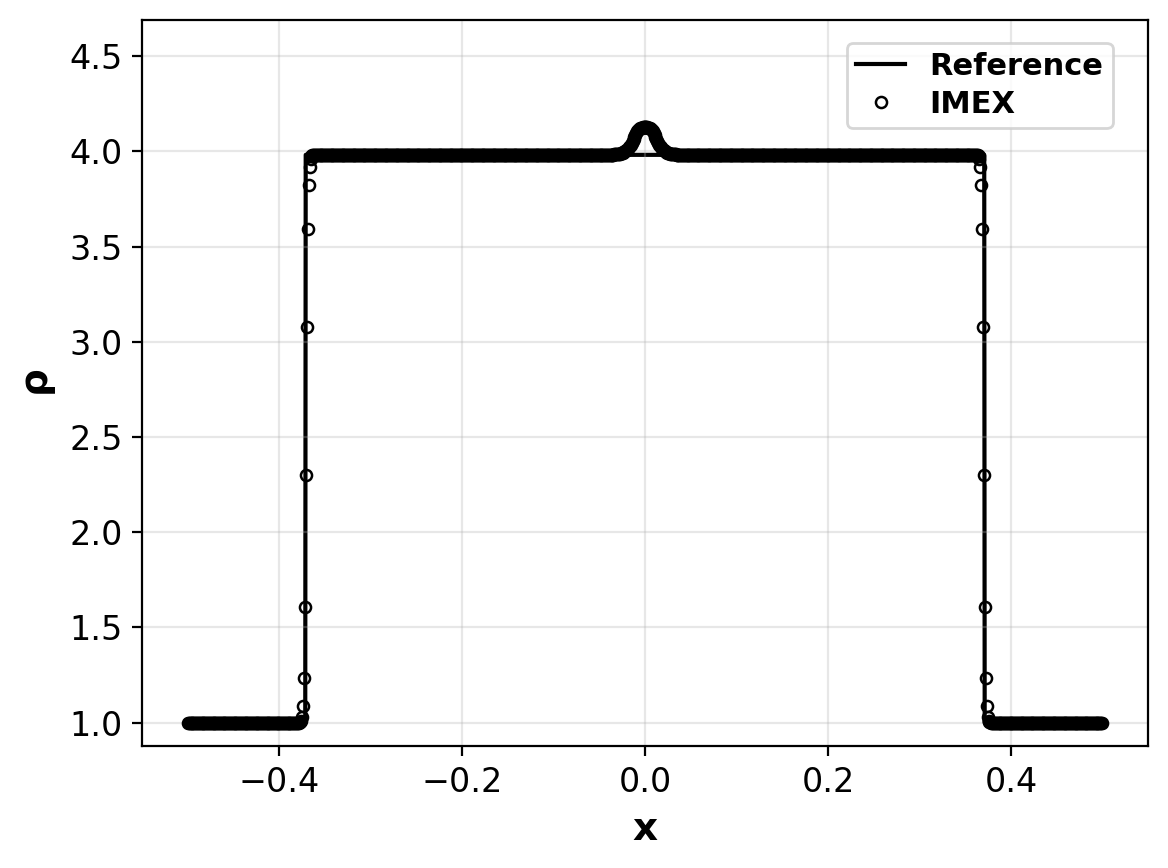}
    \end{minipage}
    \hspace{0.01\linewidth}%
    \begin{minipage}{0.32\linewidth}
        \includegraphics[width=\linewidth]{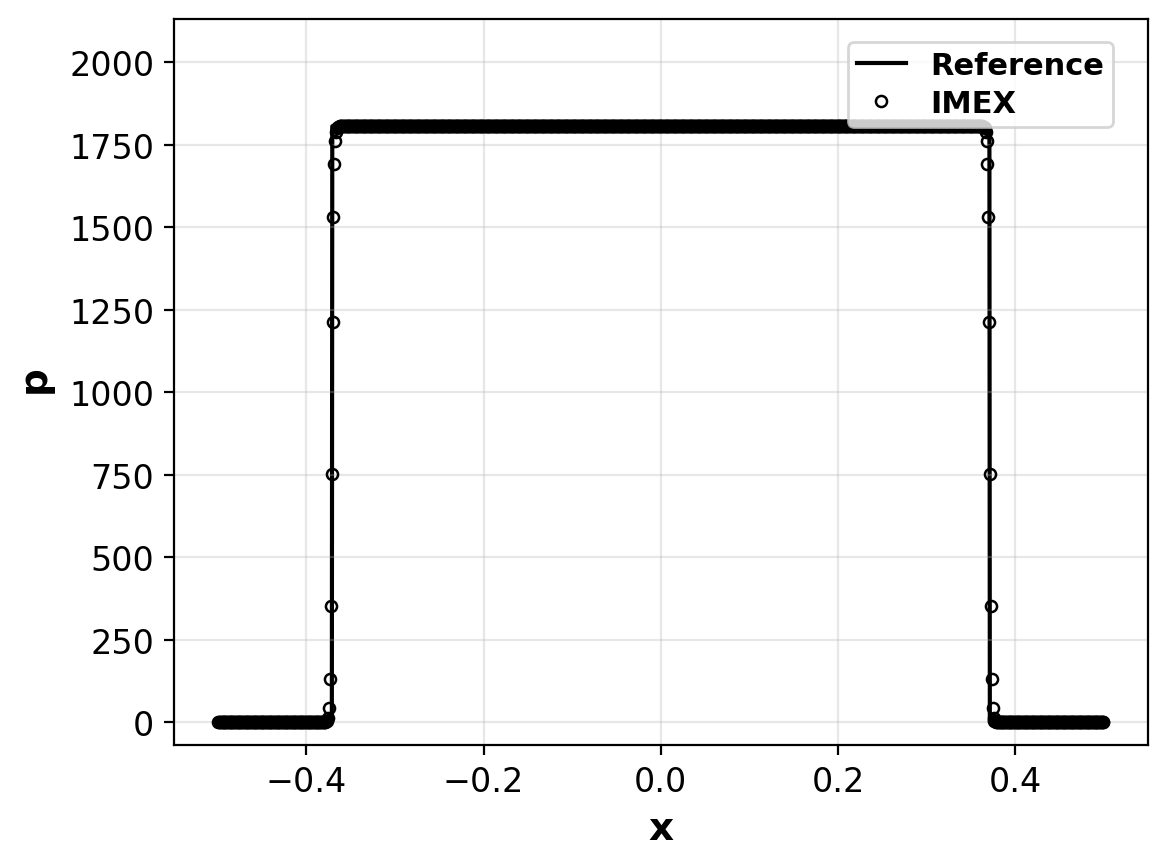}
    \end{minipage}
    \hspace{0.01\linewidth}%
    \begin{minipage}{0.32\linewidth}
        \includegraphics[width=\linewidth]{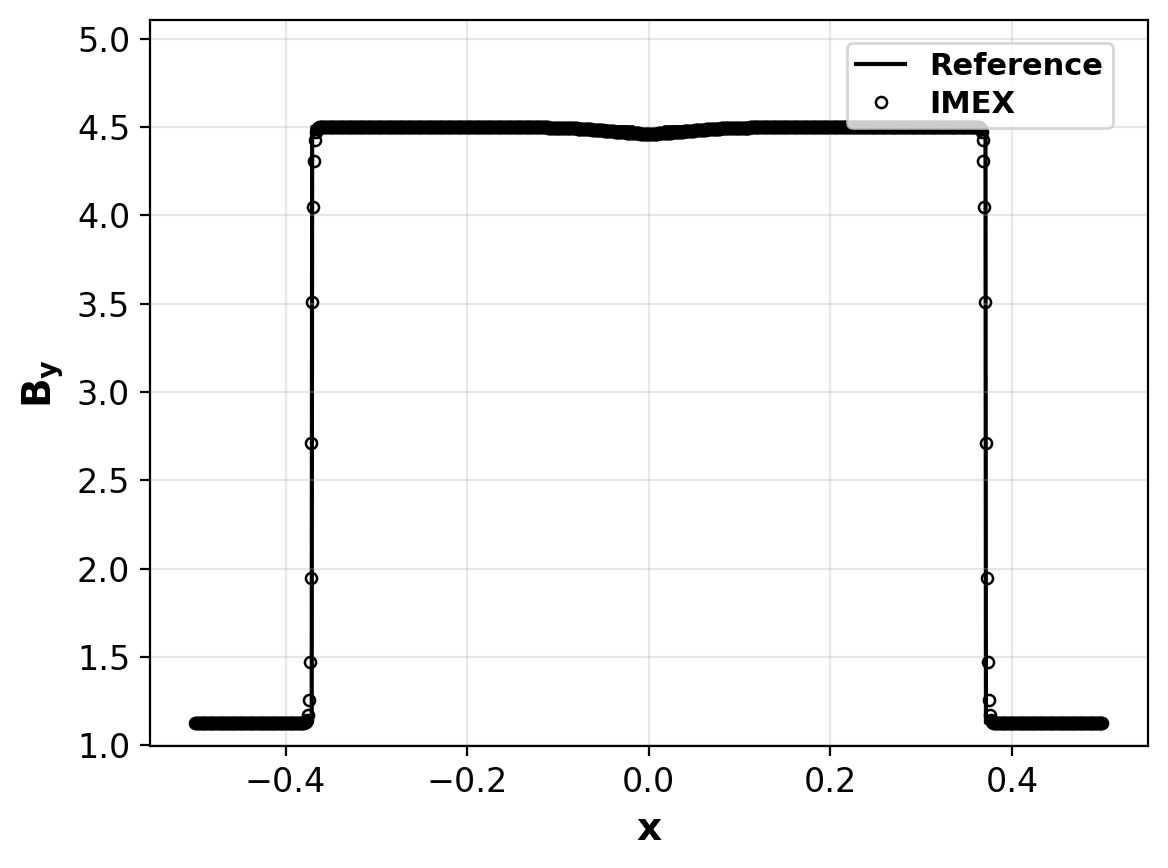}
    \end{minipage}
    \vspace{0.5em}
    \begin{minipage}{0.32\linewidth}
        \includegraphics[width=\linewidth]{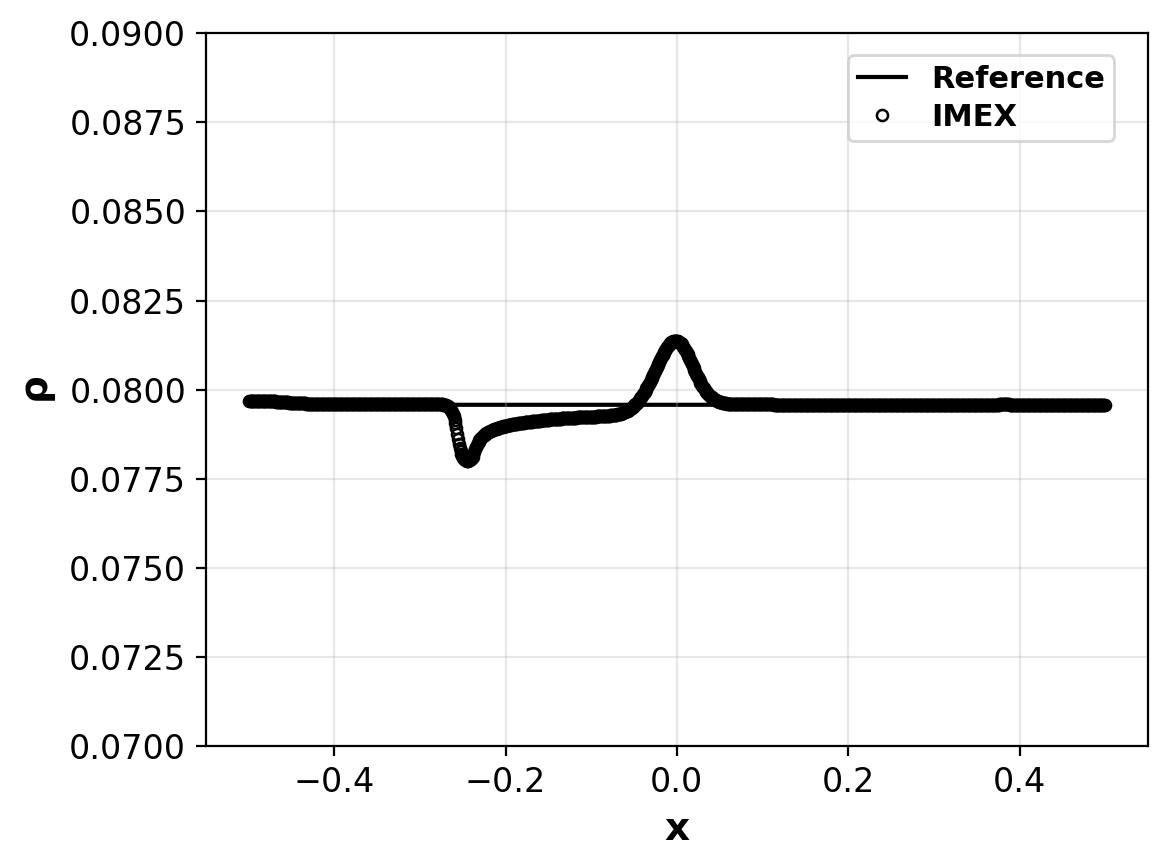}
    \end{minipage}
    \hspace{0.01\linewidth}%
    \begin{minipage}{0.32\linewidth}
        \includegraphics[width=\linewidth]{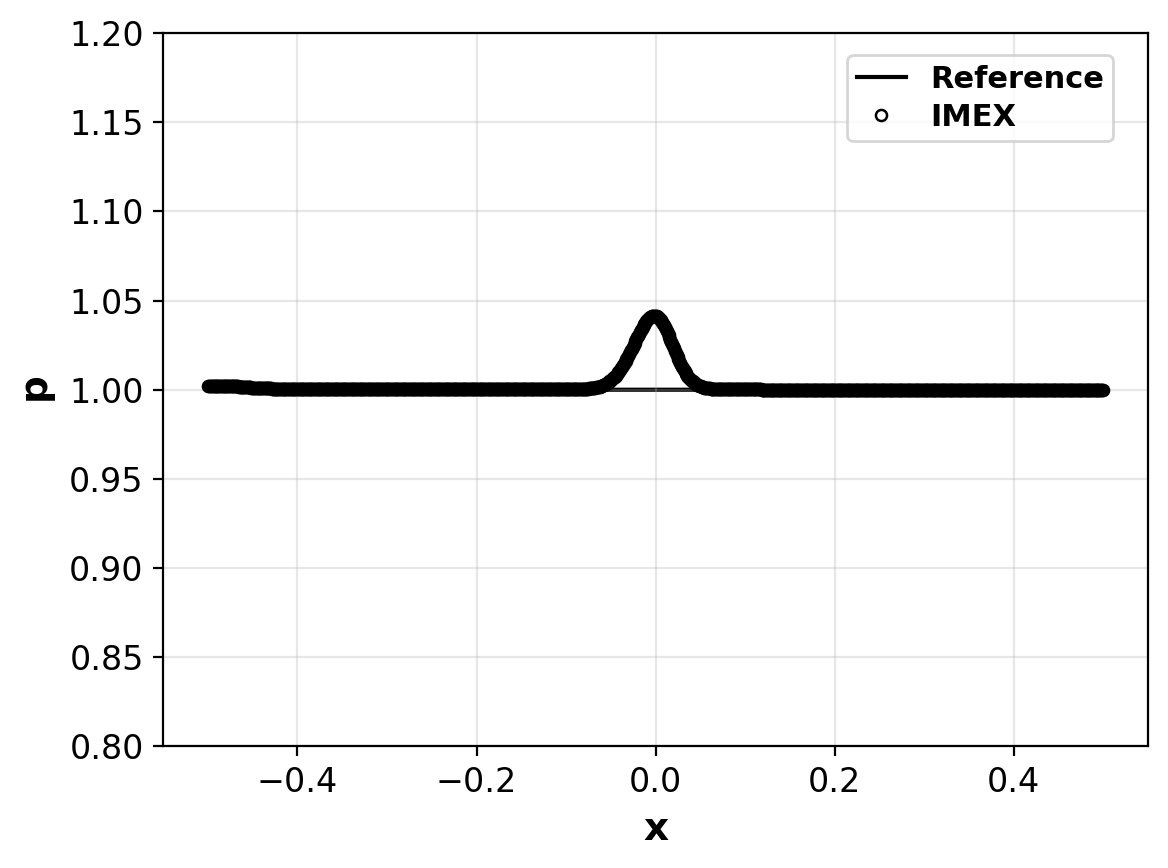}
    \end{minipage}
    \hspace{0.01\linewidth}%
    \begin{minipage}{0.32\linewidth}
        \includegraphics[width=\linewidth]{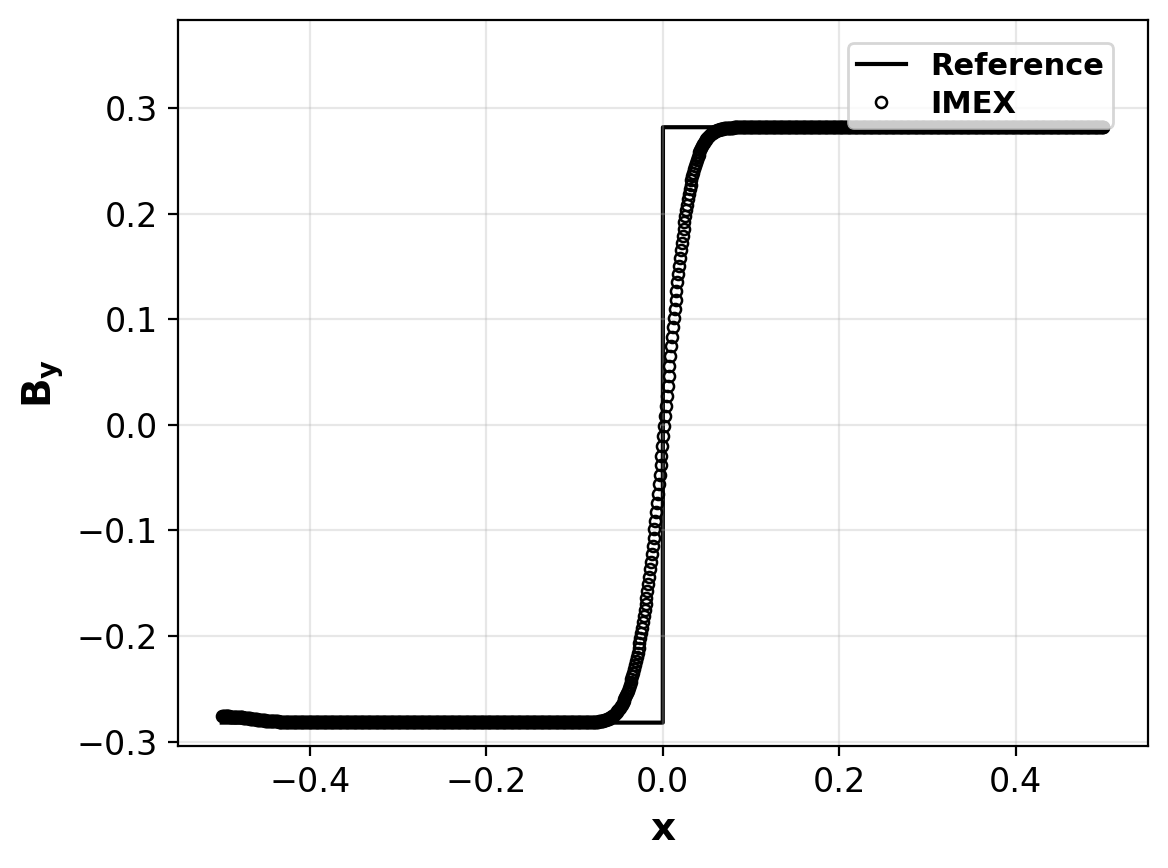}
    \end{minipage}
    \caption{Numerical results for the Riemann problems RP5, RP6 and RP7 (from top to bottom) obtained using the new semi-implicit scheme compared to exact Riemann solutions (black line). The black symbols indicate results obtained by proposed IMEX scheme.
    The fluid density (left), the fluid pressure (middle) and the y-component of the magnetic field (right) are depicted at the chosen output time for each test.}
    \label{fig:MHDRPs2}
\end{figure}

\begin{figure}
    \centering
    \includegraphics[width=0.6\linewidth]{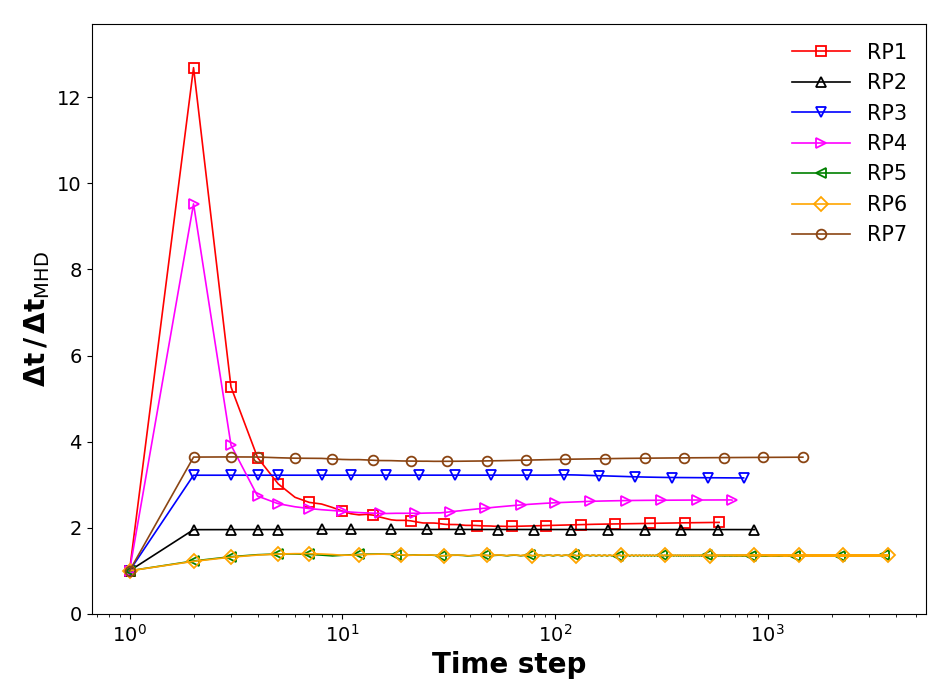}
    \caption{Time evolution of the ratio between the material time step $\Delta t$ of the present semi-implicit IMEX scheme and the magnetosonic time step of a fully explicit finite volume scheme for Riemann problems RP1–RP7.}
    \label{fig:timeofRPs}
\end{figure}

\subsection{Smooth stationary MHD vortex}
\label{chp:vortex_longtime}
 
We now consider a stationary, isodensity variant of the vortex of
Section~\ref{sec:convergence}, run to long simulation times following
Fambri~\cite{Fambri2021}. The initial condition
is given by
\begin{equation}
\label{eq:vortex_ic}
\begin{bmatrix} \rho_0 \\ u_0 \\ v_0 \\ (B_x)_0 \\ (B_y)_0 \\ p_0 \end{bmatrix}
=
\begin{bmatrix}
1 \\[2pt]
-\dfrac{\varepsilon}{2\pi}\, e^{\frac12(1-r^2)}\, y \\[6pt]
\phantom{-}\dfrac{\varepsilon}{2\pi}\, e^{\frac12(1-r^2)}\, x \\[6pt]
-\dfrac{1}{2\pi}\, e^{\frac12(1-r^2)}\, y \\[6pt]
\phantom{-}\dfrac{1}{2\pi}\, e^{\frac12(1-r^2)}\, x \\[6pt]
1 + \dfrac{1}{8\pi}\big(\tfrac{\tilde{\mu}}{2\pi}\big)^{2}(1-r^2)\,e^{(1-r^2)}
  - \dfrac12 \big(\tfrac{\varepsilon}{2\pi}\big)^{2} e^{(1-r^2)}
\end{bmatrix},
\end{equation}
with $\varepsilon = 1$ and $\tilde{\mu} = \sqrt{4\pi}$. The computational
domain is the periodic square $[-10,10]\times[-10,10]$, discretised using
$200$ cells in each coordinate direction. The stationary vortex was run up to
a final time $t=1000$ using a CFL-type stability condition only based on the
fluid velocity, with a $\mathrm{CFL}=0.45$.
 
Numerical results for the normalised magnetic field
$\mathbf{b} = \mathbf{B}/\max\lVert\mathbf{B}_0\rVert$ obtained with the
proposed semi-implicit scheme using the unstaggered CT method are shown in
Fig.~\ref{fig:MHDvortex} at times $t=500$ and $t=1000$. A one-dimensional
line-out through the vortex centre of the obtained normalised magnetic field
$\mathbf{b}$ is then presented in Fig.~\ref{fig:vortex_lineout} at time
$t=1000$. In this figure, the numerical solution obtained from the present
algorithm using the staggered CT method is compared to the solution obtained using the unstaggered CT presented in Section~\ref{chp:divfree_CT}, and in turn both are compared to the
reference (initial) solution. We observe that the proposed scheme, using
either methodology for controlling the divergence constraint, is able to
preserve the stationarity of the vortex for this long time simulation, with
the magnetic field divergence remaining at machine precision throughout the
entire run. The results are in good agreement with those obtained by
Fambri~\cite{Fambri2021} and Dematt\'e et al.~\cite{Dematte2024}.

\begin{figure}
    \centering
    \begin{minipage}{0.48\linewidth}
        \includegraphics[width=\linewidth]{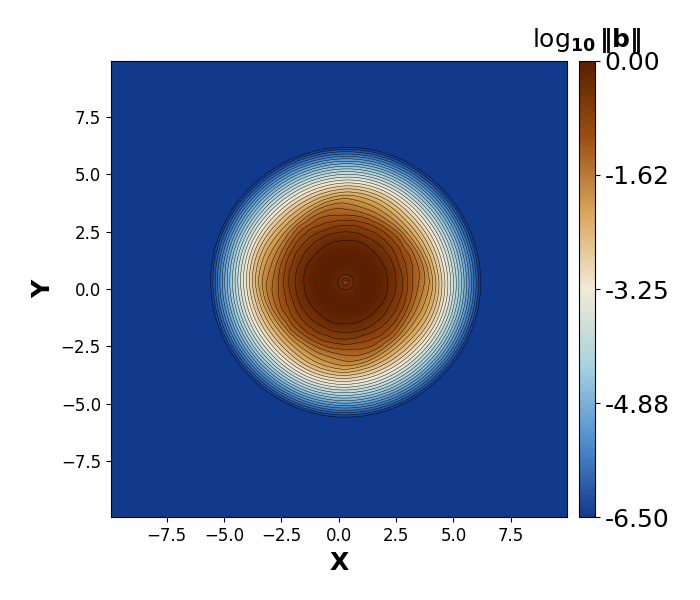}
    \end{minipage}
    \hspace{0.01\linewidth}%
    \begin{minipage}{0.48\linewidth}
        \includegraphics[width=\linewidth]{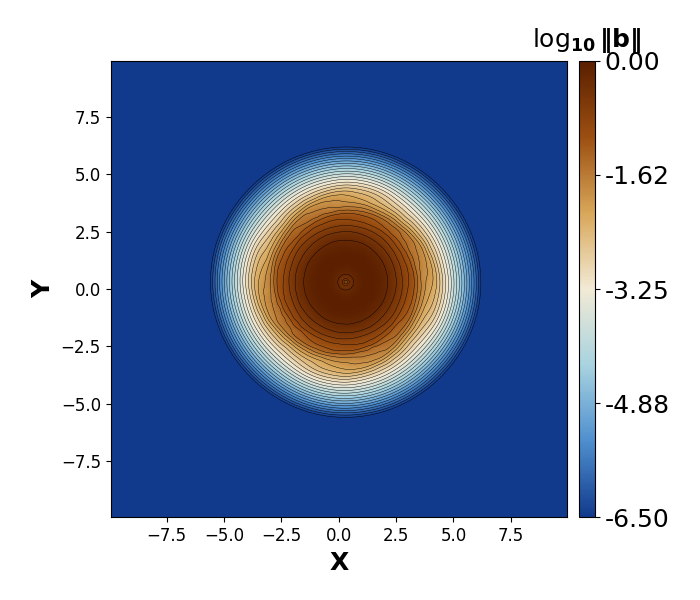}
    \end{minipage}
    \caption{Long-time simulation of a smooth stationary MHD vortex. Contour plots of $\log_{10}\|\mathbf{b}\|$, overlaid with 20 equidistant contour lines between $-6.5$ and $0$, obtained with the proposed semi-implicit scheme using the unstaggered adaptation of the CT method, shown at times $t=500$ (left) and $t=1000$ (right).}
    \label{fig:MHDvortex}
\end{figure}

\begin{figure}
    \centering
    \includegraphics[width=0.6\linewidth]{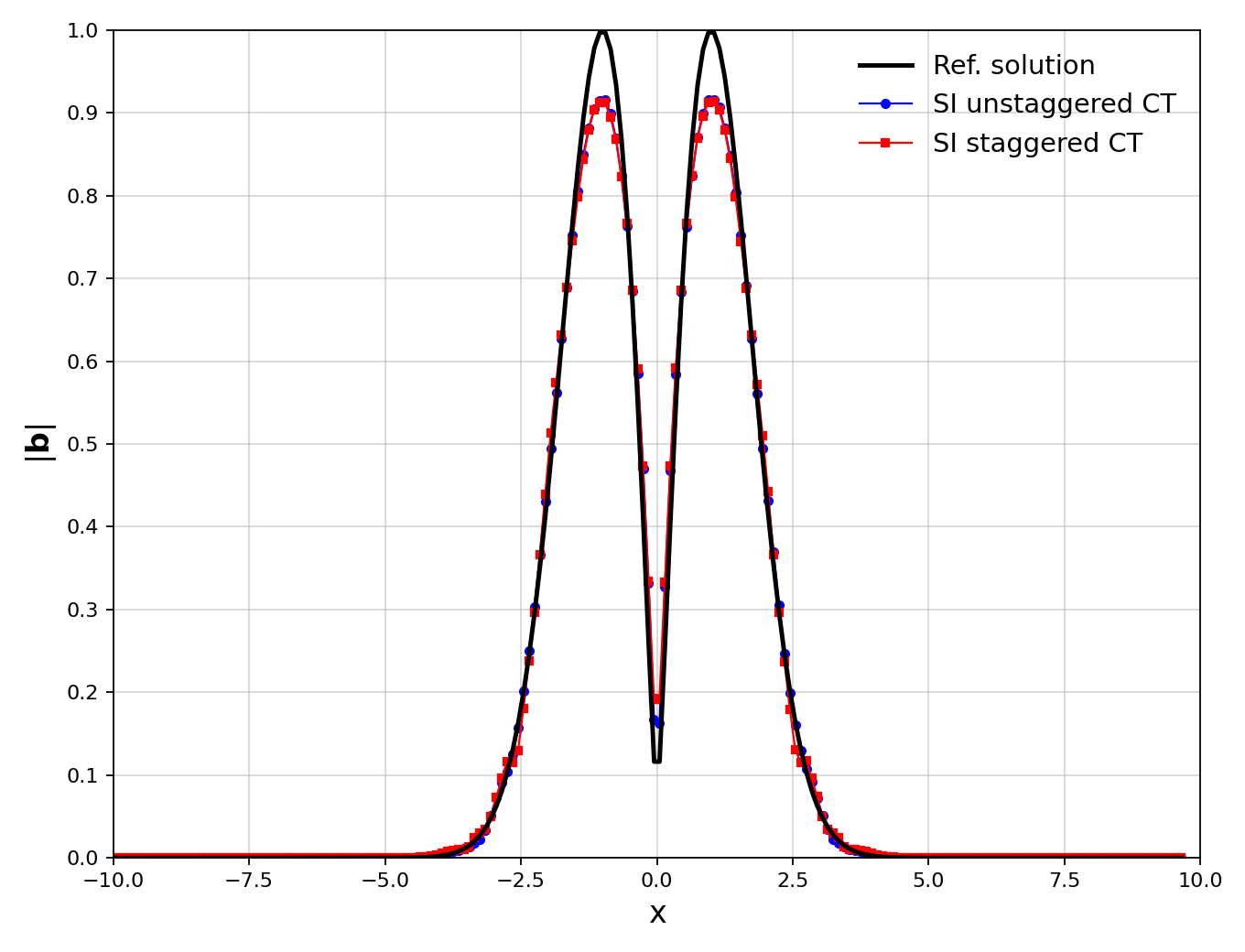}
    \caption{Long-time simulation of a smooth stationary MHD vortex at time $t=1000$. One-dimensional line-out along the $x$-axis of the normalised magnetic field $\mathbf{b}$, obtained with the proposed semi-implicit algorithm using either the staggered or the unstaggered adaptation of the CT method described in Section~\ref{chp:divfree_CT} to enforce the divergence-free condition of the magnetic field. The solution is compared with the reference (initial) solution.}
    \label{fig:vortex_lineout}
\end{figure}

\subsection{Ideal MHD rotor problem}

We next turn to the MHD rotor test of Balsara and Spicer~\cite{BalsaraSpicer1999},
by now a standard benchmark for gauging how numerical schemes cope with the
ideal MHD system. The configuration consists of a dense, rapidly spinning
cylinder of fluid immersed in a lighter fluid at rest. As the heavy core
rotates, it whips strong torsional Alfv\'en waves out into the ambient medium.
Pressure and magnetic field are initially uniform over the whole domain, and the
complete initial data read
\begin{equation}
\mathbf{B}_0 = \left(\frac{2.5}{\sqrt{4\pi}},\,0,\,0\right)^{\!\top},
\qquad p_0 = 1.0, \qquad w_0 = 0,
\end{equation}
\begin{equation}
(\rho_0,\, u_0,\, v_0) =
\begin{cases}
(10.0,\ -10\,y,\ 10\,x)^{\top}, & r \le 0.1, \\[4pt]
(1.0,\ 0.0,\ 0.0)^{\top}, & r > 0.1,
\end{cases}
\qquad r = \sqrt{x^2 + y^2}.
\end{equation}
The domain is the
unit square $[-0.5,0.5]^2$, discretised with a uniform Cartesian mesh of
$1024 \times 1024$ cells.

Fig~\ref{fig:rotor_results} collects the fluid density $\rho$, the pressure $p$ and
the magnetic energy $m$ at $t=0.25$. The left and
right columns compare the staggered and unstaggered CT formulations of
Sections~\ref{chp:divfree_CT}. In either case the solution matches the
literature closely, see ~\cite{Dumbser2019,BoscheriADER,BoscheriThomann2024}. To confirm that the solenoidal constraint is
respected, Fig.~\ref{fig:rotor_CT} tracks the $L_2$ norm of the magnetic-field
divergence error in time, which stays at machine-precision level throughout the
run for both CT strategies.

\begin{figure}
    \centering
    \includegraphics[width=0.85\linewidth]{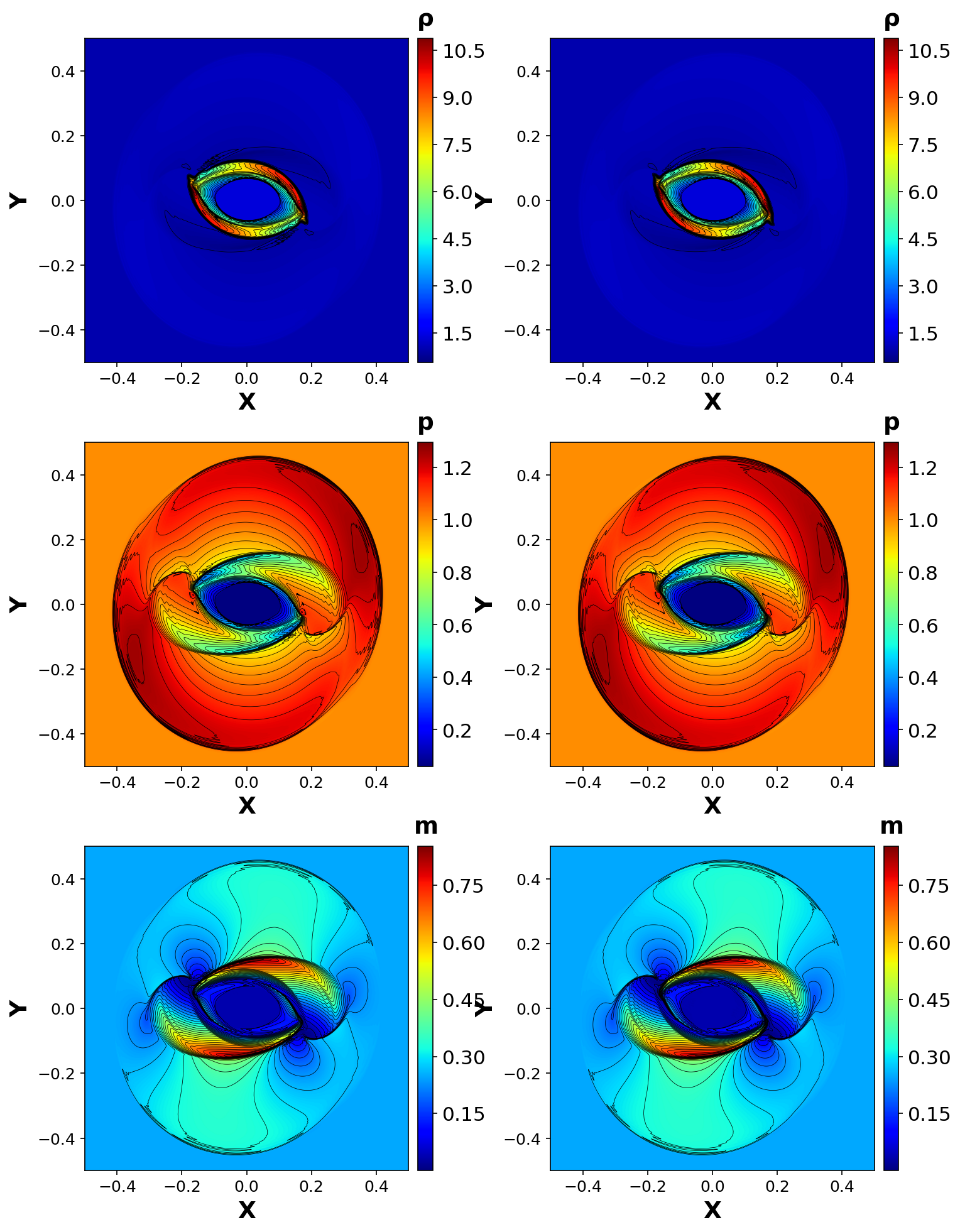}

    \caption{Ideal MHD rotor problem at $t=0.25$. Fluid density $\rho$, pressure $p$
    and magnetic energy $m$ (top to bottom), computed with the proposed semi-implicit
    scheme. The left and right columns show the staggered and unstaggered adaptations
    of the CT method, respectively.}
    \label{fig:rotor_results}
\end{figure}

\begin{figure}
    \centering
    \includegraphics[width=0.42\linewidth]{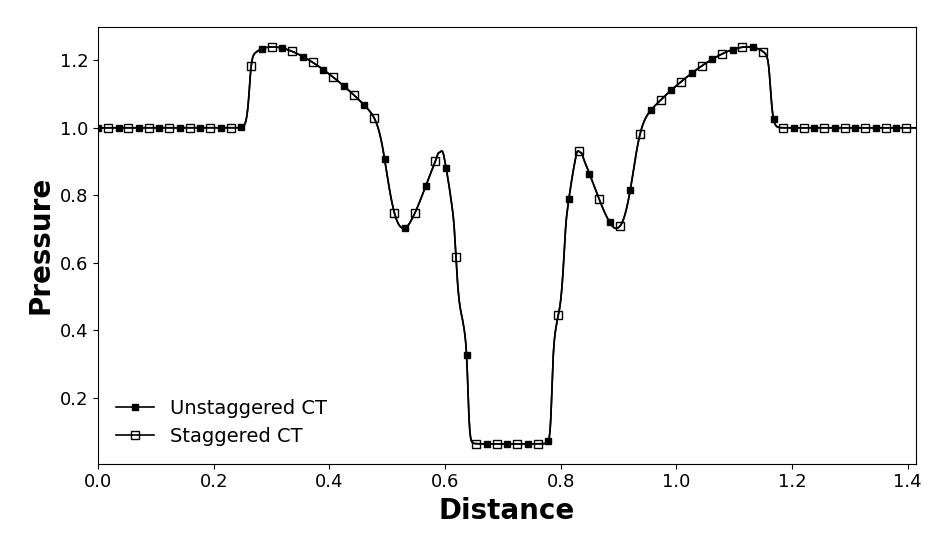}
    \includegraphics[width=0.46\linewidth]{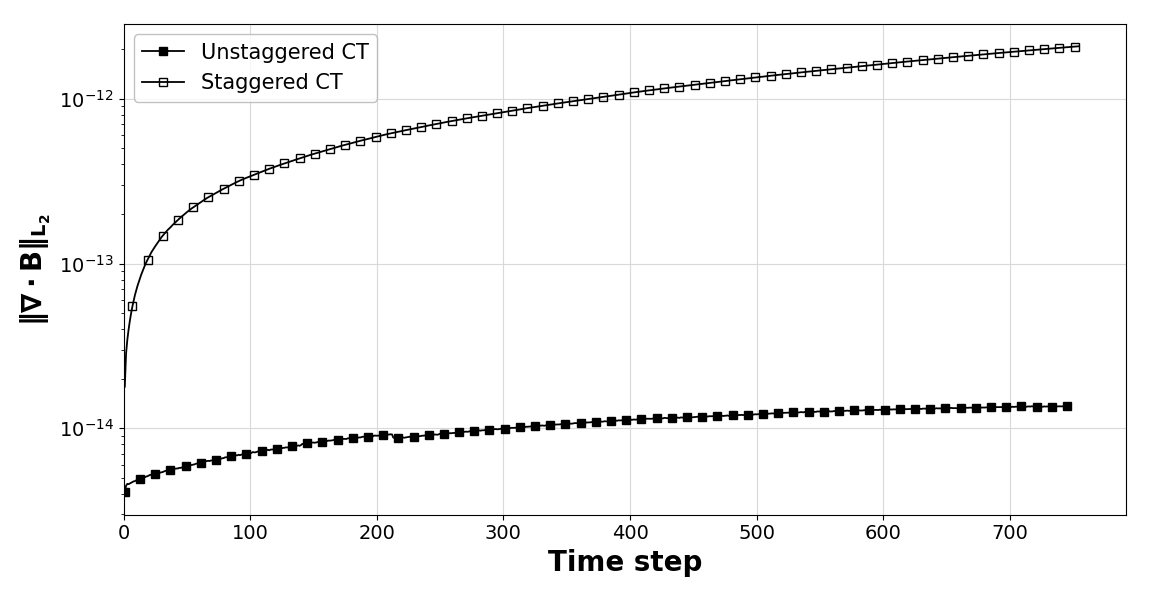}

    \caption{Detailed comparison of the two CT formulations for the ideal MHD rotor
    problem. Left: one-dimensional cross-section of the pressure field along the
    diagonal connecting $(0,0)$ and $(1,1)$, obtained with the proposed semi-implicit
    scheme using the staggered (open symbols) and unstaggered (filled symbols)
    adaptations of the CT method. Right: time evolution of
    $\|\nabla\!\cdot\mathbf{B}\|_{L^2}$ for the same two formulations, sampled every
    10 time steps.}
    \label{fig:rotor_CT}
\end{figure}

\subsection{Ideal Orszag--Tang vortex}

We now consider the Orszag--Tang vortex~\cite{orszag1979}, a popular test for
probing the robustness of a scheme in the presence of MHD shock formation,
shock--shock interactions and the ensuing transition to two-dimensional MHD
turbulence. Starting from a smooth initial state, the flow rapidly spawns a
family of shock waves travelling at disparate speeds before breaking down into
turbulence; for a thorough account of the underlying physics we refer to Orszag
and Tang~\cite{orszag1979} and Dahlburg and Picone~\cite{dahlburg1989,picone1991}. The computational domain is the unit square $[0,1]^2$ on a
uniform Cartesian grid of $1000 \times 1000$ cells, and the initial data are
\begin{align}
\rho_0 &= \gamma^{2}, \qquad p_0 = \gamma, \qquad
\mathbf{u}_0 = \bigl(-\sin(2\pi y),\ \sin(2\pi x),\ 0\bigr)^{\!\top}, \\[4pt]
\mathbf{B}_0 &= \bigl(-\sin(2\pi y),\ \sin(4\pi x),\ 0\bigr)^{\!\top},
\end{align}
with $\gamma = 5/3$. Fig.~\ref{fig:ot_vortex} shows the gas pressure and the
magnetic-field magnitude delivered by the proposed semi-implicit scheme at
$t=0.5/2\pi$, $2.0/2\pi$, $3/2\pi$ and $5/2\pi$. The $\nabla\cdot\mathbf{B}=0$
constraint is enforced through the unstaggered CT method. The computed fields agree well
with results in the literature, see e.g.\ Dumbser et
al.~\cite{Dumbser2019,dumbser2008}, Fambri~\cite{Fambri2021} and Balsara et
al.~\cite{balsara2015}.

\begin{figure}
\centering
\includegraphics[height=0.8496\textheight]{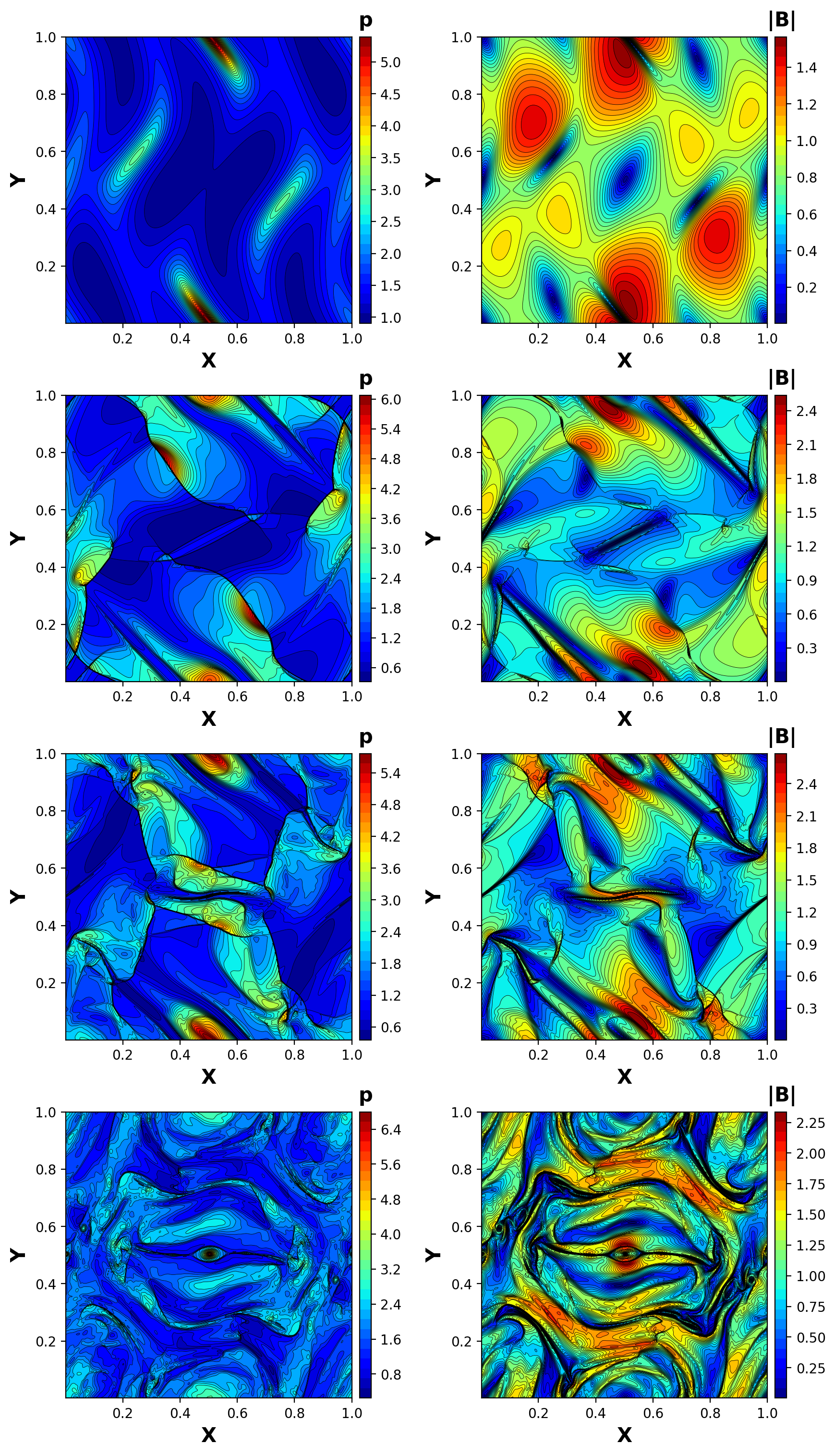}
\caption{Numerical solution for the pressure (left) and the magnetic field
magnitude (right) obtained with the proposed semi-implicit scheme for the ideal
Orszag--Tang vortex at times $t=0.5/2\pi$, $t=2.0/2\pi$, $t=3/2\pi$ and
$t=5/2\pi$ (from top to bottom).}
\label{fig:ot_vortex}
\end{figure}

\subsection{Ideal MHD blast wave problem}

The MHD blast wave of Balsara and Spicer~\cite{BalsaraSpicer1999} was devised to
scrutinise how a scheme handles strong shocks sweeping through a strongly
magnetised background, and is accordingly a demanding measure of solver
robustness. It readily provokes unphysical densities or pressures whenever the
divergence-free condition is enforced inadequately, or whenever insufficient
numerical dissipation is supplied at the shock fronts. The blast is launched from
an over-pressurised circular pocket of radius $R=0.1$ at the centre of the
domain, with the ambient fluid threaded by a very strong, uniform magnetic
field. The initial data are
\begin{equation}
\rho_0 = 1.0, \qquad \mathbf{u}_0 = \mathbf{0},
\qquad
\mathbf{B}_0 = \frac{1}{\sqrt{4\pi}}\,(100.0,\,0.0,\,0.0)^{\!\top},
\end{equation}
\begin{equation}
p =
\begin{cases}
1000.0, & 0 \le r \le R, \\[2pt]
0.1, & \text{otherwise},
\end{cases}
\end{equation}
with $r = \sqrt{x^2 + y^2}$. The domain is the square $[-0.5,0.5]^2$, resolved by
a uniform Cartesian grid of $1000 \times 1000$ cells.

Fig.~\ref{fig:blast} presents the density, velocity magnitude, gas pressure and
magnetic pressure at $t=0.01$ obtained with the proposed semi-implicit finite
volume scheme. The results compare favourably with those reported
in the literature, see for instance Balsara and Spicer~\cite{BalsaraSpicer1999}.

\begin{figure}
    \centering
    \includegraphics[width=\linewidth]{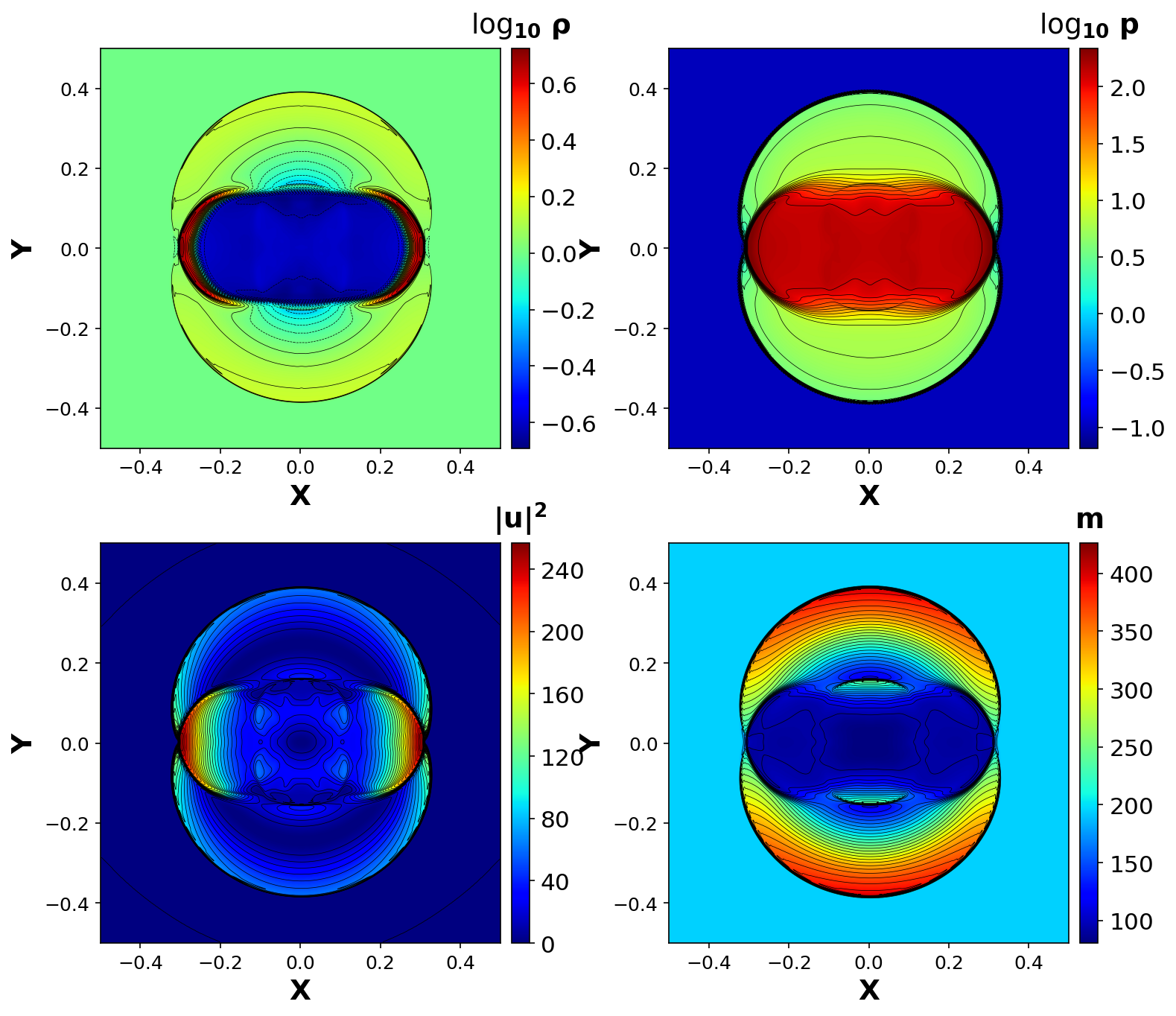}

    \caption{Numerical solution of the MHD blast wave test obtained
with the proposed semi-implicit scheme. Base-10 logarithm of the density
(top left) and of the pressure (top right), squared velocity magnitude
$|\mathbf{u}|^{2}$ (bottom left) and magnetic pressure $m$ (bottom right).}
    \label{fig:blast}
\end{figure}

\begin{figure}
    \centering
    \includegraphics[width=\linewidth]{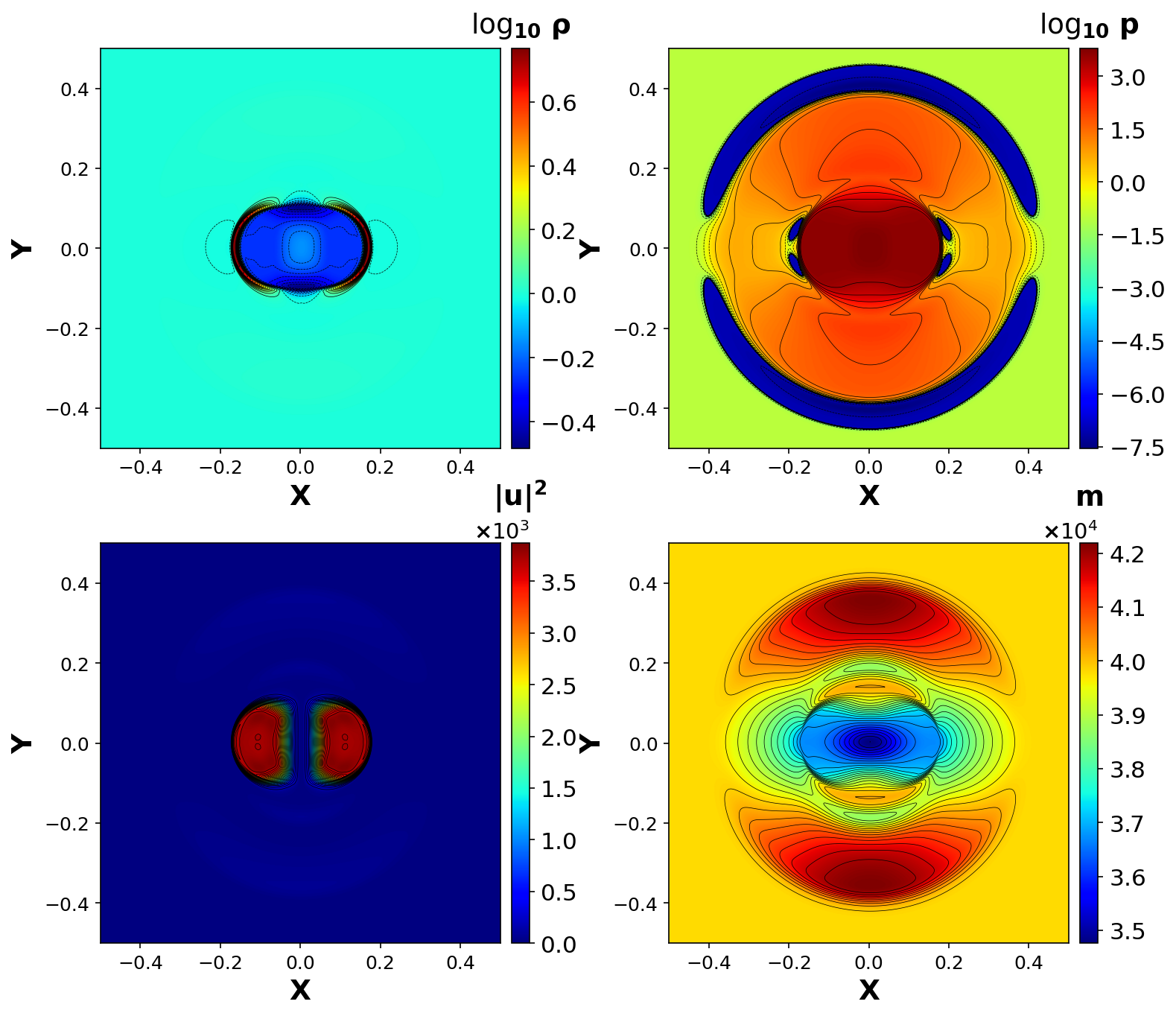}
    \caption{Numerical solution of the MHD blast wave test II obtained
with the proposed semi-implicit scheme. Base-10 logarithm of the density
(top left) and of the pressure (top right), squared velocity magnitude
$|\mathbf{u}|^{2}$ (bottom left) and magnetic pressure $m$ (bottom right).}
    \label{fig:blast2}
\end{figure}

\paragraph{Blast wave test II}
To further stress the scheme, and in particular to probe the positivity-preserving
pressure solver, we consider the more
demanding blast configuration of Liu et al.~\cite{activeflux}. The layout mirrors the
previous case, but with a markedly stronger explosion and magnetisation: the
central pressure is raised to $p_0 = 10^{4}$ and the ambient field to
\begin{equation}
\mathbf{B}_0 = \frac{1}{\sqrt{4\pi}}\,(1000.0,\,0.0,\,0.0)^{\!\top}.
\end{equation}
These choices produce far sharper discontinuities and an exceedingly low plasma
beta, $\beta \approx 2.51 \times 10^{-6}$, roughly one hundredth of the value in
the standard blast setup above. In this regime the internal energy is recovered
as a tiny difference between the total energy and the dominant kinetic and
magnetic contributions, so that a naive update readily yields negative pressures
or densities; the test therefore constitutes a stringent assessment of the
positivity machinery. The remaining initial data, the domain $[-0.5,0.5]^2$ and
the density and pressure background, are left unchanged, and the mesh is refined
to a uniform Cartesian grid of $1000 \times 1000$ cells. The solution is advanced
to $t = 0.001$.

The corresponding results are collected in Fig.~\ref{fig:blast2}. Throughout the
entire run the pressure and density are kept strictly positive by the
nested-Newton positivity-preserving solver, and no spurious oscillations develop
across the strong fronts, so that the computation completes without recourse to
any ad hoc floors or limiters. Owing to the greatly increased magnetisation, the
outer fast shock is considerably weaker and can no longer be discerned in the
density contours, in line with earlier reports~\cite{activeflux,SC3,wu2018}. This
test confirms the robustness of the proposed scheme, and of its
positivity-preserving pressure update, in the strongly magnetised,
low-plasma-beta regime.

\subsection{Field loop advection}

We consider the advection of a magnetic field loop, a classical benchmark
introduced by Gardiner and Stone~\cite{GardinerStone2005}, here in the low-acoustic-Mach
variant of~\cite{BoscheriThomann2024}. The medium has constant density
$\rho=1$ and a high background pressure
$p=10^{5}$, and is transported by a uniform diagonal velocity
$(u,v,w)=(2,1,0)$. The domain $\Omega=[-1,1]\times[-0.5,0.5]$ is covered by
$N_x\times N_y=512\times256$ cells with periodic boundaries. The magnetic field
is obtained from a $z$-aligned vector potential, which reads
\begin{equation}
A_z =
\begin{cases}
\dfrac{A_0}{\sqrt{4\pi}}\,(R-r), & r \le R,\\[2mm]
0, & r > R,
\end{cases}
\qquad r=\sqrt{x^2+y^2},
\end{equation}
with $A_0=10^{-3}$ and $R=0.3$. Taking the curl yields a field of constant
magnitude $\|\mathbf{B}\|=A_0/\sqrt{4\pi}$ inside the loop, with a singular
point at the centre of the domain that makes the test particularly demanding.
This setting corresponds to a very low acoustic Mach number,
$M_c\approx5.5\times10^{-3}$, together with a very high Alfv\'en Mach number,
$M_b\approx7.9\times10^{3}$. 

To make the test more stringent, we carry out four simulations with a
progressively larger amplitude $A_0$, up to $A_0=1$, which lowers the Alfv\'en
Mach number to $M_b\approx7.9$. The final time is $t_f=1$ and the results are
collected in Fig.~\ref{fig:fieldloop}. Across the whole sequence the magnitude
of the magnetic field retains essentially the same shape, despite the wide
variation in Alfv\'en Mach number, which confirms that the behaviour of the
proposed scheme is independent of the magnetic scales. Indeed, the time step is
identical in all four runs and is dictated solely by the convective velocity
$(u,v)=(2,1)$ transporting the loop diagonally across the domain, irrespective
of the underlying Alfv\'en speed. Finally, Fig.~\ref{fig:fieldloop_divb} reports
the divergence error of the magnetic field for all the runs, measured as the
infinity norm of $\nabla\!\cdot\mathbf{B}$ taken over the whole domain at every
time level; in all cases the discrete divergence-free constraint is preserved
down to machine accuracy.

\begin{figure}
    \centering
    \includegraphics[width=\linewidth]{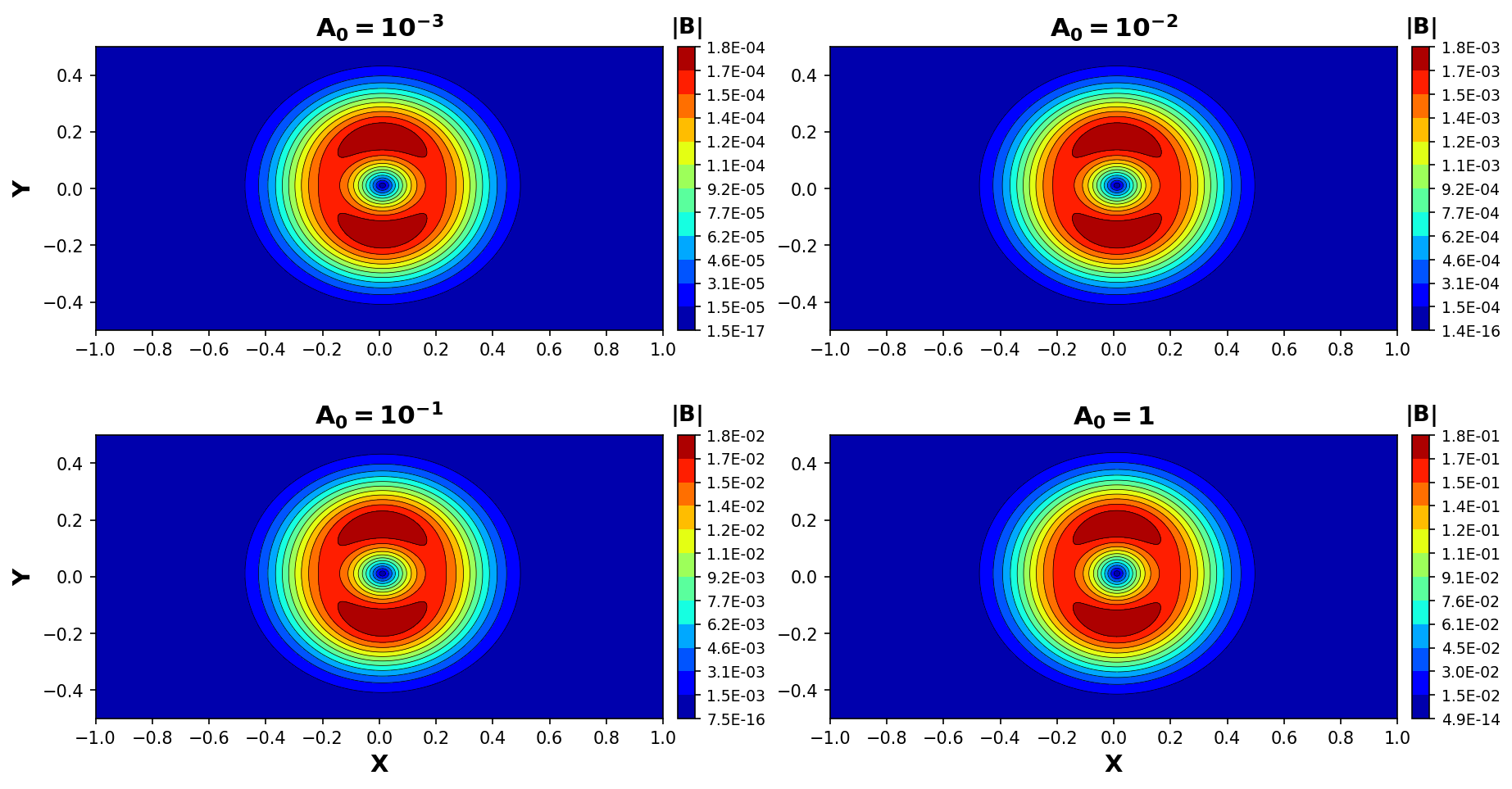}
    \caption{Magnetic field loop advection at $t_f=1$. Magnitude of the magnetic
field $\|\mathbf{B}\|$ obtained with the proposed semi-implicit scheme for four
increasing loop amplitudes, from $A_0=10^{-3}$ (top) to $A_0=1$ (bottom).}
    \label{fig:fieldloop}
\end{figure}

\begin{figure}
    \centering
    \includegraphics[width=0.6\linewidth]{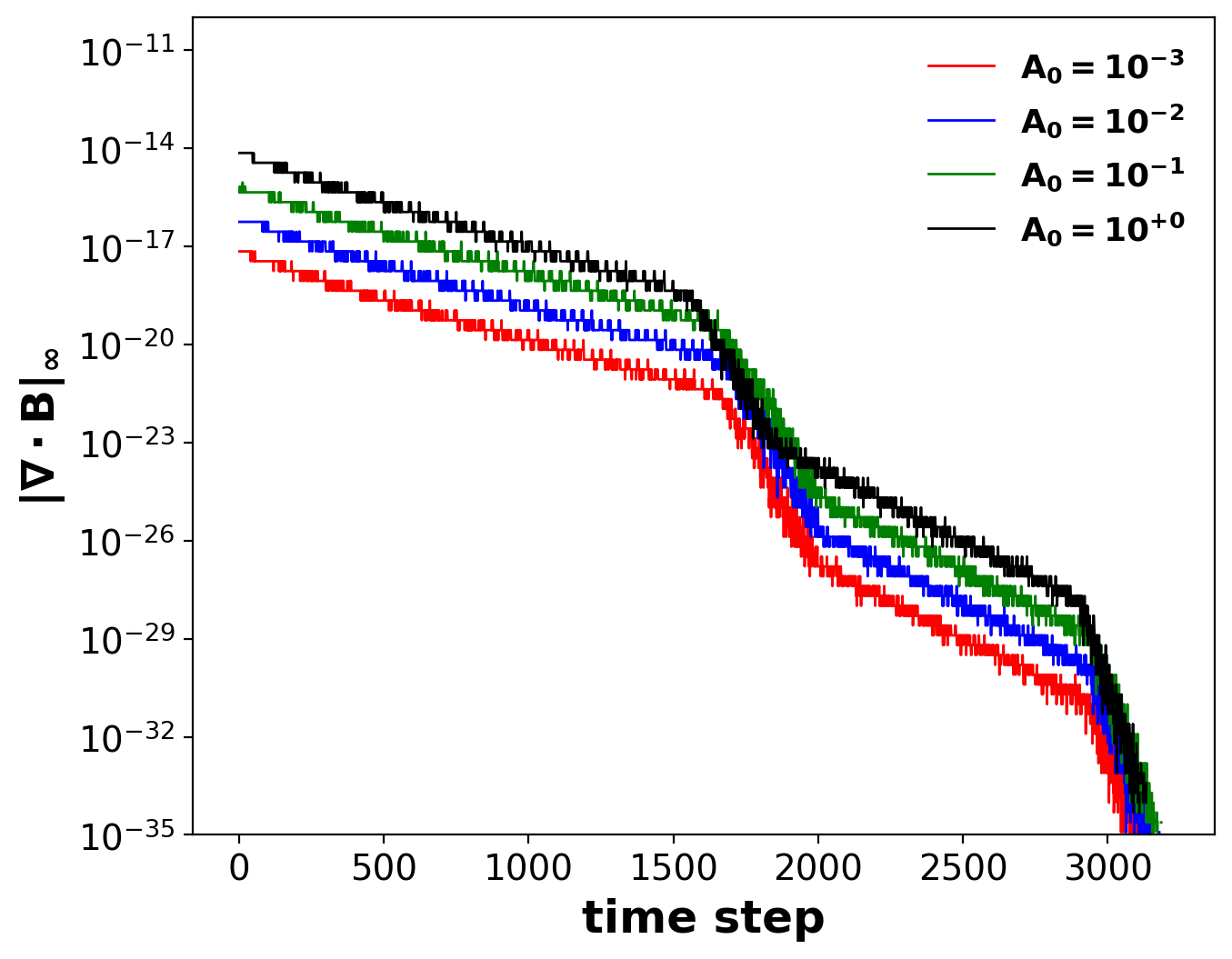}
    \caption{Magnetic field loop advection. Time evolution of the divergence error
    of the magnetic field, measured as the infinity norm of $\nabla\!\cdot\mathbf{B}$
    over the whole domain, for the four amplitudes $A_0$ of Fig.~\ref{fig:fieldloop}.
    In all cases the discrete divergence-free constraint is maintained down to
    machine accuracy.}
    \label{fig:fieldloop_divb}
\end{figure}

\subsection{Magnetized Kelvin--Helmholtz instability}

We adopt the ideal-MHD Kelvin--Helmholtz setup recently proposed
in~\cite{BoscheriThomann2024}. The computational domain is the rectangular box
$\Omega = [0,2] \times [-0.5,0.5]$ with periodic boundaries, paved by a grid of
$N_x \times N_y = 512 \times 256$ control volumes. The fluid is initialised at
uniform density $\rho = \gamma$, with $\gamma = 1.4$, and uniform pressure
$p = 1$. Velocity and magnetic field are parametrised through $M_x$, the maximum
acoustic Mach number of the horizontal stream. The velocity field is prescribed
as
\begin{equation}
u = M_x\,\bigl(1 - 2\,\eta(y)\bigr), \qquad
v = 0.1\,M_x\,\sin(2\pi x),
\label{eq:kh_vel}
\end{equation}
where the shear-layer profile $\eta(y)$ reads
\begin{equation}
\eta(y) =
\begin{cases}
0.5\,\bigl(1 + \sin(16\pi(y+0.25))\bigr), & -\tfrac{9}{32} \le y < -\tfrac{7}{32}, \\[4pt]
1, & -\tfrac{7}{32} \le y < \tfrac{7}{32}, \\[4pt]
0.5\,\bigl(1 - \sin(16\pi(y-0.25))\bigr), & \tfrac{7}{32} \le y < \tfrac{9}{32}, \\[4pt]
0, & \text{otherwise}.
\end{cases}
\label{eq:kh_eta}
\end{equation}
The initial magnetic field is uniform and aligned with the $x$-axis,
$B_x = 0.1\,M_x$, $B_y = 0$. This state yields a unit adiabatic sound speed and a
minimum Alfv\'en Mach number $M_{b,\min} = 11.82$ at $t=0$. Simulations are
carried out for several values of $M_x$, each advanced to the final time
$t_f = t_{\max}/6$ with $t_{\max} = 4.8/M_x$. During this window a smooth,
well-resolved interface forms along the shear layer, so that the early-stage
dynamics remain convergent.

Fig.~\ref{fig:kh_mach} shows the acoustic Mach number at the final time,
rescaled by the prescribed $M_x$. The profiles obtained for the different $M_x$
are virtually indistinguishable, because the implicit central discretisation of
the pressure and magnetic terms keeps numerical dissipation low. A fully
explicit scheme, by contrast, would fail to retain the vortical structures as
$M_x$ is reduced, as already observed in~\cite{BoscheriThomann2024}. Moreover,
for these test cases the proposed scheme operates with a time step up to four
orders of magnitude larger than that of a fully explicit method, as reported in
Fig.~\ref{fig:kh_dtratio}.

\begin{figure}
    \centering
    \includegraphics[width=\linewidth]{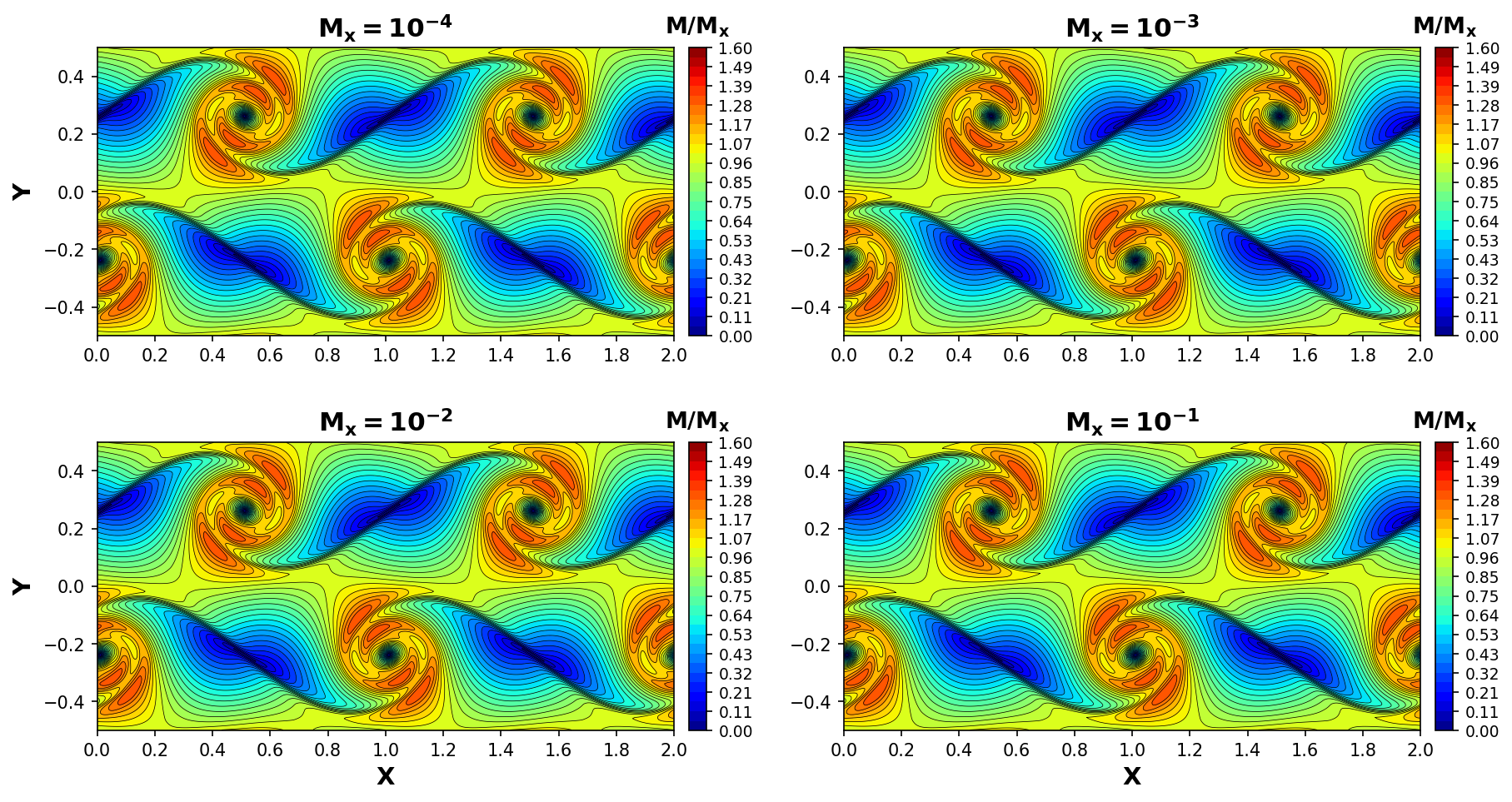}
    \caption{Magnetized Kelvin--Helmholtz instability at $t_f = t_{\max}/6$.
    Acoustic Mach number rescaled by the prescribed value $M_x$, for increasing
    $M_x$ from top left to bottom right. The rescaled profiles are almost
    identical across the whole range, showing that the proposed scheme avoids
    excessive numerical dissipation at low acoustic Mach number.}
    \label{fig:kh_mach}
\end{figure}

\begin{figure}
    \centering
    \includegraphics[width=0.6\linewidth]{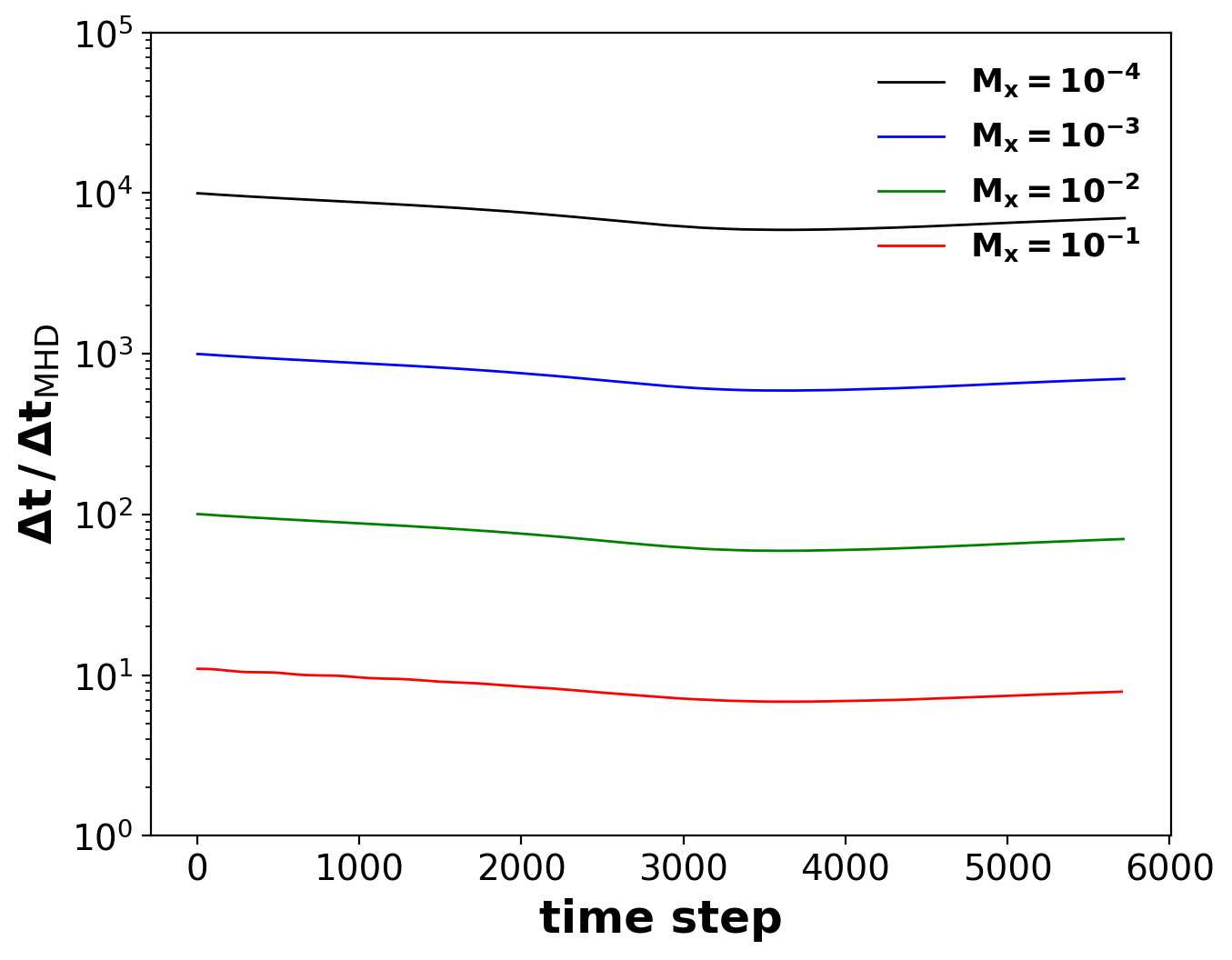}
    \caption{Magnetized Kelvin--Helmholtz instability. Time evolution of the
    ratio between the material time step of the proposed semi-implicit scheme and
    the magnetosonic time step of a fully explicit method, for the different
    values of $M_x$. The ratio reaches up to four orders of magnitude as $M_x$
    decreases.}
    \label{fig:kh_dtratio}
\end{figure}

\subsection{Shock--cloud interaction}

The shock–cloud interaction problem models the break-up of a dense cloud struck by a strong shock and the ensuing cascade of discontinuities and small-scale instabilities. We follow
the configuration of~\cite{Sc1,SC2,SC3,wu2018,activeflux}.
The initial data consist of a left and a right state separated by a vertical
discontinuity at $x = 0.6$,
\begin{equation}
(\rho, \mathbf{u}, \mathbf{B}, p) =
\begin{cases}
(3.86859,\ 0,\ 0,\ 0,\ 0,\ 2.1826182,\ -2.1826182,\ 167.345), & x < 0.6, \\[4pt]
(1,\ -11.2536,\ 0,\ 0,\ 0,\ 0.56418958,\ 0.56418958,\ 1), & x \ge 0.6,
\end{cases}
\end{equation}
where the ordering is $\rho$, $(v_x,v_y,v_z)$, $(B_x,B_y,B_z)$, $p$. Embedded in
the right state is a circular cloud of enhanced density $\rho = 10$, centred at
$(0.8,0.5)$ with radius $0.15$. The computational domain is the unit square
$[0,1]^2$, discretised with a uniform $400 \times 400$ Cartesian mesh. Inflow
conditions are prescribed on the right boundary and outflow conditions on the
remaining three sides.

Fig.~\ref{fig:shockcloud} presents the solution at $t = 0.06$. The scheme
resolves the intricate flow structure that develops as the shock sweeps over the
cloud, including the bow shock, the reflected shocks, the contact
discontinuities and the fine-scale features shed downstream, in good agreement
with the results reported in the literature~\cite{Sc1,SC2,SC3,wu2018,activeflux}. As in the preceding tests, the pressure and density remain
strictly positive throughout the computation, providing a further confirmation of
the robustness and positivity-preserving character of the proposed scheme.

\begin{figure}
    \centering
    \includegraphics[width=\linewidth]{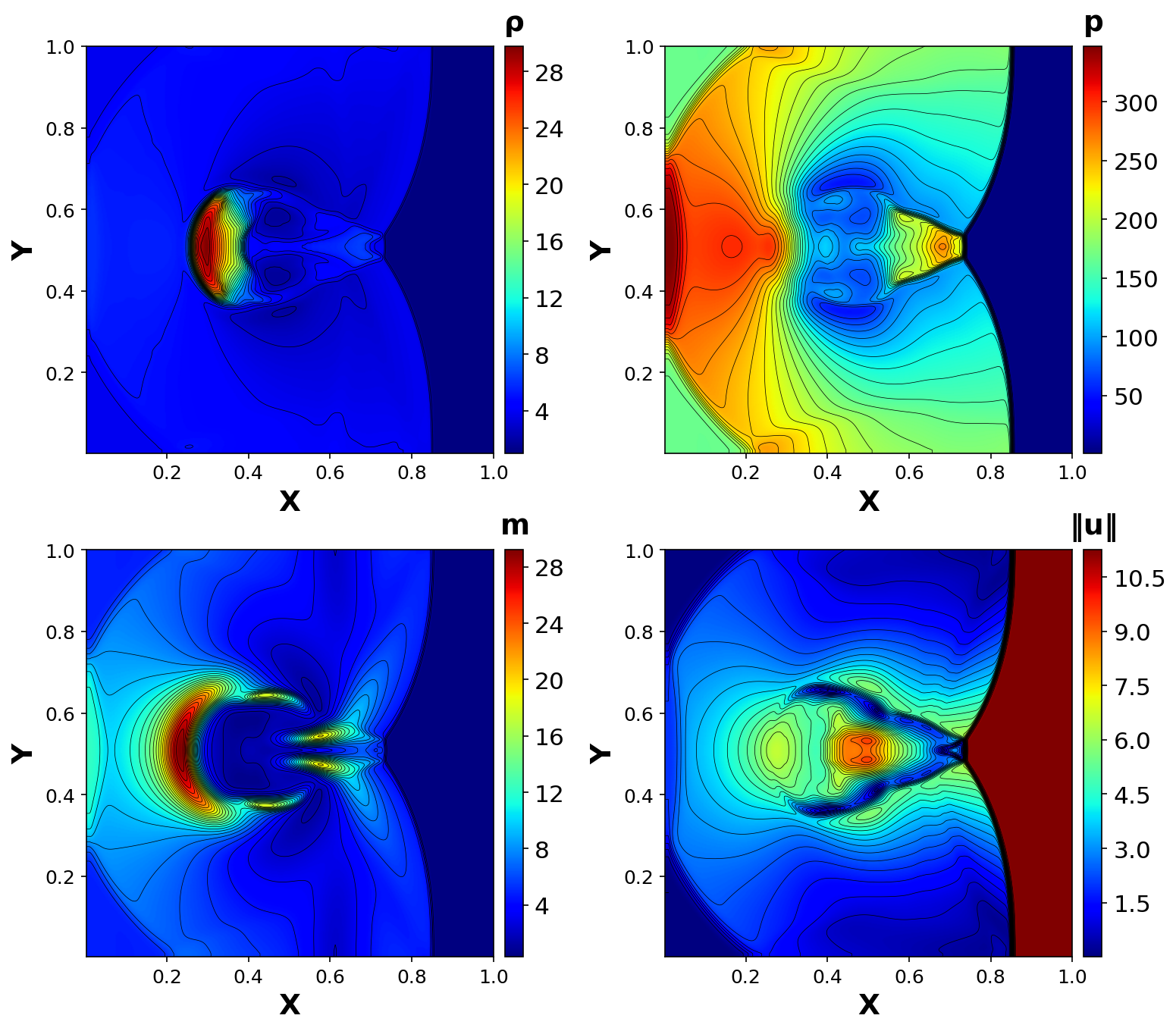}
    \caption{Shock--cloud interaction problem. Density (top left), thermal pressure
(top right), magnetic pressure (bottom left) and velocity magnitude (bottom
right).}
    \label{fig:shockcloud}
\end{figure}

\subsection{MHD Jets}
\begin{figure}
    \centering
    \includegraphics[width=\linewidth]{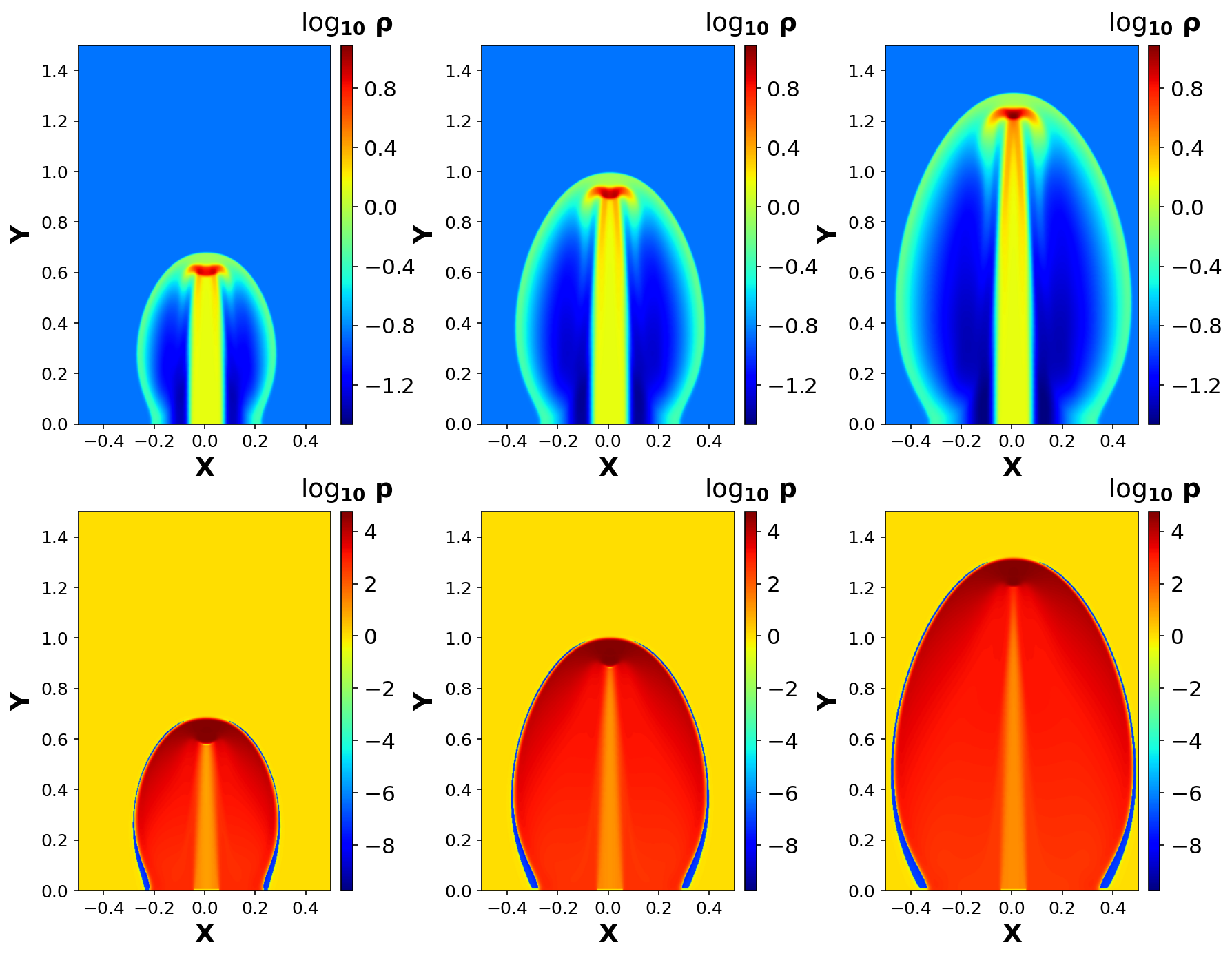}
    \caption{Acoustic Mach 800 jet problem with $B_0=\sqrt{200}$ ($\beta_0=10^{-2}$). Logarithm of the density (top row) and logarithm of the pressure (bottom row) at $t=0.001$, $0.0015$ and $0.002$ (from left to right).}
    \label{fig:MHDJet_200}
\end{figure}
\begin{figure}
    \centering
    \includegraphics[width=\linewidth]{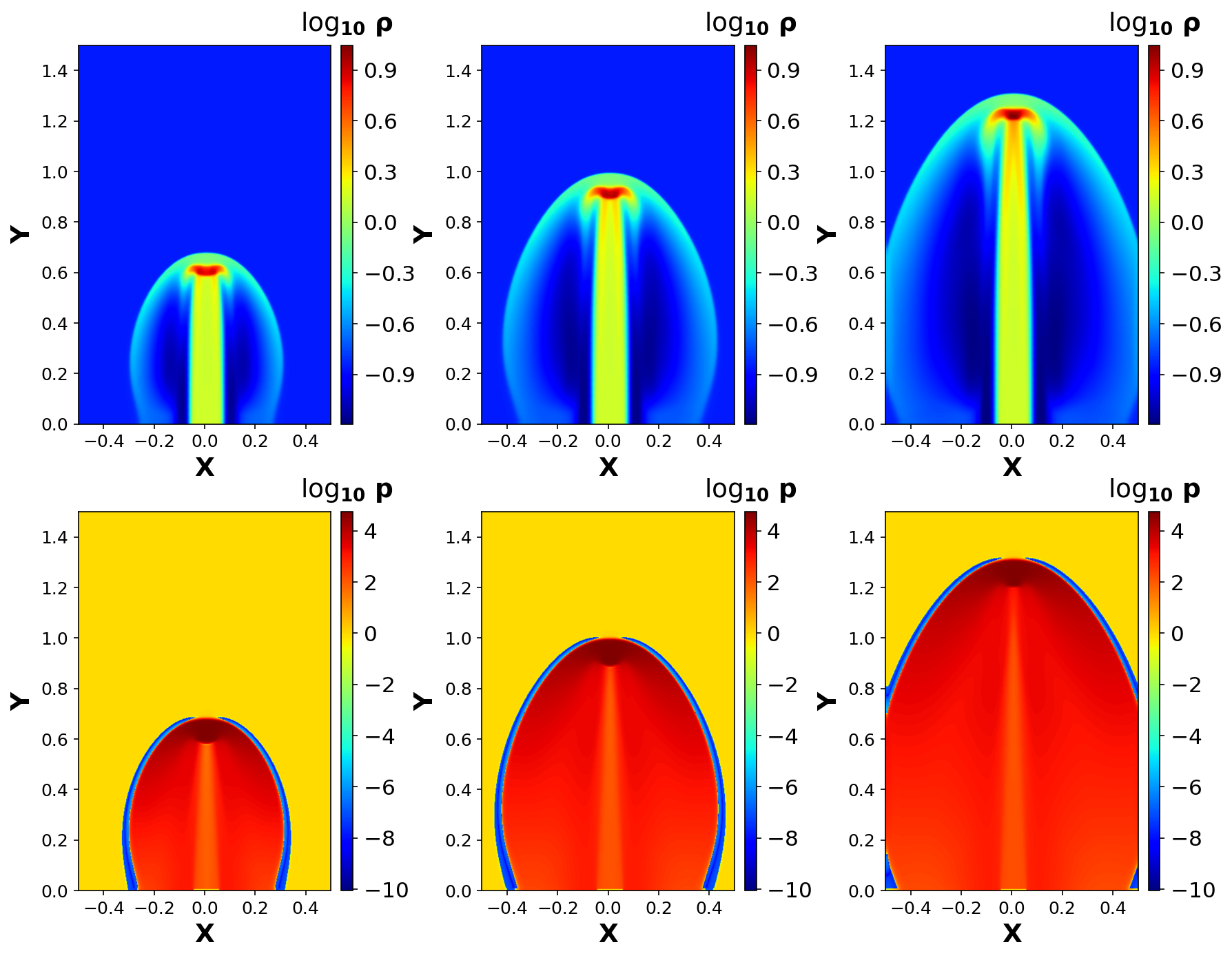}
    \caption{Acoustic Mach 800 jet problem with $B_0=\sqrt{2000}$ ($\beta_0=10^{-3}$).
    Logarithm of the density (top row) and logarithm of the pressure (bottom row)
    at $t=0.001$, $0.0015$ and $0.002$ (from left to right).}
    \label{fig:MHDJet_2000}
\end{figure}
\begin{figure}
    \centering
    \includegraphics[width=\linewidth]{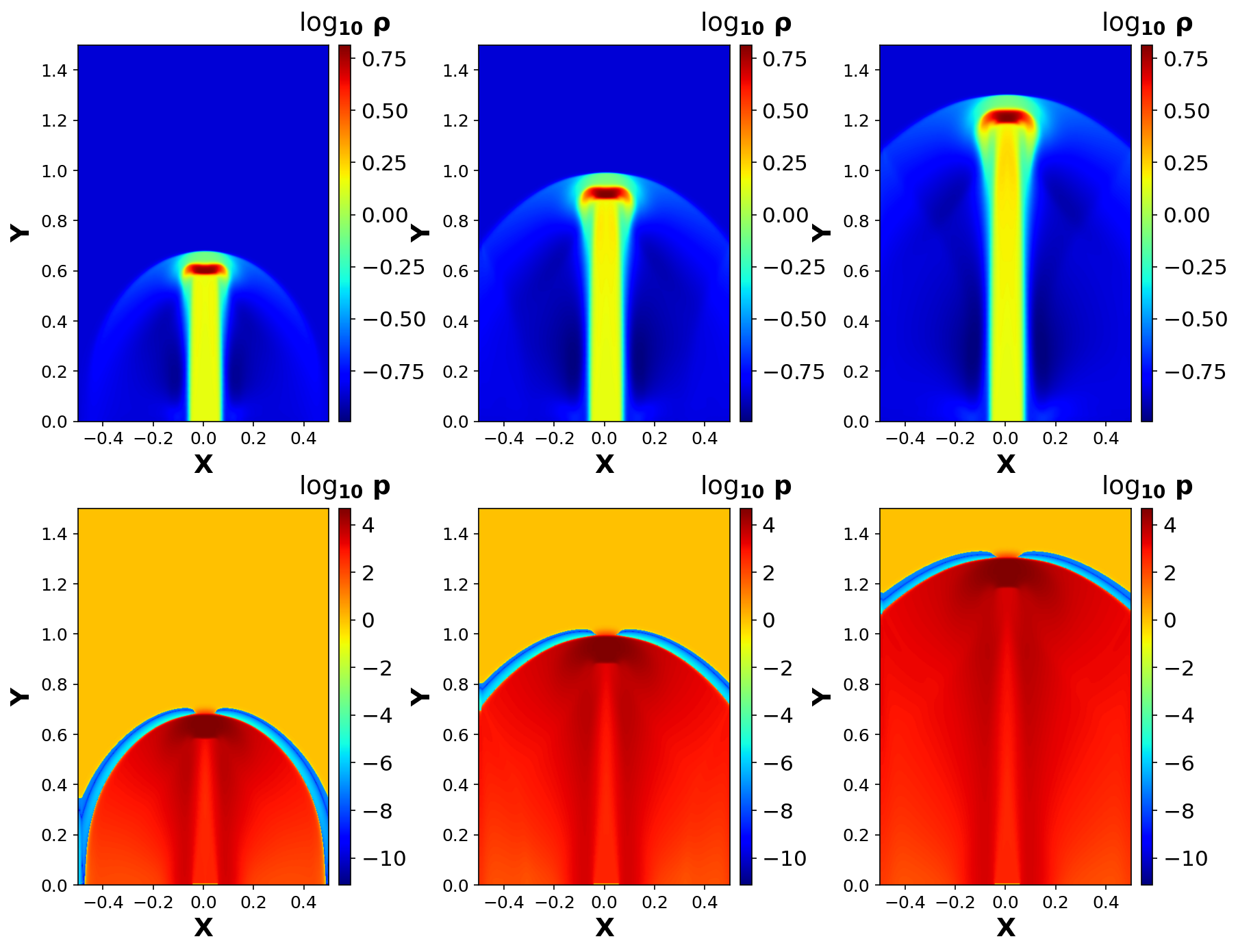}
    \caption{Acoustic Mach 800 jet problem with $B_0=\sqrt{20000}$ ($\beta_0=10^{-4}$).
Logarithm of the density (top row) and logarithm of the pressure
(bottom row) at $t=0.001$, $0.0015$ and $0.002$ (from left to right).}
    \label{fig:MHDJet_20000}
\end{figure}

\begin{figure}
    \centering
    \includegraphics[width=\linewidth]{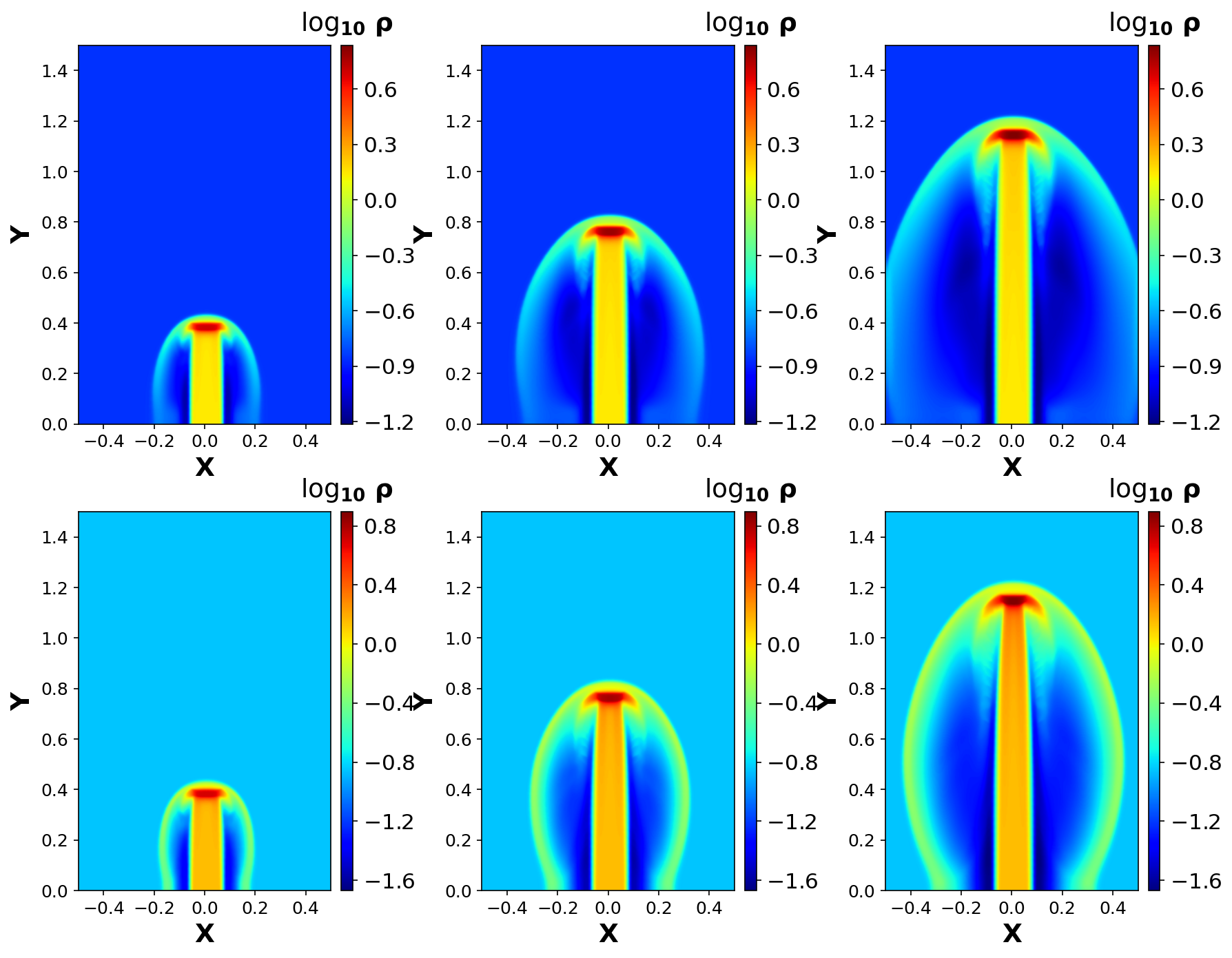}
    \caption{Acoustic MHD jet problem with $B_0=\sqrt{20000}$ ($\beta_0=10^{-4}$). Logarithm of the density for the acoustic Mach 2000 jet at $t=0.001$, $0.0015$ and $0.002$ (top row) and for the acoustic Mach 10000 jet at $t=0.00025$, $0.0005$ and $0.00075$ (bottom row), with time increasing from left to right.}
    \label{fig:MHD_Jet_diffmMch}
\end{figure}

As a last and particularly severe benchmark we turn to the high acoustic Mach number MHD
jet, a standard positivity-preserving test~\cite{SC3,wu2018,activeflux}. The
problem couples supersonic jet propagation, strong magnetic fields and very high
kinetic energy; the interplay of strong shocks, shear layers and interface
instabilities routinely drives approximate solvers toward negative pressures,
which makes it a demanding probe of both robustness and accuracy.

We consider the acoustic Mach $800$ jet on the domain $[-0.5,0.5] \times [0,1.5]$,
discretised uniformly with $400 \times 600$ cells. The ambient medium is
initialised as
\begin{equation}
(\rho, \mathbf{u}, \mathbf{B}, p) = (0.1\gamma,\ 0,\ 0,\ 0,\ 0,\ B_0,\ 0,\ 1),
\qquad \gamma = 1.4,
\end{equation}
where the ordering is $\rho$, $(v_x,v_y,v_z)$, $p$, $(B_x,B_y,B_z)$. Along a
narrow inlet $|x| < 0.05$ on the bottom boundary the jet is injected with the
state
\begin{equation}
(\rho, \mathbf{u}, p, \mathbf{B}) = (\gamma,\ 0,\ 800,\ 0,\ 1,\ 0,\ B_0,\ 0),
\end{equation}
while all the remaining boundaries are treated as outflow.

Following~\cite{SC3,wu2018,activeflux}, we assess the scheme at three levels of
magnetisation by varying the field strength $B_0$:
\begin{equation}
\text{(i)}\ B_0 = \sqrt{200}\ (\beta_0 = 10^{-2}), \quad
\text{(ii)}\ B_0 = \sqrt{2000}\ (\beta_0 = 10^{-3}), \quad
\text{(iii)}\ B_0 = \sqrt{20000}\ (\beta_0 = 10^{-4}).
\end{equation}
As $B_0$ grows and the plasma beta $\beta_0$ falls, the dynamics become
progressively dominated by magnetic forces, which sharpens the numerical
stiffness and heightens the risk of nonphysical states, thereby stressing the
robustness of the method ever more severely.

The results for the three cases are collected in
Figs.~\ref{fig:MHDJet_200}--\ref{fig:MHDJet_20000}, showing logarithmic density and
pressure contours. In every case the acoustic Mach stem, the shear layers and the contact
discontinuities are captured sharply and without spurious oscillations, and no
negative pressure or density is produced at any point of the simulation. This
confirms that the positivity-preserving pressure update keeps the scheme reliable
even under these highly magnetised, extreme conditions.

To push the method still further, we additionally simulate the acoustic Mach $2000$ and
acoustic Mach $10\,000$ jets~\cite{wu2018,activeflux} at the strongest magnetisation
$B_0 = \sqrt{20000}$, reported in Fig.~\ref{fig:MHD_Jet_diffmMch}. As the injection acoustic Mach number rises the jet grows
narrower and more elongated, developing distinct features across the different
magnetisation levels. The scheme resolves the acoustic Mach stem and the attendant
discontinuities cleanly, again without generating any negative density or
pressure, in agreement with its provable positivity-preserving property.
Moreover, the axial symmetry of the flow is well preserved throughout, which
further attests to the physical consistency of the proposed method.

\section{Conclusions}
\label{sec:conclusions}

In this paper we have presented a conservative, cell-centred, semi-implicit finite volume
method for the ideal MHD equations which is simultaneously structure- and
pressure-positivity-preserving. Two features distinguish it from the schemes on which it
builds. First, the implicit step is posed for the pressure rather than for the total
energy, so that a general nonlinear equation of state is admitted without altering the
structure of the pressure system $\mathbf{V}(\mathbf{p}) + \mathbf{T}\,\mathbf{p} =
\mathbf{b}$: the closure enters only through the diagonal term $\mathbf{V}(\mathbf{p})$,
leaving the positive semi-definite coupling matrix $\mathbf{T}$ unchanged, so
that a single linear solve suffices for a $\gamma$-law gas and a nested Newton iteration
for a nonlinear closure. For a tabulated closure such as SESAME or QEOS the
pressure--energy relation is nonlinear, and an energy-based formulation would let that
nonlinearity spread into the coupling matrix, whereas here it is confined to the diagonal.
Second, the positivity of the pressure is obtained from within that same system, as the
solution of a modified but still flux-consistent problem, rather than by correcting the
solution once it has been computed. Both are dictated by the applications that motivated
this work: in magneto-inertial fusion the compressed material is described by tabulated
closures whose pressure--energy relation is nonlinear, and the flows of interest are
strongly magnetised, so that the internal energy is recovered as a small difference
between two much larger quantities.

The construction relies on a three-splitting of the flux \cite{Fambri2021}, which isolates the transport
by the fluid velocity from the magnetic and the acoustic contributions. Only the first of
the three enters the explicit stability constraint: it is advanced by a conservative
finite volume update whose numerical dissipation is proportional to the material velocity
only, while the magnetic and pressure sub-systems, carrying the Alfv\'en, magnetosonic
and acoustic waves, are integrated implicitly. The two implicit solves are performed in a
single one-directional sequence following \cite{BoscheriThomann2024}, namely the magnetic field first, the pressure once that
field is known, so that no outer iteration is required to couple them. A semi-implicit
linearisation of the quadratic field terms keeps the magnetic system linear. The pressure
system is linear for a $\gamma$-law gas and only mildly nonlinear otherwise, since the
equation of state enters exclusively through the diagonal and leaves the positive semi-definite coupling matrix independent of the thermodynamic closure; the
nested Newton iteration therefore converges in a few steps \cite{DumbserCasulli2016}, and collapses to a single
linear solve in the ideal-gas case. Inspired by \cite{BrugnanoCasulli2009, Casulli2009}, positivity is built into that same iteration by a steep-slope modification of the
internal-energy relation. Being confined to the diagonal, it alters no inter-cell flux,
so that the pressure returned is the genuine solution of a modified but still
flux-consistent system rather than a value overwritten after the fact, and the discrete
conservation of the total energy is preserved. 

The scheme attains second order in space and in time. Its generic equation-of-state
branch was verified in isolation, with the magnetic field switched off, against exact
solutions of the Redlich--Kwong Riemann problem. The design order was then recovered on a
smooth magnetohydrodynamic vortex across five decades of background density, and hence
across a correspondingly wide range of Alfv\'en speeds, with an admissible time step
exceeding the explicit magnetosonic one by up to four orders of magnitude in the stiffest
configurations. Shock-capturing was assessed on classical one- and two-dimensional
benchmarks, and the positivity machinery on a set of deliberately punishing tests such as
strongly magnetised blast waves, shock-cloud interaction, and magnetised jets up to acoustic Mach
$10^4$ at plasma beta $10^{-4}$, throughout which neither the pressure nor the density
became negative at any point, and no pressure floor, clipping or positivity limiter was required, beyond the slope limiter used for shock capturing. Taken together,
these results support the claim that the method is a genuine all-acoustic-Mach and all-Alfv\'en-Mach 
number solver, in which structure preservation, guaranteed positivity and a material-only
CFL condition coexist within a single cell-centred framework.

Several directions follow from this work. The most immediate, and the closest to the
applications discussed above, is the coupling of the scheme to tabulated closures of
SESAME \cite{SESAME} or QEOS \cite{QEOS} type. Such tables enter the algorithm only through the pointwise value of
the internal-energy relation and of its derivative with respect to the pressure, and
therefore leave both the coupling matrix and the linear solver untouched. The extension to three space dimensions follows the same construction
and is a natural next step. Beyond that, we intend to derive a well-balanced variant able to preserve
magnetohydrostatic and Grad--Shafranov equilibria of the ideal sub-system exactly at the
discrete level, without which the growth of small perturbations on a static equilibrium
cannot be resolved reliably \cite{BoscheriWB,FAMBRIWB}; and to admit the viscous and
resistive terms, treating them implicitly alongside the pressure sub-system, for which
\cite{Dematte2024} provides a natural starting point.







\printcredits


\bibliographystyle{elsarticle-num}
\bibliography{main}

\begin{thebibliography}{10}
\expandafter\ifx\csname url\endcsname\relax
  \def\url#1{\texttt{#1}}\fi
\expandafter\ifx\csname urlprefix\endcsname\relax\def\urlprefix{URL }\fi
\expandafter\ifx\csname href\endcsname\relax
  \def\href#1#2{#2} \def\path#1{#1}\fi

\bibitem{GoedbloedPoedts2019}
J.~P. Goedbloed, H.~Goedbloed, R.~Keppens, S.~Poedts, Magnetohydrodynamics: of laboratory and astrophysical plasmas, Cambridge University Press, 2019.

\bibitem{Priest2014}
E.~Priest, Magnetohydrodynamics of the Sun, Cambridge University Press, 2014.

\bibitem{Jardin2010}
S.~Jardin, Computational methods in plasma physics, CRC press, 2010.

\bibitem{AtzeniMeyer2004}
S.~Atzeni, J.~Meyer-ter Vehn, The Physics of Inertial Fusion: BeamPlasma Interaction, Hydrodynamics, Hot Dense Matter, Oxford University Press, 2004.

\bibitem{disruption}
T.~Hender, G.~Arnoux, P.~de~Vries, S.~Gerasimov, A.~Huber, M.~Johnson, Jet disruption studies in support of iter, in: Fusion Energy 2010 Proc. 23rd Int. Conf., IAEA, 2010.

\bibitem{Slutz2010}
S.~Slutz, M.~Herrmann, R.~Vesey, A.~Sefkow, D.~Sinars, D.~Rovang, K.~Peterson, M.~Cuneo, Pulsed-power-driven cylindrical liner implosions of laser preheated fuel magnetized with an axial field, Physics of Plasmas 17~(5) (2010).

\bibitem{Gomez2014}
M.~R. Gomez, S.~A. Slutz, A.~B. Sefkow, D.~B. Sinars, K.~D. Hahn, S.~B. Hansen, E.~C. Harding, P.~F. Knapp, P.~F. Schmit, C.~A. Jennings, et~al., Experimental demonstration of fusion-relevant conditions in magnetized liner inertial fusion, Physical review letters 113~(15) (2014) 155003.

\bibitem{Sinars2020}
D.~Sinars, M.~Sweeney, C.~Alexander, D.~Ampleford, T.~Ao, J.~Apruzese, C.~Aragon, D.~Armstrong, K.~Austin, T.~Awe, et~al., Review of pulsed power-driven high energy density physics research on z at sandia, Physics of Plasmas 27~(7) (2020).

\bibitem{Chorin1968}
A.~J. Chorin, Numerical solution of the navier-stokes equations, Mathematics of computation 22~(104) (1968) 745--762.

\bibitem{BrioWu1988}
M.~Brio, C.~C. Wu, An upwind differencing scheme for the equations of ideal magnetohydrodynamics, Journal of computational physics 75~(2) (1988) 400--422.

\bibitem{MiyoshiKusano2005}
T.~Miyoshi, K.~Kusano, A multi-state hll approximate riemann solver for ideal magnetohydrodynamics, Journal of Computational Physics 208~(1) (2005) 315--344.

\bibitem{GuillardViozat2017}
H.~Guillard, B.~Nkonga, On the behaviour of upwind schemes in the low mach number limit: A review, Handbook of Numerical Analysis 18 (2017) 203--231.

\bibitem{Dellacherie2010}
S.~Dellacherie, Analysis of godunov type schemes applied to the compressible euler system at low mach number, Journal of Computational Physics 229~(4) (2010) 978--1016.

\bibitem{barsukow2021truly}
W.~Barsukow, Truly multi-dimensional all-speed schemes for the euler equations on cartesian grids, Journal of Computational Physics 435 (2021) 110216.

\bibitem{barsukow2023all}
W.~Barsukow, All-speed numerical methods for the euler equations via a sequential explicit time integration, Journal of Scientific Computing 95~(2) (2023) 53.

\bibitem{leidi2022finite}
G.~Leidi, C.~Birke, R.~Andrassy, J.~Higl, P.~V. Edelmann, G.~Wiest, C.~Klingenberg, F.~K. R{\"o}pke, A finite-volume scheme for modeling compressible magnetohydrodynamic flows at low mach numbers in stellar interiors, Astronomy \& Astrophysics 668 (2022) A143.

\bibitem{viallet2011towards}
M.~Viallet, I.~Baraffe, R.~Walder, Towards a new generation of multi-dimensional stellar evolution models: development of an implicit hydrodynamic code, Astronomy \& Astrophysics 531 (2011) A86.

\bibitem{RKIMEX}
S.~Boscarino, F.~Filbet, G.~Russo, High order semi-implicit schemes for time dependent partial differential equations, Journal of Scientific Computing 68~(3) (2016) 975--1001.

\bibitem{RKIMEX2}
L.~Pareschi, G.~Russo, Implicit--explicit runge--kutta schemes and applications to hyperbolic systems with relaxation, Journal of Scientific computing 25~(1) (2005) 129--155.

\bibitem{HarlowWelch1965}
F.~H. Harlow, J.~E. Welch, et~al., Numerical calculation of time-dependent viscous incompressible flow of fluid with free surface, Physics of fluids 8~(12) (1965) 2182.

\bibitem{CasulliGreenspan1984}
V.~Casulli, D.~Greenspan, Pressure method for the numerical solution of transient, compressible fluid flows, International Journal for Numerical Methods in Fluids 4~(11) (1984) 1001--1012.

\bibitem{DumbserCasulli2016}
M.~Dumbser, V.~Casulli, A conservative, weakly nonlinear semi-implicit finite volume scheme for the compressible navier- stokes equations with general equation of state, Applied Mathematics and Computation 272 (2016) 479--497.

\bibitem{ToroVazquez2012}
E.~F. Toro, M.~V{\'a}zquez-Cend{\'o}n, Flux splitting schemes for the euler equations, Computers \& Fluids 70 (2012) 1--12.

\bibitem{CasulliZanolli2012}
V.~Casulli, P.~Zanolli, Iterative solutions of mildly nonlinear systems, Journal of Computational and Applied Mathematics 236~(16) (2012) 3937--3947.

\bibitem{BoscheriPareschi2021}
W.~Boscheri, L.~Pareschi, High order pressure-based semi-implicit imex schemes for the 3d navier-stokes equations at all mach numbers, Journal of Computational Physics 434 (2021) 110206.

\bibitem{TavelliDumbser2017}
M.~Tavelli, M.~Dumbser, A pressure-based semi-implicit space--time discontinuous galerkin method on staggered unstructured meshes for the solution of the compressible navier--stokes equations at all mach numbers, Journal of Computational Physics 341 (2017) 341--376.

\bibitem{Dumbser2019}
M.~Dumbser, D.~S. Balsara, M.~Tavelli, F.~Fambri, A divergence-free semi-implicit finite volume scheme for ideal, viscous, and resistive magnetohydrodynamics, International Journal for Numerical Methods in Fluids 89~(1-2) (2019) 16--42.

\bibitem{Fambri2021}
F.~Fambri, A novel structure preserving semi-implicit finite volume method for viscous and resistive magnetohydrodynamics, International Journal for Numerical Methods in Fluids 93~(12) (2021) 3447--3489.

\bibitem{Dematte2024}
R.~Dematt{\'e}, A.~A. Farmakalides, S.~Millmore, N.~Nikiforakis, An all mach number scheme for visco-resistive magnetically-dominated mhd flows, Journal of Computational Physics 514 (2024) 113229.

\bibitem{BalsaraMontecinosToro2016}
D.~S. Balsara, G.~I. Montecinos, E.~F. Toro, Exploring various flux vector splittings for the magnetohydrodynamic system, Journal of Computational Physics 311 (2016) 1--21.

\bibitem{BoscheriThomann2024}
W.~Boscheri, A.~Thomann, A structure-preserving semi-implicit imex finite volume scheme for ideal magnetohydrodynamics at all mach and alfv{\'e}n numbers, Journal of Scientific Computing 100~(3) (2024) 67.

\bibitem{FVVEM}
W.~Boscheri, S.~Busto, M.~Dumbser, An all mach number semi-implicit hybrid finite volume/virtual element method for compressible viscous flows on voronoi meshes, Computer Methods in Applied Mechanics and Engineering 433 (2025) 117502.

\bibitem{LSDIRK2}
S.~Avgerinos, F.~Bernard, A.~Iollo, G.~Russo, Linearly implicit all mach number shock capturing schemes for the euler equations, Journal of Computational Physics 393 (2019) 278--312.

\bibitem{SESAME}
S.~P. Lyon, J.~D. Johnson, {SESAME}: The {L}os {A}lamos national laboratory equation of state database, Tech. Rep. LA-UR-92-3407, Los Alamos National Laboratory (1992).

\bibitem{QEOS}
R.~M. More, K.~H. Warren, D.~A. Young, G.~B. Zimmerman, A new quotidian equation of state ({QEOS}) for hot dense matter, Physics of Fluids 31~(10) (1988) 3059--3078.

\bibitem{RedlichKwong1949}
O.~Redlich, J.~N. Kwong, On the thermodynamics of solutions. v. an equation of state. fugacities of gaseous solutions., Chemical reviews 44~(1) (1949) 233--244.

\bibitem{Balsara2012}
D.~S. Balsara, Self-adjusting, positivity preserving high order schemes for hydrodynamics and magnetohydrodynamics, Journal of Computational Physics 231~(22) (2012) 7504--7517.

\bibitem{wu2018}
K.~Wu, C.-W. Shu, Provably positive high-order schemes for ideal magnetohydrodynamics: analysis on general meshes, arXiv preprint arXiv:1807.11467 (2018).

\bibitem{ChristliebEtAl2015}
A.~J. Christlieb, Y.~Liu, Q.~Tang, Z.~Xu, Positivity-preserving finite difference weighted eno schemes with constrained transport for ideal magnetohydrodynamic equations, SIAM Journal on Scientific Computing 37~(4) (2015) A1825--A1845.

\bibitem{DingWu2024}
S.~Ding, K.~Wu, A new discretely divergence-free positivity-preserving high-order finite volume method for ideal mhd equations, SIAM Journal on Scientific Computing 46~(1) (2024) A50--A79.

\bibitem{WLSENO_PP2021}
M.~Liu, M.~Zhang, C.~Li, F.~Shen, A new locally divergence-free wls-eno scheme based on the positivity-preserving finite volume method for ideal mhd equations, Journal of Computational Physics 447 (2021) 110694.

\bibitem{Casulli2009}
V.~Casulli, A high-resolution wetting and drying algorithm for free-surface hydrodynamics, International Journal for Numerical Methods in Fluids 60~(4) (2009) 391--408.

\bibitem{BrugnanoCasulli2009}
L.~Brugnano, V.~Casulli, Iterative solution of piecewise linear systems and applications to flows in porous media, SIAM Journal on Scientific Computing 31~(3) (2009) 1858--1873.

\bibitem{Toth2000}
G.~T{\'o}th, The {$\nabla \cdot \mathbf{B} = 0$} constraint in shock-capturing magnetohydrodynamics codes, Journal of Computational Physics 161~(2) (2000) 605--652.

\bibitem{EvansHawley1988}
C.~R. Evans, J.~F. Hawley, Simulation of magnetohydrodynamic flows-a constrained transport method, Astrophysical Journal, Part 1 (ISSN 0004-637X), vol. 332, Sept. 15, 1988, p. 659-677. 332 (1988) 659--677.

\bibitem{GardinerStone2005}
T.~A. Gardiner, J.~M. Stone, An unsplit godunov method for ideal mhd via constrained transport, Journal of Computational Physics 205~(2) (2005) 509--539.

\bibitem{BalsaraSpicer1999}
D.~S. Balsara, D.~S. Spicer, A staggered mesh algorithm using high order godunov fluxes to ensure solenoidal magnetic fields in magnetohydrodynamic simulations, Journal of Computational Physics 149~(2) (1999) 270--292.

\bibitem{Rossmanith2006}
J.~A. Rossmanith, An unstaggered, high-resolution constrained transport method for magnetohydrodynamic flows, SIAM Journal on Scientific Computing 28~(5) (2006) 1766--1797.

\bibitem{Dedner2002}
A.~Dedner, F.~Kemm, D.~Kr{\"o}ner, C.-D. Munz, T.~Schnitzer, M.~Wesenberg, Hyperbolic divergence cleaning for the mhd equations, Journal of Computational Physics 175~(2) (2002) 645--673.

\bibitem{Powell1999}
K.~G. Powell, P.~L. Roe, T.~J. Linde, T.~I. Gombosi, D.~L. De~Zeeuw, A solution-adaptive upwind scheme for ideal magnetohydrodynamics, Journal of Computational Physics 154~(2) (1999) 284--309.

\bibitem{BrackbillBarnes1980}
J.~U. Brackbill, D.~C. Barnes, The effect of nonzero {$\nabla \cdot \mathbf{B}$} on the numerical solution of the magnetohydrodynamic equations, Journal of Computational Physics 35~(3) (1980) 426--430.

\bibitem{ChengJuSchochet2021}
B.~Cheng, Q.~Ju, S.~Schochet, Convergence rate estimates for the low mach and alfv{\'e}n number three-scale singular limit of compressible ideal magnetohydrodynamics, ESAIM: Mathematical Modelling and Numerical Analysis 55 (2021) S733--S759.

\bibitem{QiSun1993}
L.~Qi, J.~Sun, A nonsmooth version of newton's method, Mathematical programming 58~(1) (1993) 353--367.

\bibitem{leVeque2002}
R.~J. LeVeque, Finite volume methods for hyperbolic problems, Vol.~31, Cambridge university press, 2002.

\bibitem{BOSCARINO2019594}
S.~Boscarino, J.-M. Qiu, G.~Russo, T.~Xiong, \href{https://www.sciencedirect.com/science/article/pii/S0021999119303109}{A high order semi-implicit imex weno scheme for the all-mach isentropic euler system}, Journal of Computational Physics 392 (2019) 594--618.
\newblock \href {https://doi.org/https://doi.org/10.1016/j.jcp.2019.04.057} {\path{doi:https://doi.org/10.1016/j.jcp.2019.04.057}}.
\newline\urlprefix\url{https://www.sciencedirect.com/science/article/pii/S0021999119303109}

\bibitem{GMRES}
Y.~Saad, M.~H. Schultz, \href{https://doi.org/10.1137/0907058}{Gmres: A generalized minimal residual algorithm for solving nonsymmetric linear systems}, SIAM Journal on Scientific and Statistical Computing 7~(3) (1986) 856--869.
\newblock \href {http://arxiv.org/abs/https://doi.org/10.1137/0907058} {\path{arXiv:https://doi.org/10.1137/0907058}}, \href {https://doi.org/10.1137/0907058} {\path{doi:10.1137/0907058}}.
\newline\urlprefix\url{https://doi.org/10.1137/0907058}

\bibitem{accuracyvortex}
D.~S. Balsara, Second-order-accurate schemes for magnetohydrodynamics with divergence-free reconstruction, The Astrophysical Journal Supplement Series 151~(1) (2004) 149--184.

\bibitem{exact1}
S.~Falle, Rarefaction shocks, shock errors and low order of accuracy in zeus, arXiv preprint astro-ph/0207419 (2002).

\bibitem{exact2}
S.~A. E.~G. FALLE, S.~S. KOMISSAROV, On the inadmissibility of non-evolutionary shocks, Journal of Plasma Physics 65~(1) (2001) 29–58.
\newblock \href {https://doi.org/10.1017/S0022377801008856} {\path{doi:10.1017/S0022377801008856}}.

\bibitem{exact3}
M.~Torrilhon, \href{https://www.sciencedirect.com/science/article/pii/S0021999103003474}{Non-uniform convergence of finite volume schemes for riemann problems of ideal magnetohydrodynamics}, Journal of Computational Physics 192~(1) (2003) 73--94.
\newblock \href {https://doi.org/https://doi.org/10.1016/S0021-9991(03)00347-4} {\path{doi:https://doi.org/10.1016/S0021-9991(03)00347-4}}.
\newline\urlprefix\url{https://www.sciencedirect.com/science/article/pii/S0021999103003474}

\bibitem{DumbserHLLEM}
M.~Dumbser, D.~S. Balsara, \href{https://www.sciencedirect.com/science/article/pii/S0021999115006786}{A new efficient formulation of the hllem riemann solver for general conservative and non-conservative hyperbolic systems}, Journal of Computational Physics 304 (2016) 275--319.
\newblock \href {https://doi.org/https://doi.org/10.1016/j.jcp.2015.10.014} {\path{doi:https://doi.org/10.1016/j.jcp.2015.10.014}}.
\newline\urlprefix\url{https://www.sciencedirect.com/science/article/pii/S0021999115006786}

\bibitem{DumbserToro}
M.~Dumbser, E.~F. Toro, On universal osher-type schemes for general nonlinear hyperbolic conservation laws, Communications in Computational Physics 10~(3) (2011) 635–671.
\newblock \href {https://doi.org/10.4208/cicp.170610.021210a} {\path{doi:10.4208/cicp.170610.021210a}}.

\bibitem{BoscheriADER}
W.~Boscheri, M.~Dumbser, D.~S. Balsara, \href{https://onlinelibrary.wiley.com/doi/abs/10.1002/fld.3947}{High-order ader-weno ale schemes on unstructured triangular meshes—application of several node solvers to hydrodynamics and magnetohydrodynamics}, International Journal for Numerical Methods in Fluids 76~(10) (2014) 737--778.
\newblock \href {http://arxiv.org/abs/https://onlinelibrary.wiley.com/doi/pdf/10.1002/fld.3947} {\path{arXiv:https://onlinelibrary.wiley.com/doi/pdf/10.1002/fld.3947}}.
\newline\urlprefix\url{https://onlinelibrary.wiley.com/doi/abs/10.1002/fld.3947}

\bibitem{orszag1979}
S.~A. Orszag, C.-M. Tang, Small-scale structure of two-dimensional magnetohydrodynamic turbulence, Journal of Fluid Mechanics 90~(1) (1979) 129--143.

\bibitem{dahlburg1989}
R.~Dahlburg, J.~Picone, Evolution of the orszag--tang vortex system in a compressible medium. i. initial average subsonic flow, Physics of Fluids B: Plasma Physics 1~(11) (1989) 2153--2171.

\bibitem{picone1991}
J.~M. Picone, R.~B. Dahlburg, Evolution of the orszag--tang vortex system in a compressible medium. ii. supersonic flow, Physics of Fluids B: Plasma Physics 3~(1) (1991) 29--44.

\bibitem{dumbser2008}
M.~Dumbser, D.~S. Balsara, E.~F. Toro, C.-D. Munz, A unified framework for the construction of one-step finite volume and discontinuous galerkin schemes on unstructured meshes, Journal of Computational Physics 227~(18) (2008) 8209--8253.

\bibitem{balsara2015}
D.~S. Balsara, M.~Dumbser, Divergence-free mhd on unstructured meshes using high order finite volume schemes based on multidimensional riemann solvers, Journal of Computational Physics 299 (2015) 687--715.

\bibitem{activeflux}
M.~Liu, D.~Pang, R.~Abgrall, K.~Wu, \href{https://arxiv.org/abs/2510.19721}{An active-flux-type scheme for ideal mhd with provable positivity and discrete divergence-free property} (2025).
\newblock \href {http://arxiv.org/abs/2510.19721} {\path{arXiv:2510.19721}}.
\newline\urlprefix\url{https://arxiv.org/abs/2510.19721}

\bibitem{SC3}
K.~Wu, C.-W. Shu, A provably positive discontinuous galerkin method for multidimensional ideal magnetohydrodynamics, SIAM Journal on Scientific Computing 40~(5) (2018) B1302--B1329.

\bibitem{Sc1}
W.~Dai, P.~R. Woodward, A simple finite difference scheme for multidimensional magnetohydrodynamical equations, Journal of Computational Physics 142~(2) (1998) 331--369.

\bibitem{SC2}
G.-S. Jiang, E.~Tadmor, Nonoscillatory central schemes for multidimensional hyperbolic conservation laws, SIAM Journal on Scientific Computing 19~(6) (1998) 1892--1917.

\bibitem{BoscheriWB}
C.~Birke, W.~Boscheri, C.~Klingenberg, A well-balanced semi-implicit imex finite volume scheme for ideal magnetohydrodynamics at all mach numbers, Journal of Scientific Computing 98~(2) (2024) 34.

\bibitem{FAMBRIWB}
F.~Fambri, E.~Zampa, S.~Busto, L.~Río-Martín, F.~Hindenlang, E.~Sonnendrücker, M.~Dumbser, A well-balanced and exactly divergence-free staggered semi-implicit hybrid finite volume / finite element scheme for the incompressible mhd equations, Journal of Computational Physics 493 (2023) 112493.
\newblock \href {https://doi.org/https://doi.org/10.1016/j.jcp.2023.112493} {\path{doi:https://doi.org/10.1016/j.jcp.2023.112493}}.

\end{thebibliography}



\end{document}